\documentclass[11pt]{amsart}
\usepackage{amsmath}
\usepackage{amssymb}
\usepackage{amsfonts}
\usepackage{epsfig,graphics,epic}
\usepackage[all]{xy}
\newcommand{\CP}{\mathbb{P}}

\newcommand{\Sym}{{\rm Sym}}  
\newcommand{\Braid}{{\rm Braid}}  %
\newcommand{\p}{\pi_1}

\newcommand{\PP}{\mathop{\mathbb{P}}}  
\newcommand{\Hom}{\mathop{\rm{Hom}}}
\newcommand{\tensor}{\mathop{\otimes}}

\newcommand{\Aut}{ \mathop{\rm Aut} }
\newcommand{\ord}{ \mathop{\rm ord} }
\newcommand{\coker}{ \mathop{\rm coker} }

\newcommand{\Spec}{ \mathop{\rm Spec} }
\newcommand{\supp}{\mathop{\rm supp}\nolimits}
\newcommand{\Pic}{\mathop{\rm Pic}\nolimits}     
\newcommand{\Sing}{\mathop{\rm Sing}\nolimits}
\newcommand{\Pol}{\mathop{\rm Pol}\nolimits}     
\newcommand{\res}{\mathop{\rm res}\nolimits}     
\newcommand{\can}{ \mathop{\rm can} }
\newcommand{\Weil}{{\rm Weil}}  
\newcommand{\Cartier}{{\rm Cartier}}  

\newcommand{\arrow}[1]{\stackrel{#1}{\to}} 

\newcommand{\QQ}{\mathbb{Q}}
\newcommand{\N}{\mathbb{N}}
\newcommand{\C}{\mathbb{C}}
\newcommand{\ZZ}{\mathbb{Z}}

\newcommand{\Z}{\mathbb{Z}}
\newcommand{\G}{\Gamma}
\newcommand{\g}{\gamma}
\newcommand{\n}{\nu}

\newcommand{\xyIsom}{\ar^-*@{~}}          
\newcommand{\xySimeq}{\ar@{-}^*@{~}}      
\newcommand{\xyRavno}{\ar@{=}}
\newcommand{\displayJ}[8]
{
  \xymatrix
  {
      {} & {} & 0 & 0 & {}  \\
      {} & {} & {#1} \ar[u]\xyIsom[r] & {#2} \ar[u] & {} \\
      0 \ar[r] & {#3} \ar[r]\xyRavno[d] & {#4} \ar[r]\ar[u] & {#5} \ar[r]\ar[u] & 0 \\
      0 \ar[r] & {#6} \ar[r] & {#7} \ar[r]\ar[u] & {#8} \ar[r]\ar[u] & 0 \\
      {} & {} & 0 \ar[u] & 0 \ar[u] & {}
   }
}

\newcommand{\vdim}{ \mathop{\rm vdim} }     
\newcommand{\conv}{{\operatorname{conv}}}
\newcommand{\Del}{{\Delta}}
\newcommand{\Tor}{{\operatorname{Tor}}}

\newcommand{\OO}{\mathcal{O}}    
\newcommand{\Jk}{J_{\xi}}        
\newcommand{\cp}{\mathbb{P}}  
\newcommand{\pp}{\mathbb{P}}  

\newcommand{\intersect}{\cap}                    
\newcommand{\isom}{\simeq}                      
\newcommand{\Isom}{\stackrel{\sim}{\to}}     
\newcommand{\dual}{^{\vee}}
\newtheorem{thm}{Theorem}[section]
\newtheorem*{thm*}{Theorem}
\newtheorem{corollary}[thm]{Corollary}
\newtheorem{lemma}[thm]{Lemma}
\newtheorem{conjecture}[thm]{Conjecture}
\newtheorem*{lemma*}{Lemma}
\newtheorem{definition}[thm]{Definition}
\newtheorem{example}[thm]{Example}
\newtheorem{prs}[thm]{Proposition}
\newtheorem{remark}[thm]{Remark}
\newtheorem{notation}[thm]{Notation}

\newtheorem*{notation*}{Notation}

\begin{document}

\title [On ramified covers of the projective plane]{On ramified covers
  of the projective plane  I:\\ Segre's theory and classification in
  small degrees\\ \small{(with an Appendix by Eugenii  Shustin)}}

\author[M. Friedman, M. Leyenson]{Michael Friedman and Maxim Leyenson$^1$}\address{
Michael Friedman, Department of Mathematics,
 Bar-Ilan University, 52900 Ramat Gan, Israel}
 \email{fridmam@macs.biu.ac.il}
\address{
Maxim Leyenson}
 \email{leyenson@gmail.com}
  \address{
Eugenii  Shustin, School of Mathematical Sciences,
Raymond and Beverly Sackler Faculty of Exact Sciences,
Tel Aviv University,
Ramat Aviv, 69978 Tel Aviv, Israel}
 \email{shustin@post.tau.ac.il}
\stepcounter{footnote} \footnotetext{This work is
partially supported by the Emmy Noether Research Institute for Mathematics (center of the Minerva
Foundation of Germany), the Excellency Center ``Group Theoretic Methods in the Study of Algebraic Varieties"
of the Israel Science Foundation.}

\maketitle
\begin{abstract}

We study ramified covers of the projective plane $\pp^2$. Given a smooth
surface $S$ in $\pp^n$ and a generic enough projection $\pp^n \to \pp^2$, we
get a cover $\pi: S \to \pp^2$, which is ramified over a plane curve
$B$. The curve $B$ is usually singular, but is classically known to have only cusps and
nodes as singularities for a generic projection.

    Several questions arise: First, What is the \emph{geography} of branch curves among all cuspidal-nodal curves?
%
And second, what is the \emph{geometry} of branch curves; i.e., how can one distinguish a
branch curve from a non-branch curve with the same numerical invariants? For
example, a plane sextic with six cusps is known to be a branch curve of a
generic projection iff its six cusps lie on a conic curve, i.e., form a
special 0-cycle on the plane.

We start with reviewing what is known about the answers to these questions,
both simple and some non-trivial results. Secondly, the classical work of
Beniamino Segre gives a complete answer to the second question in the case
when $S$ is a smooth surface in $\pp^3$. We give an interpretation of the work
of Segre in terms of relation between Picard and Chow groups of 0-cycles
on a singular plane curve $B$. We also review examples of small degree.

In addition, the Appendix written by E. Shustin shows the existence of new Zariski pairs.

\end{abstract}

\tableofcontents

\section{Introduction} \label{secIntro}

Let $S$ be a non-singular algebraic surface of degree $\n$ in the complex
projective space $\pp^r$. One can obtain information on $S$ by projecting it
from a generically chosen linear subspace of codimension 3 to the projective
plane $\pp^2$. The ramified covers of the projective plane one gets in this way
were studied extensively by the Italian school (in particular, by Enriques, who called a
surface with a given morphism to $\pp^2$ a ``multiple plane'', and later by
Segre, Zariski and others.) The main questions of their study were which curves
can be obtained as branch curves of the projection, and to which extent the
branch curve determines the pair $(S, \pi: S \to \pp^2)$. In the course of this study, Zariski
studied the fundamental groups of complements of plane curves and in particular
introduced what later became known as Enriques-Zariski-Van Kampen theorem (see
Zariski's foundational paper \cite{Za1}). It was also discovered by the Italian
school that a branch curve of generic projections of surfaces in
characteristics 0 has only nodes and cusps as singularities, though we were not able to trace a reference
to a proof from that era (but see
\cite{CiroFla} for a modern proof). Segre-Zariski criterion (see \cite{Za1} and
also Zariski's book \cite{Za2}) for a degree 6 plane curve with 6 cusps to be a
branch curve is well known and largely used, but Segre's generalization
(\cite{Se}), where he gave a necessary and sufficient condition for a plane
curve to be a branch curve of a ramified cover of a smooth surface
 in $\pp^3$ in terms of adjoint linear systems to the branch curve,
was largely forgotten (see Theorem \ref{thmSegre}).

Two recent surveys, by D'Almeida (\cite{DA}) and  Val. S. Kulikov (\cite{Kul}), were written on
Segre's generalization.
However, our approach is different, for we give an interpretation of Segre's
work in terms of studying various equivalence relations of 0-cycles on
nodal-cuspidal curves. We also emphasize the logic of passing from adjoint curves for
a plane singular curve to rational objects which become regular on the normalization
of this curve, (what would be called "weakly holomorphic functions" in analytic geometry \cite[Chapter VI]{Nara} ),
and control geometry of its space models.

The geometry of ramified covers in dimension two is very
different from the geometry in dimension one. In dimension one, for
any set of points $B$ in the projective line $\pp^1$ we always have many
possible (non isomorphic) ramified covers of $\pp^1$ for which $B$ is the branch locus.
In terms of monodromy data, the fundamental group $G = \pi_1(\pp^1 - B)$ is free,
and thus admits many epimorphic representations $G \to \Sym_\nu$ for multiple
values of $\nu$; and, moreover, every such a representation is actually
coming from a ramified cover due to the Riemann-Grauert-Remmert theorem (see Theorem \ref{thmGR}). However,
in dimension two Chisini made a surprising conjecture (circa 1944, cf. \cite{Ci}) that the
pair $(S, \pi: S \to \pp^2)$, where $\pi$ is generic, can be uniquely determined
by the branch curve $B$, if $\deg \pi \ge 5$ (and in the case of generic linear projections, if this curve is of sufficiently high degree). This conjecture was proved only recently
by Kulikov (\cite{Ku},\cite{Ku2}). In terms of monodromy data, by Grauert-Remmert theorem,
even though it is true that every
representation $\rho: \pi_1(\pp^2 - B) \to \Sym_\nu$ comes from a
ramified cover $S \to \cp^2$ of degree $\n$ with a normal surface $S$, one has to ensure certain ``local'' conditions on the representation $\rho$ which ensure that $S$ is non-singular, which
sharply reduce the number of admissible representations into the
symmetric group. In fact, the Chisini's conjecture implies that
once the degree of the ramified cover sufficiently high, there is
only one such representation.

%
%

The structure of the paper is as follows. Sections \ref{subsecGen} and \ref{subsecZarVar} contain preliminary material. In section
\ref{subsecGen} we recall some facts about ramified coverings
and Grauert-Remmert theorems, and in the following section \ref{subsecZarVar} we look at $V(d,c,n)$ (resp. $B(d,c,n)$) the variety of degree $d$ plane curves (resp. branch curves) with $c$ cusps and $n$ nodes (in addition, we prove the following interesting fact: in coordinates $(d,c,\chi)$, with geometric Euler characteristic $\chi$, the duality map
$(d,c,\chi) \to (d^*,c^*,\chi^*)$ becomes a linear reflection). We also
 recall a number of necessary numerical conditions for a curve to be a branch curve.
 In the main section, Section \ref{secSurInP3} we re-establish the results of Segre for smooth
surfaces in $\pp^3$  and discuss the geography of surfaces with
ordinary singularities and their branch curves.  Looking at the variety of  nodal cuspidal plane curves,  we also compute the dimension of the component which
    consists of branch curves of smooth surfaces in $\pp^3$ (see Subsection \ref{subsecDimB3}).
In Section \ref{subSecExample} we classify  admissible (i.e. nodal-cuspidal irreducible) branch
curves of small degree. Appendix A is written by Eugenii Shustin, where new Zariski pairs are found. Each Zariski pair consists of a branch curve of a smooth surface in $\pp^3$ and a nodal--cuspidal curve which is not a branch curve. In the other Appendices we recall some facts on the Picard and Chow groups of
nodal cuspidal curves we use, and on the bisecants to complete intersection curves in $\pp^3$.

 In the subsequent papers (see \cite{FLL})  we will deal with an analogue  of the Segre
theory for surfaces with ordinary singularities in $\pp^3$,  and also
give a combinatorial reformulation of the Chisini conjecture.

\textbf{Acknowledgments}:
Both authors are deeply thankful to Prof. Mina Teicher and the Emmy Noether Research
Institute at the Bar Ilan University (ENRI) for support during their work, and
also thankful to Prof. Teicher  for attention to the work, various important
suggestions and stimulating discussions.

The authors wish also to thank deeply E. Shustin for writing the Appendix. This is a major and important
contribution to this work.

The authors are also very grateful to Valentine S. Kulikov and to Viktor S.
Kulikov, for paying our attention to the paper of Valentine S. Kulikov on branch curves
in Russian \cite{Kul}.  We are grateful to Fabrizio Catanese for scanning and sending us a
rare paper by Segre (\cite{Se}). We also thank Tatiana Bandman, Ciro Ciliberto, Dmitry Kerner,
Ragni Piene, Francesco Polizzi and Rebecca Lehman  for fruitful discussions and advices.

The second author also wants to thank the department of Mathematics at the Bar Ilan University
for an excellent and warm scientific and working atmosphere.

%

\section{Ramified covers}
\label{subsecGen}
In this section we start with a  general discussion on branched coverings, continuing afterwards
to  investigation of surfaces and generic projections.

\newcommand{\CC}{\mathbb C}

\subsection {\'{E}tale covers}

\newcommand{\San}{S_{an}} Let $S$ be a scheme (of finite type) over
$\CC$, and $\San$ be the corresponding analytic space. Let $Et^f_S$ be
the category of finite \'{e}tale schemes over $S$, and $Et^f_{\San}$ be
the category of finite \'{e}tale complex-anaytic spaces over $\San$. One
can verify (cf. \cite{GR} and \cite[XI.4.3]{Grothendieck.descent-1}) that the
``analytization'' functor
$$
   a_S: Et^f_S \to Et^f_{\San}
$$
is faithfully flat. The following Grauert-Remmert theorem  generalizes the so-called
Riemann existence theorem in case $\dim S = 1$:

\begin{thm}[Grauert-Remmert]
If $S$ is normal, then $a_S$ is equivalence of categories.
\end{thm}


\subsection{Ramified covers of complex analytic spaces}
         Let $X$ be a complex analytic space, $Y \subset X$ be a closed analytic subspace in
$X$, and $U = X - Y$ be the complement. Assume that $U$ is dense in $X$.

 \begin{thm}[Grauert-Remmert] \label{thmGR} If $X$ is normal, then the
restriction functor
$$
     \res_U: (\text{normal analytic covers of } X \text{ \'{e}tale over } U)
             \longrightarrow
     (\text {\'{e}tale analytic covers of U} )
$$
is an equivalence of categories.
\end{thm}

For other formulations of Theorem \ref{thmGR} and the proof, see
\cite[Proposition 12.5.3, Theorem 12.5.4.]{SGA1}.\\

We say that $f: X' \to X$ is a ramified cover branched over $Y$ if
$f|_U$ is \'{e}tale and the ramification locus of $f$ (i.e.
supp$(\Omega^1_{X'/X})$) is contained in $Y$. Note that even if
$X$ is smooth, we still have to allow ramified covers $X' \to X$
with normal singularities in order to get an essentially
surjective restriction functor, as seen in the following example.
Let $X = {\mathbb A}^2$, $Y = (xy = 0)$, $U = X - Y$, and $f: U'
\to U$ be a degree 2 unramified cover given by the monodromy
representation $\pi_1^{an} (U) \isom \ZZ \oplus \ZZ \to \ZZ/2$
which sends both generators to the generator of $\ZZ/2$. It is
easily seen in this example that $f$ can not be extended to a
ramified cover $X' \to X$ with smooth $X'$, but if we allow normal
singularities one gets canonical extension given in coordinated by
$z^2 = x y$, a cone with $A_1$ singularity.

\subsection{Ramified covers of $\pp^2$} \label{subsecGenProj}

From now on, we restrict ourselves to char $ = 0$. Let $S$ be a
smooth surface in $\pp^r = \PP(V)$. Let $W \subset V$ be a
codimension 3 linear subspace such that $\PP(W) \intersect S =
\emptyset$ and let $p$ the resulting projection map $p: \PP(V) \to
\PP(V/W)$ and $\pi:S\to \cp ^2$ its restriction to $S$. It is
clear that $\pi$ is a finite morphism of degree equals to $\deg
S$.

%
%

for a {\it generic} choice of $W$, $\pi$ is called a \emph{
generic projection map} and the following is classical (see, for
example, \cite{SR}, \cite{Piene} and \cite{CiroFla}):
\begin{enumerate}
\item[(i)] $\pi$ is ramified along an irreducible curve $B\subset \pp^2$
  which has only nodes and cusps as singularities;
\item[(ii)] The ramification divisor $B^* \subset S$ is irreducible and
  smooth, and the restriction $\pi: B^* \to B$ is a resolution of
  singularities;
\item[(iii)] $\pi^{-1}(B)=2B^*+Res$ for some residual curve $Res$ which is
  reduced.
\end{enumerate}

\begin{remark} \label{remRamCov} \emph{
Note that not every ramified cover $S$ of $\pp^2$ with a branch curve $B \subset \pp^2$ can be given as a restriction
of generic linear projection $\pp^r \to \pp^2$ (to a smooth surface $S$). See, for example, Remark \ref{remQuarticDoubleLine}.}

\end{remark}

\begin{remark} \emph{
Note that generically cusps to not occur in a generic projection
of a smooth space curve, but do occur for the projection of a
ramification curve of surfaces, already in the basic example of
smooth surfaces in $\pp^3$ and its projection to $\pp^2$.
Consider, for example, the case of a smooth surface $S$ in $\pp^3$
and its generic projection to $\pp^2$. Since the branch curve $B$
is the projection of the ramification curve $B^*$ which is a space
curve, it generically has double points corresponding to bisecants
of $B^*$ containing the projection center $O$. The cuspidal points
are somewhat more unusual for projections of smooth space curves,
since they do not occur in the projections of generic smooth
space curves.  However, the projections of generic ramification
curves have cusps. To give a typical example, consider a family of
plane (affine) cubic curves $z^3 - 3 az + x = 0$ in the $(x,z)$ -
plane, where $a$ is a parameter. The real picture is the
following: for $a > 0$ the corresponding cubic parabola has 2
extremum point, for $a = 0$ one inflection point and for $a < 0$
no real extremums; the universal family in the $(x,z,a)$ space is
the so-called real Whitney singularity , and projection to the
``horizontal'' $(x,a)$ plane gives a semi-cubic parabola $a^2 -
x^3$ with a cusp.
%
%
In other words, substituting $y = -3a$, we see that the affine cubic
surface $S$ can be considered as the ``universal cubic polynomial'' in
$z$, $ p(z) = z^3 + y \cdot z + x = 0,$ and its discriminant $\Delta =
27 y^2 + 4 x^3$ has an $A_2$ singularity, which is a cusp.  (Recall
that in general a discriminant of a polynomial of degree $n$ with
$a_{n-1} = 0$ has singularity of type $A_{n-1}$).}
\end{remark}


\section{Moduli of branch curves and their
geography}\label{subsecZarVar} 
The geography of surfaces was introduced and studied by Bogomolov-Miyaoka-Yau, Persson, Bombieri, Catanese and more.
Parallel to the terminology of geography of surfaces, we will use the term geography of branch curves for the distribution of branch curves in the variety of nodal-cuspidal curves. Subsection \ref{subsecNodeCupsCur}  recalls few facts on nodal-cuspidal degree $d$ curves with $c$ cusps and $n$ nodes and introduces a more natural coordinate to work with: $\chi$ -- the Euler characteristic. The main subsection is Subsection \ref{subsecBranch}, which compares the geography of branch curves in the $(d,c,\chi)$ coordinates to the geography of surfaces in $(c_1^2,c_2)$ coordinates. Subsection \ref{secB_dcn} constructs the variety of branch curves.

\subsection{Severi-Enriques varieties of nodal-cuspidal curves} \label{subsecNodeCupsCur}
%
\begin{notation}
   For a triple $(d, c, n)\in \N^3$
  let $V(d,c,n)$ be the variety of plane curves of degree $d$ with $c$ cusps
  and $n$ nodes as their only singularities.
\end{notation}
It is easy to prove that $V(d,c,n)$ is a disjoint union of locally closed subschemes of $\pp^N$, where
 $N = \frac{1}{2}d(d+3)$.

A curve $C \in V(d,c,n)$, has arithmetic genus $p_a$ and geometric genus $g = p_g$ when

\begin{equation} p_a  = \frac{1}{2} (d-1)(d-2),\end{equation}
\begin{equation}\label{eqnGenus} g = p_a - c - n = \frac{1}{2}(d-1)(d-2) - c - n, \end{equation}
and we let $\chi$ to be  the topological Euler characteristics of the normalization of $C$
\begin{equation} \chi \doteq 2 - 2g.\end{equation}

 We shall use the coordinates $(d,c,\chi)$ instead of  $(d, c, n)$ since many formulas, such as Pl\"{u}cker formulas,
 become {\it linear} in these coordinates.
 Note that one can present $n$ in terms of $(d,c,\chi)$ as follows:
  $$n = \frac{1}{2}(d-1)(d-2) - c + \frac{1}{2} \chi - 1 =
\frac{1}{2} d(d-3) - c + \frac{1}{2} \chi.$$



Let $C\in V(d,c,n)$ be a Pl\"{u}cker curve, i.e.,  a
curve that its dual $C\dual$ is also a curve in some $V(d^*, c^*, n^*)$ (Note that this is an open condition in
$V(d,c,n)$ and that $(C\dual)\dual = C$.)  Then the following Pl\"{u}cker formulas hold:

 \begin{equation} \label{eqnDegree}   d^* = d (d - 1)     - 3c   - 2n,\end{equation}
  \begin{equation}   g = g^* \end{equation}

where $g^*$ is the geometric genus of $C \dual$.

The formula from $c^*$ can be induced from Equations (\ref{eqnGenus}), (\ref{eqnDegree}) for $C\dual$,
i.e.
$$
c^* = 3d^2 - 6d - 8c - 6n.
$$

\subsubsection[Linearity of the Pl\"{u}cker formulas]{Linearity of the Pl\"{u}cker formulas}
The Pl\"{u}cker formulas become linear in the $(d,c,\chi)$
coordinates (and also the formulas for the Chern classes of a
surface whose branch curve $B \in V(d,c,\chi)$. See Lemma
\ref{lem_c_1^2} and \ref{lem_c_2}.), which is the primarily reason
we want to consider them. Namely,


%
\begin{gather}
    d^* = 2d -  c -   \chi, \\
    c^* = 3d - 2c - 3 \chi, \\
    \chi^* = \chi,
\end{gather}
in other words, in these coordinates projective duality is given by a linear transformation
$$
D =
\begin{pmatrix}
   2 & -1 & -1 \\
   3 & -2 & -3 \\
   0 &  0 &  1
\end{pmatrix}
$$
which is diagonalizable with eigenvalues $(-1,1,1)$ where the
eigenvector $d - c - \chi = d^* - d$ corresponds to the eigenvalue (-1), i.e.,
gives a reflection in the lattice $\ZZ \oplus \ZZ \oplus 2\ZZ$. We hope to explain this
phenomenon elsewhere.\\\\


The fact that the invariants $d,c,n$ and $g$ of the curve are not negative implies, in the
$(d,c,\chi)$ coordinates, the following inequalities:
\begin{gather}
     (n \geq 0) \Rightarrow \,  2c - \chi \leq d(d-3),  \\
    (g \geq 0) \Rightarrow \,  \chi \leq 2,\\
    (d^* \geq 0) \Rightarrow \,  c + \chi    \leq 2d,  \\
    (c^* \geq 0) \Rightarrow \,  2c + 3 \chi \leq 3d.
\end{gather}

Zariski also proved (\cite[Section 3]{Za}) the following inequality
\begin{equation} \label{eqZar}
         c < \frac{1}{2}(d - \beta)(d - \beta -3) + 2,
\end{equation}
where $\beta = [(d-1)/6]$. His proof uses the computation of the
virtual dimension of complete linear system of curves of order $d
- \beta -3$ passing through the $c$ cusps of $C$. (see also
\cite[Chapter VIII]{Za2}). But his inequality is stronger then the
ones given by Plucker formulas only for small $d$`s; we use it
once for $d = 8$ when classifying branch curves of small degree (see Section \ref{subSecExample}).

\begin{remark}
{\rm For a nodal--cuspidal curve $C \in V(d,c,n)$ we have the following inequality
$$
2c + n \leq (d-1)^2
$$
or, in $(d,c,\chi)$ coordinates:
$$
2c + \chi \leq d^2 -d + 2,
$$
which is induced from intersecting two generic polars of $C$ and B\'{e}zout theorem.}
\end{remark}

\begin{center}

  \epsfig{file=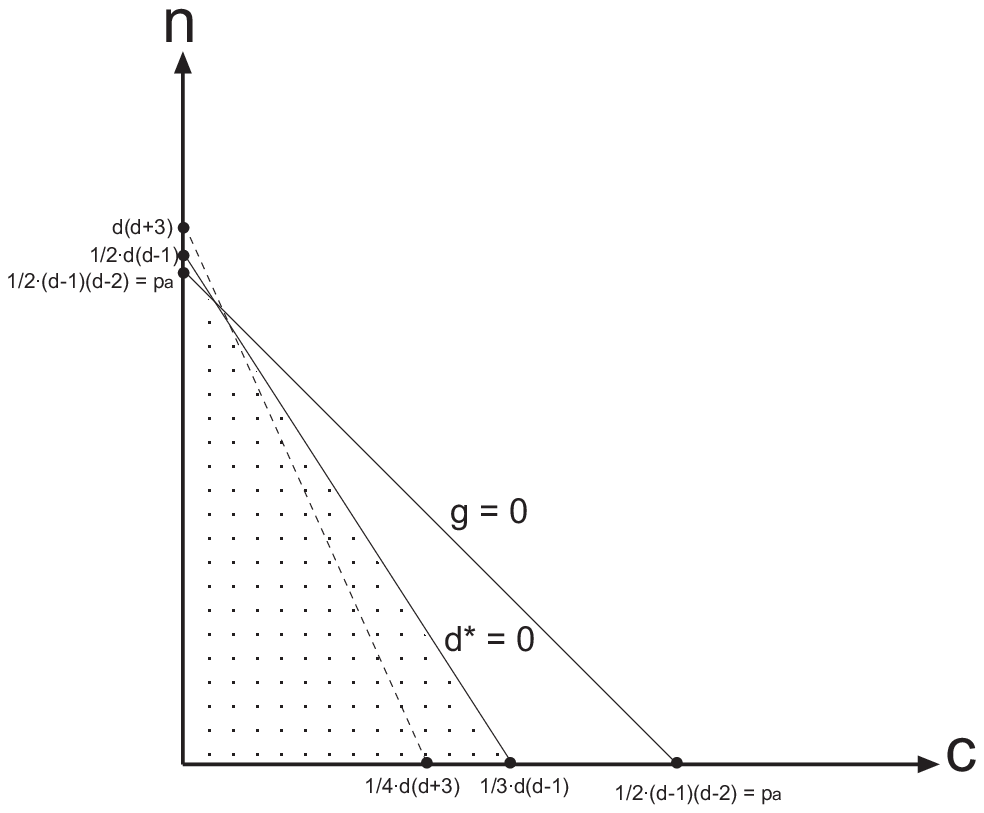, height=7cm, width=7cm} \\

  \small{Figure 1 : Geography of admissible plane curves in the $(c,n)$--plane  for large $d$.\\
    The dashed line is where the expected dimension of \{family of degree $d$ curves with $n$ nodes and $c$ cusps\} = $\frac{1}{2}d(d+3) - n - 2c$ = 0.}

\end{center}

\begin{center}

  \epsfig{file=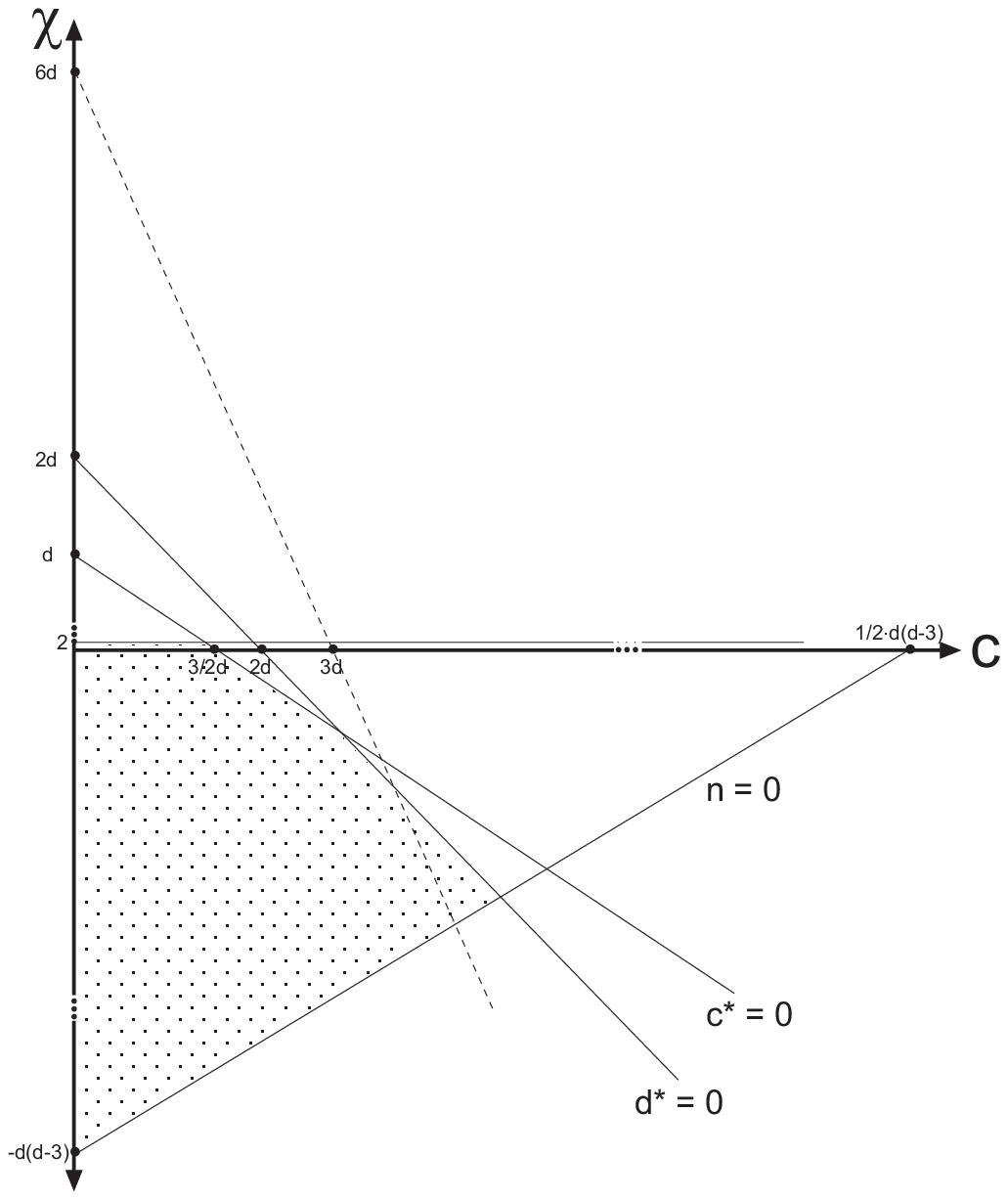, height=8.5cm, width=8.5cm} \\
  \small{Figure 2 : Geography of admissible plane curves in the $(c,\chi)$--plane  for large $d$.\\
   The dashed line is where the expected dimension of \{family of degree $d$ curves with $n$ nodes and $c$ cusps\} = $3d - \frac{1}{2}\chi - c$ = 0.}

\end{center}

For more obstructions on the existence of singular plane curves and a recent survey on equisingular families and, in particular, nodal-cuspidal curves see \cite{Shus}.

\subsection{Geography of branch curves}
\label{subsecBranch}
\begin{notation*}
  Let $B(d,c,n)$ be the subvariety in $V(d,c,n)$ consisting of branch curves of generic linear projections to $\pp^2$.
  We discuss it in Subsection \ref{secB_dcn}.
\end{notation*}

Let $B \in B(d,c,n)$ be the branch curve of a generic linear projection $\pi : S
\to \pp^2$ for a smooth irreducible projective surface $S \subset \pp^r$. Let $\nu = \deg \pi$, and $g = p_g(B)$ be the geometric genus of $B$. An important invariant of $B$ is the fundamental group of its complement $\pi_1(\pp^2 - B)$.

\begin{remark}
\label{remNodalCurve}
  \emph{Let $C \in V(d,c,n)$. If $c = 0$, i.e., $C$ is a nodal curve, then, by
  Zariski-Deligne-Fulton's theorem, the fundamental group $\pi_1(\pp^2 - C)$ of the
  complement of $C$ is abelian (\cite{Za2},\cite{D},\cite{F}). This
  theorem was proved by Zariski under the assumption that the Severi
  variety of nodal curves is irreducible (this was assumed to be
  established by Severi, but later was found to be mistaken). The
  correct proof of the irreducibility of the Severi variety $V(d,0,n)$
  was given by Harris \cite{H}, which completed Zariski's proof.
  Independent proofs were given later by Deligne and Fulton
  (\cite{D},\cite{F}) and others.}
\end{remark}

We begin with a consequence from Nori's result on fundamental groups of complements plane curves. Though the proof is known, we bring it as it is enlightening and brings together various aspects of the subject.
\begin{lemma}[Nori \cite{Nori}]
Let $B \in B(d,c,n)$. Then $6c + 2n \ge d^2$.
\label{lemNoAbelian}
\end{lemma}

\begin{proof}
Let $\psi : \p(\pp^2 - B) \rightarrow Sym_\nu$ be the monodromy representation, sending each generator to a
permutation, which describes the exchange of the sheets. Since $\pi : S \to \pp^2$ is a generic projection, the image
$H = \text{Im}(\psi)$ is generated by transpositions. As $S$ is irreducible, $H$ is  acts transitively on a set of $n$ points.
This implies that $H = Sym_\n$ and thus $\psi$ is an epimorphism. Thus for $\n > 2$, the fundamental group
$\p(\pp^2 - B)$ is not abelian and therefore $c > 0$ (by Remark \ref{remNodalCurve}).

Nori proved (\cite{Nori}) that for a cuspidal plane
curve $C$ with $d^2 > 6c + 2n$ and $c > 0$ the fundamental group of
the complement $\pi_1(\pp^2 - C)$ is abelian.
Thus, by the above discussion, $\n = 2$. However, there is no smooth double
cover of $\pp^2$ ramified over a singular $C$ (indeed, locally $S$ would be isomorphic to the singular cone $z^2 =
xy$ in a formal neighborhood of a node of $C$ and to the singular
surface $z^2 = x^2 - y^3$ in a formal neighborhood of a cusp). This
implies that Nori's condition cannot hold for a branch curve. Therefore

\begin{equation}
\label{eq2_2}
     6c + 2n -d^2 \geq 0,
\end{equation}
 or in  $(d,c,\chi)$ coordinates: $$ 4c + \chi -3d \geq 0.$$
\end{proof}
%

In the spirit of the above Lemma, we have the following result of
Shimada:

\begin{lemma} [Shimada \cite{Sh}]
Let $B \in B(d,c,n)$. Then $2n < d^2 - 5d + 8$.
\end{lemma}
\begin{proof}
Let $C \in V(d,c,n)$. By \cite{Sh}, if $2n \geq d^2 - 5d + 8$ then
$\p(\pp^2 - C)$ is abelian. However, for a branch curve $B \in
B(d,c,n)$, the corresponding fundamental group is not abelian, and
we have
\begin{equation}
\label{eqShimada}
     0 < \frac{1}{2}(d^2 - 5d + 8) - n,
\end{equation}
or, in $(d,c,\chi)$ coordinates: $$2c - \chi - 2d + 8 > 0.$$
\end{proof}

The following conditions on $c$ and $n$ are less obvious then the
previous Lemmas:

\begin{lemma}

\begin{equation} \label{eq2_1} %
          c = 0 \mod 3, \quad  n = 0 \mod 4
\end{equation}
or in  $(c,\chi)$ coordinates:
$$          c = 0 \mod 3,\quad \chi = 2c - d(d-3) \mod 8$$

\end{lemma}
\begin{proof} see \cite{MoTe2}. \end{proof}

%

%
\subsubsection{Geography of branch curves in $(d,c,\chi)$ versus geography
of surfaces in $(c_1^2,c_2)$}
Let $S$ be a smooth algebraic surface and $\pi: S \to \pp^2$ be a
generic ramified cover. Let $B$ be the branch curve of $\pi$, $B \in
B(v)$ for some vector $v \in L$,\, and let $\nu = \deg \pi$. It is well known that
$d \geq 2\nu-2$ but for the convenience of the reader we bring the proof of this fact.
\begin{lemma}
\label{lem0}
\begin{equation}\label{eqn1}
       d \geq 2\nu-2.
\end{equation}
\end{lemma}

\begin{proof}
\, Denote by $\pi: S \to \cp^2$ the projection map and let $C = f^{-1}(l)$ for a generic line $l$.
The curve $C$ is irreducible and smooth. Applying the Riemann-
Hurwitz's formula to the map $\pi|_C: C \to l$, we get $2g(C) - 2 =
-2\nu + d$, since all the ramification points of the map $\pi|_C$ are
of ramification index 2, and $(C^2)_{S} = \nu$, which implies
$$
   d =   2\nu - 2 + 2g(C) \geq 2\nu - 2
$$

A different proof will be given in Subsection \ref{subSecSingHyp}
when we discuss the geometry of surfaces with ordinary singularities
in $\pp^3$.
\end{proof}

\begin{remark} \label{remEvenDeg}
\emph{Note that the proof of the above Lemma implies that the
degree of a branch curve is even.}
\end{remark}


 We want to express
the Chern invariants $c_1^2(S)$ and $c_2(S)$ in terms of
$(d,c,\chi)$ and, equivalently, in terms of $(d,c,n)$, so we give 2 formulas for each invariant.

\begin{lemma} (see \cite{Ku})
\label{lem_c_1^2}
\begin{gather}
\label{eqn5}
   c_1^2(S) = 9 \nu -  \frac{9}{2} d - \frac{1}{2} \chi \\
   c_1^2(S) = 9 \nu -  \frac{9}{2} d +
           \left( \frac{(d-1)(d-2)}{2} - n - c \right)  - 1
\end{gather}
\end{lemma}

\begin{proof}

Let $\pi: S \to \pp^2$ be the ramified cover, $R = B^*$, the
ramification curve and $C = f^{-1}(l)$ for $l$ a generic line.
%

First, we want to compute $[R]^2$ and $[C]^2$.

By Riemann-Hurwitz, $K_S = -3 f^* ([l]) + [R] = -3[C] + [R]$. As $\pi: R \to B$ is a
normalization of the branch curve $B$, we apply
adjunction formula to $R$ we get
\begin{gather*}
   2g-2 = (K_S + [R]) \cdot R = (-3[C]  + 2[R])\cdot R =
   -3[C] \cdot R + 2[R] \cdot R =  -3 f^* [l] \cdot R + 2 [R]^2 =
   \\
   =  -3 [l] \cdot f_* [R] + 2 [R]^2 = -3 \deg B + 2 [R]^2 = -3 d + 2 [R]^2
\end{gather*}
and thus
$$
     [R]^2 = \frac{3}{2}d + g - 1.
$$
We also have
\begin{equation} \label{eqnCsquare}
    [C]^2 = f^* [l] \cdot [C] = [l] \cdot f_* [C] =  [l] \cdot (\deg f  [l] )   = \nu [l]^2 = \nu.
\end{equation}
\\
We can now compute $c_1^2(S)$:
\begin{gather*}
    c_1^2(S) = K_S^2 = (-3 [C] + [R])^2 = 9 [C]^2 - 6 [C] \cdot [R]
    + [R]^2 = %
    9 \nu - 6 d  + \frac{3}{2} d + g - 1 = \\
    = 9\nu - \frac{9}{2}d + g - 1  =
    1 = 9 \nu -  \frac{9}{2} d - \frac{1}{2} \chi
\end{gather*}
The expression in $(d,c,n)$-coordinates follows easily.
\end{proof}

\begin{lemma} \label{lem_c_2} (see \cite{Ku})
\begin{gather}
\label{eqn4}
   c_2(S) = 3 \nu - \chi - c,  \text{ or}  \\
   c_2(S) = 3 \nu + d^2 - 3d  - 3c - 2n.
\end{gather}
\end{lemma}
\begin{proof}
%

To compute $c_2(S)$ we use the usual trick of considering a
pencil of lines in $\pp^2$ passing through a
generic point $p \in \pp^2$ and its corresponding preimage with
respect to $\pi: S \to \pp^2$ -- the Lefshetz pencil $C_t$ of curves on $S$.

 We then apply the following formula on $C_t$
$$
  c_2(S) = \chi(S) = 2 \chi( \text{generic fiber} ) + \# (\text{singular
   fibers})- (\text{self-intersection of } C_t)
$$
(see, for example, \cite[section 4.2]{GH}).

%
The generic fiber of $C_t$ is a ramified cover of a line $l$ with
$d$ simple ramification points (i.e. ramification index 2 at every
point), and thus $\chi (\text{generic fiber}) = 2 \nu - d$ by the Riemann--Hurwitz
formula.

The number of singular fibers in the pencil $C_t$ is clearly equal
to the degree $d^*$ of the curve $B^{\vee}$ (the dual to the branch curve
$B$), which by the Pl\"ucker formulas for $B$ satisfies  $d^* = d(d-1) - 3c - 2n.$
The self-intersection $[C_t] ^2$ of the fiber equals to $\nu$
(by (\ref{eqnCsquare})). Thus

\begin{gather*}
   c_2(S) = \chi(S)= 2 (2 \nu - d) + d^* - \nu =
    2(2\nu-d) + (d(d-1) - 3c -2n) - \nu = \\
   = 3\nu + d^2 - 3d - 3c - 2n = 3 \nu - \chi - c.
\end{gather*}
\end{proof}

\begin{remark}
\emph{Equation (\ref{eqn4}) can be written as  an analog to Riemann-Hurwitz formula for the map $S \to \pp^2$
$$c_2(S) - \nu c_2(\pp^2) = -\chi - c,$$  as Iversen described in \cite{Iver}.}
\end{remark}

\begin{remark}
\label{remNumCuspsNodes} \emph{Inverting the formulas above, we
get $n$ and $c$ in terms of $c_1^2,c_2,\n$ and $d$:
\begin{gather*}
n = -3c_1^2(S) + c_2(S) + 24\nu + \frac{d^2}{2} - 15d, \\
c = 2c_1^2(S) - c_2(S) -15\nu + 9d
\end{gather*}
We use these formulas below in Subsection \ref{subsecExSingSur}.}
\end{remark}

The next two  results are rather surprising, as one gets an inequality for the branch curve which is independent
of the degree of the projection:
\begin{lemma} (see, .e.g., the introduction of \cite{Libgo}) Let $B \in B(d,c,n)$ a branch curve of a linear projection to $\pp^2$ of a surface of general type, where $d,c,n,\chi$ and $\n$ as above. Then
\begin{equation}
\label{eqnBogomolov}
    5\chi + 6c - 9d \leq 0
\end{equation}
or, equivalently, in $(d,c,n)$ coordinates:
$$
10n + 16c - 5d^2 + 6d \leq 0
$$

\end{lemma}
%
\begin{proof}
Substituting the expressions for $c_1^2(S)$ and $c_2(S)$ (from Lemmas \ref{lem_c_1^2},\ref{lem_c_2}) in terms of
$\nu$ and $(d,c,\chi)$  and $(d,c,n)$ into the Bogomolov inequality
$
   c_1^2(S) \leq 3c_2(S),
$
we get the desired inequality.
\end{proof}

There is, however, an inequality which is true for every branch curve, restricting the sum of the nodes
and the cusps (though it is  weaker than inequality (\ref{eqnBogomolov})).
\begin{lemma} Let $B \in B(d,c,n)$. Then
$$
15d - 5\chi - 6c > 0
$$
or
$$
10n+16c < 5d^2.
$$
\end{lemma}

\begin{proof}
Note that Nemirovski's inequality (see \cite{Ne}) for branch curves is
$$
\frac{3d - \chi}{3d - \chi - c} < 6
$$
or, equivalently
$$
15d - 5\chi - 6c > 0.
$$
\end{proof}

\begin{remark}
\emph{The variety $B(d,c,n)$ is not necessarily connected. See, for example, \cite{FrTe}, where it is proven that $B(48,168,840)$ has at least two disjoint irreducible components.
}\end{remark}

\subsubsection{Chisini's conjecture} \label{secChisiniConj}

The following theorem was known as Chisini's Conjecture, by now proved
by Victor Kulikov (see \cite{Ku}, \cite{Ku2}):

\begin{thm} \label{thmChisiniConj} Let $B$ be the branch curve of generic
 projection $\pi:~S \to \cp^2$ of degree at
least 5. Then $(S,f)$ is uniquely determined by the pair
$(\cp^2,B)$.
\end{thm}

Kulikov proved this conjecture for generic covers of  degree greater than 11 and for
 generic linear projections of degree greater than 4. Kulikov considered two surfaces $S_1, S_2$
ramified over the same branch curve, and studied the fibred product $S_1
\times_{\cp^2} S_2$, proving that the normalization of this fibred
product contradicts Hodge's Index Theorem if $(S_1,f_1)$ is not
isomorphic to $(S_2,f_2)$.


\begin{remark}
\emph{
We want to mention that a version of a Generalized Chisini's conjecture also exists, for surfaces
with normal isolated singular points:
\begin{conjecture}
Let $f_i : S_i \to \pp^2,\,i=1,2$ be two generic coverings with the same branch curve $B$ where $S_i$
can have singular points, denoted as $\Sing S_i$. Assume $f_1(\Sing S_1) = f_2(\Sing S_2)$. Then either
there exists a morphism $\phi : S_1 \to S_2$ s.t. $f_1 = f_2 \circ \phi$ or $(f_1,f_2)$ is an exceptional pair.
\end{conjecture}
See \cite{Gen} for the definition of an exceptional pair. This theorem was partially proven by V. S. Kulikov and Vik. S. Kulikov for $f_1, f_2$ generic $m$--canonical
coverings, for $m \geq 5$ (see \cite{ADE}) or when $max(\deg f_1, \deg f_2) \geq 12$ or $max(\deg f_1, \deg f_2) \leq 4$ (see \cite{Gen}).
}
\end{remark}

\begin{remark}
\emph{One of the theorems induced from the proof of the
Chisini's conjecture was the fact that a class of certain
factorization associated to the branch curve $B$ (i.e. the Braid
Monodromy Factorization)  determines the diffeomorphism
type of $S$ as a smooth 4-manifold.  We refer the reader to \cite{Mo0}, \cite{MoTe1}
for an introduction of this factorization, and to Kulikov and Teicher's proof
 \cite{KuTe} of the above theorem.}
 %
%
%
%
\end{remark}

\subsubsection{Representation-theoretic reformulation}

Let $G_i$ (resp. $\G_i$) be the local fundamental group of $\cp^2 - B$
at the neighborhood of a cusp (resp. a node) of $B$.  Note that each
$G_i$ is isomorphic to the group with presentation $\{a,b : aba =
bab\}$ and every $\G_i$ is isomorphic to the group with presentation
$\{a,b : ab=ba\} = \Z^2$.

Let $l$ be a line in $\cp^2$ in generic position with $B$, $p_i$ ($i =
1, \dots, d$) be the intersection points of $B$ and $l$, $p_*$ be a
generically chosen point in $l$ and $\gamma_i$ be a small loop around
$p_i$ starting and ending at $p_*$. The map $Free_d \to \pi_1(\cp^2 -
B)$ sending generators of $Free_d$ to $[\gamma_i]$ is epimorphic by
Zariski--Van Kampen theorem, and the classes $[\gamma_i]$ are called
\textsl{geometric generators} of $\pi_1(\cp^2 - B)$.

It is well known (see \cite{Mo} or \cite[Proposition 1]{Ku}) that given a ramified cover  $S \to \cp^2$, the monodromy map
$ \varphi:\p(\cp^2 - B)\to Sym_\nu$ satisfies the following three conditions:
\begin{enumerate}
  \item[(i)] for each geometric generator $\g$, the image
    $\varphi(\g)$ is a transposition in $\Sym_\nu$;
  \item[(ii)] for each cusp $q_i$, the image of the two geometric
    generators of $G_i$ is two non-commuting transpositions in
    $\Sym_\nu$;
  \item[(iii)] for each node $p_i$, the images of two geometric
    generators of $\G_i$ are two different commuting transpositions in
    $\Sym_\nu.$
\end{enumerate}

The inverse assertion is a group theoretic reformulation on the Chisini's theorem (\cite{Mo}):
%

\begin{prs} \label{prsChisiniEquiv} The map associating the monodromy
  representation with each ramified cover $S \to \cp^2$ gives an
  isomorphism of the set of the isomorphism classes of generic
  ramified covers of $\cp^2$ of degree $\nu$ with the branch curve $B$
  and the set of isomorphism classes of epimorphisms $ \varphi:\p(\cp^2 -
  B)\to Sym_\nu$ satisfying the  conditions (i),(ii) and (iii) above,
with respect to the action of $\Sym_\nu$ on the set of such
representations by inner automorphisms.
\end{prs}

\subsection{Construction of the variety of branch curves $B(d,c,n)$}
%
%
\label{secB_dcn}

Let $V = V(d,c,n)$ be the Severi-Enriques subvariety  in $|dh|$ of degree $d$ plane curves with $n$ nodes
and $c$ cusps. Let $B = B(d,c,n) \subseteq V$ 
the subset consists of branch curves. In this subsection we show that
 $B(d,c,n)$ is a subvariety of $V(d,c,n)$.
Although it is standard, we have not found it in
the literature, though references to its existence can be found in
 \cite{Wahl} or in \cite{Vak}. The following lemma proves that the variety of branch curves of ramified
 covers is a union of connected components of $V$. Using the same techniques in the following proof, and the
fact that the
Chisini's conjecture is proven (for generic linear projections), one can prove that also $B$ is a
union of connected components of $V$.

\begin{lemma} \label{lemBranchVar}
Over the field $k = \CC$, every connected component $V_i$ of $V$
either does not contain branch curves of generic covers at all, or
every curve $C \in V_i$ is a branch curve of a generic cover. Explicitly,
 every component of $B$ is a connected component of $V$.
\end{lemma}

\newcommand{\geomHom}{\mathop{\Hom_{ \text{geom} }}  }   

\begin{proof}

Let us fix a connected component $V_1$ of $V = V(d,c,n)$, let $p
\in V_1$, and let $C$ be the corresponding plane curve. Take $q
\in V_1,\, q \neq p$ and choose a path $I = [0,1] \to V_1$
connecting $p$ and $q$. Let us denote $G_C = \p(\pp^2-C), G_{C_t}
= \p(\pp^2-C_t)$ with $C_t \in V_1, t \in I$ where $C_1$
corresponds to $q$. As these curves are equisingular, we get an
identification of fundamental groups
$$
G_{C_t} \overset{\sim}{\rightarrow} G_C.
$$
For every $t \in I$. Consider the group $\Hom(G_C,\Sym_N)$ and its
subgroup $\Hom_{\text{geom}}(G_C,\Sym_N)$ of geometric homomorphisms
-- i.e., homomorphisms which satisfy the conditions (i),(ii),(iii)
above -- which can be empty. From the above identification, we get
a canonical set bijection from $\Hom(G_{C_t},\Sym_N) \to \Hom(G_C,\Sym_N)$
preserving the set of geometric homomorphisms. In particular,
$\Hom_{\text{geom}}(G_C,\Sym_N)$ is empty if and only if
$\Hom_{\text{geom}}(G_{C_1},\Sym_N)$ is empty, and thus $C$ is a branch
curve if and only if $C_1$ is. Therefore $B(d,c,n)$ is a union of
connected components of $V$ and thus it is a subvariety.
\end{proof}

\begin{remark}
\emph{We want to describe here on the action of the fundamental group $\pi_1(V)$
on $G = \p(\pp^2-C)$. Let $p \in V$, $C$ be the corresponding
degree $d$ plane curve and $U = \cp^2 - C$.
A loop $\gamma: I \to V$ (starting and ending at $p$), induces an
automorphism of the group $G = \pi_1(U)$, and thus an automorphism
of the set of representations $\Hom(\pi_1(U),Sym_N)$
which preserves the set of geometric representations
$\Hom_{\text{geom}}(\pi_1(U),Sym_N)$. To
describe it more explicitly, note that we can choose a line $l
\subset \cp^2$ in generic position to every $C_t$, $t \in I$
(since the set of lines in special position to a fixed curve in
$\cp^2$ forms a dual curve in the dual plane, and thus the space
of lines which are special to some $C_t$ is of real codimension 1
in the dual plane). Note that $l - l \intersect C = l \intersect U
\simeq \pp^1 - \{ d \text{ points}\}$. Let us now choose a base point
$a_*$ on $l$ not belonging to any of the curves $C_t$, and a
``geometric basis'' $\G$ of $\pi_1(U,a_*) = \pi_1(\pp^2 - C,a_*)$,
 which gives an epimorphism
$$
      e(\G): \pi_1(l \intersect U , a_*) \to \pi_1(U,a_*).
$$
 Recall that the group of classes of diffeomorphism of ($\cp^1 - d$ points)
modulo diffeomorphisms homotopic to identity can be identified
with the commutator of the braid group $B_d' = Braid_d/ Center(Braid_d)$
(see e.g. \cite{Mo0}).
A loop $\gamma \in \p(V)$ gives a diffeomorphism of $l \intersect
U$, which in turn induces an
automorphism of $\pi_1(l \intersect U)$, i.e. an element in
$\Aut(\pi_1(l \intersect U))$ or equivalently, an element $b_{\g} \in B_d'$.
%
%
It follows that there is a natural diagram
$$
   \xymatrix
   {
        { \pi_1(V) }  \ar[r]^{\alpha}\ar[rd]^{\beta}    & {\Aut \pi_1(U)}
                       \ar[r]  &  { \Aut (\Hom_{\text{geom}} (\pi_1(U), Sym_N)) } \\
        {  }                                  & {\Aut \pi_1(l \intersect U) \simeq B_d'} &  \\
   }
$$
and a commutative triangle:
$$
   \xymatrix
   {
        { \pi_1(V) }  \ar[r]^-{\alpha} \ar[rd]^{\beta}  & {\text{Im}(\alpha) \subseteq \Aut \pi_1(U)}
                        \\
        {  }                                  & {\text{Im}(\beta) \subseteq B_d'} \ar[u]&  \\
   }
$$
An element $b_\g \in B_d'$ 
which is the image of $\gamma$ admits a decomposition of $b_\g$ 
 into a product of canonical generators of $B'_d$, i.e. $b_\g = x_1^{\pm 1} \cdot ... \cdot x_k^{\pm 1}$.
Since $\pi_1(l \intersect U) = Free_d = \langle y_1,...,y_d
\rangle$,  we can describe explicitly the action of each $x_i$ on
$\Aut (Free_d)$:
$$
x_i(y_j) = y_j \,\,   \text{if} \,\,  j \neq i,i+1$$$$ x_i(y_i) =
y_{i+1}$$
$$ x_i(y_{i+1}) = y_{i+1}^{-1}\cdot y_i \cdot y_{i+1}.
$$
Thus, the action of an element $\g \in \pi_1(V)$ on the group $G = \pi_1(U,a_*)$ can be expressed
as a map on the generators $\{y_i\}$ of $G$ : $(y_i \mapsto b_{\g}(y_i) = (x_1^{\pm 1} \cdot ... \cdot x_k^{\pm 1})\cdot y_i)$
where $x_j \cdot y_i$ is given by the above action.
Note that this action is non-trivial in general, and thus
$\pi_1(V,p)$ acts generically non-trivially on the set of good
covers $S \to \cp^2$ ramified over a given curve $C$. However, in
a situation when such a cover is unique up to a deck
transformation, like in the case of a high degree ramified cover, (due to Chisini's conjecture), this action reduces to the
action of the deck transformation group $\Aut(S / \cp^2)$ which is the trivial group, for geometric reasoning.}
\end{remark}


\newcommand{\Kpure}{B_{res}}   
\newcommand{\Kres}{\Kpure}
\newcommand{\Qpure}{Q_{res}}   
\newcommand{\Qres}{\Qpure}

\section{Surfaces in $\cp^3$} \label{secSurInP3}

Let $X$ be a smooth surface in $\pp^r$ and $p: \pp^r \to \pp^2$ be
generic projection; we decompose $p$ as a composition of projections $
\pp^r \stackrel{p_1}{\to} \pp^3 \stackrel{p_2}{\to} \pp^2$ such that
$S = p_1(X)$ is smooth or has ordinary singularities in $\pp^3$. We begin in section
\ref{subsecSmHyp} with the examination of branch curves of smooth surfaces in $\cp^3$ and proceed to
singular surface in section  \ref{subSecSingHyp}.


\subsection{Smooth surfaces in $\cp^3$} \label{subsecSmHyp}

 Our goal here is to reformulate and give a more
modern proof to a result of Segre \cite{Se} published in 1930. Segre
proved that the set of singular points of the branch curve of a
smooth surface in $\cp^3$ is a special 0-cycle with respect to
some linear systems on $\cp^2$, i.e., it lies on some curves of
unexpectedly low degree. (We remind that a curve passing through
the singularities of a given one is called adjoint curve. See
Definition \ref{defAdjoint}).  For example, if $\deg S = 3$, we
get the following result Zariski published in 1929 (cf.
\cite{Za1}): the variety
of plane 6-cuspidal sextics
has two disjoint irreducible components. Every curve in
the first component is a branch curve of a smooth cubic surface
and all its six cusps are lying on a conic, while the second
component does not contain any branch curves.
(Miraculously, this condition does not define a subvariety of
positive codimension in the
variety of all plane curves of degree 6 with 6 cusps,
but rather selects one
of its two irreducible components, which was probably the most
surprising discovery of Zariski concerning this variety.)

In the following paragraphs we recall Segre's method for constructing some
adjoint curves to branch curves of ramified covers. The main result is the
following: a nodal--cuspidal curve $B$ is a branch curve iff there are two
adjoint curves of (some particular) low degree passing through all the
singularities of $B$ (see Theorem \ref{thmSegre}).  Though this result was
presented in  \cite{Kul} (by Val. S.  Kulikov) and in \cite{DA} (by J.
D'Almeida), our point of view is different, as we emphasize  the relations
between the Picard and Chow groups of 0--cycles of the singularities of the
branch curve.  We also investigate the connections between adjoint curves and
the sheaf of weakly holomorphic rational functions on a nodal--cuspidal curve
$C$. We hope that the study of the Picard group of branch curves and the study
of adjunction with values in sheaf of weakly holomorphic rational functions
(see \cite{Nara}) gives a new understanding of the work of Segre.

Let $S$ be a smooth surface of degree $\nu$ in $\cp^3$, and let
$\pi: \cp^3 \to \cp^2$ be a projection from a point $O$ which is not on
$S$. Let $B \subset \cp^2$ be a branch curve of $\pi$. It is easy
to see that the degree of $B$ is $d = \n(\n-1)$: indeed, $B$ is
naturally a discriminant of a homogeneous polynomial of degree
$\nu$ in one variable. The curve $B$ is in general singular,
however, for a generic projection it has only nodes and cusps as
singularities (see e.g. \cite{CiroFla}).

Assume now that $S$ is given by a homogeneous form $f(x_0, \dots
,x_3)$ of degree $\n$, and $O = (O_0,..,O_3)$ is a point in
$\cp^3$ which is not on $S$. The polar surface $\Pol_{O}(S)$ is
given by the degree $\n - 1$ form $\sum O_i f_i$, where $f_i =
\frac{\partial f}{\partial x_i}$. The following lemma is well
known:

\begin{lemma} Let $\pi:S \to \pp^2$ be the projection with center $O$. The ramification curve
  $B^*$ of $\pi$ is the intersection of $S$ and the first polar
  surface $\Pol_{O}(S)$.
\end{lemma}

Indeed, the intersection of $S$ and $\Pol_O(S)$ consists of such points
$p$ on $S$ that the tangent plane to $S$ at $p$, $T_p S$, contains the
point $O$.  This implies that the line joining $O$ and $p$ intersects
$S$ with multiplicity at least 2 at $p$.

Note that this gives yet another proof that $\deg B^* = \deg S \cdot \deg
(\Pol_O(S)) = \nu (\nu - 1)$.\\


\noindent {\textbf{Notation}}:
\begin{enumerate}
  \item [(1)] $H \in A_2 \cp^3$ is a class of hyperplane in $\cp^3$;
  \item [(2)] $h \in A_1 \cp^2$ is a class of a line in $\cp^2$;
  \item [(4)] $\ell^* = H|_{B^*}$, $\ell^*  \in A_0 B^*$;
  \item [(3)] $\ell = h|_B$, $\ell \in A_0 B$; 
\end{enumerate}

We also denote
\begin{enumerate}
  \item [(5)] $S'_O =  \Pol_{O}(S) \subset \cp^3$, and
  \item [(6)] $S''_O = \Pol^2_{O}(S) $ is the second polar surface to $S$ w.r.t. the point $O$;
it is given by a homogeneous form %
$ f'' = (\sum O_i \frac{\partial}{\partial x_i})^2 f = \sum O_{i}
O_{j} f_{ij}$ of degree $\n -2$.

 \item [(7)] We call a 0-subscheme
with length 1 at every point a \emph{0-cycle}.
 \item [(8)] Let $P
\subset B$ be the 0-cycle of nodes on $B$, and $P^*$ be its
preimage on $B^*$.  Note that $\deg P^* = 2 \deg P$, as can be
seen from Lemma \ref{lemNumNodeCuspSm}. \item [(9)] Let $Q \subset
B$ be the 0-cycle of cusps on $B$, and $Q^*$ be its preimage on
$B^*$.  Note that $\deg Q^* = \deg Q$ (see Lemma
\ref{lemNumNodeCuspSm}).
\item[(10)]$\xi$  be the 0-cycle of singularities
of $B$.
\end{enumerate}

From now on we assume that $O$ is chosen generically for a given surface $S$.
It follows that $B^*$ is smooth, and  $B$ has only nodes and cusps as
singularities. Already in the $19^{th}$ century the number of nodes and cusps
of a branch curve was computed for a smooth surface of a given degree.

\begin{lemma}[Salmon \cite{Sa}]\label{lemNumNodeCuspSm}
\emph{(a)} There is one-to-one correspondence between bisecant
lines for $B^*$ passing through $O$ and nodes of $B$. Moreover,
the number of bisecant lines through $O$ does not depend on $S$,
and is equal to
\begin{equation} \label{eqnNumNode}
     n = n(\nu) = \frac{1}{2}\n(\n-1)(\n-2)(\n-3)
\end{equation}
\emph{(b)} If $Q^*$ is the set  of points $q$ on $B^*$ such that the tangent
line $T_q B^*$ contains the point $O$, then the set $Q = \pi(Q^*)$ iis the set
of cusps of $B$.

\emph{(c)} Moreover, $Q^*$ is the scheme-theoretic intersection of
$B^*$ and the second polar surface $S''_O$. In other words, they
intersect transversally at each point of $Q^*$, and
$B^*\intersect S'' = Q^*$. In particular, the class $[Q^*]$ in $A_0 B^*$
is equal to $(\n - 2) l^*$.

\emph{(d)} It follows that $\deg Q$ does not depend on a choice of
the surface $S$, and is equal to
\begin{equation} \label{eqnNumCusp}
   c = c(\nu) = \n(\n-1)(\n-2)
\end{equation}

\end{lemma}
\begin{proof}
(a) The first statement is  geometrically clear; for the second
see \cite[art. 275, 279]{Sa}. Yet another proof is given below,
in Proposition \ref{prs2P3Q}. See also \cite[Chapter IX, sections
1.1,1.2]{SR} for a way to induce the formula for the number of bisecant of a
complete intersection curve in $\pp^3$ (i.e. the number $n+c$).
 For (b), see \cite[art. 276]{Sa}. (c) is a straightforward computation, and
 (d) follows from (c).  \end{proof}

\begin{lemma} \label{lemClassOfCusps}
Let $\ell \in A_0(B)$ be the
class of a plane section on $B$. Then

$$
    [Q] = (\n-2)\ell  \;\;\;   \text{ in } A_0(B),
$$

(2) The  equality above can be lifted to $\Pic B$: there is a Cartier divisor
$Q_0$ such that $\can (Q_0) = Q$ with respect to the canonical map
$$
   \can:  \Cartier(B) \to \Weil(B)
$$
associating Weil divisor with a Cartier divisor, and

$$
     [Q_0] = (\n-2)\ell
$$
in  $\Pic(B)$.

\end{lemma}

\begin{proof}  We have $Q = \pi_*(Q^*)$, and $Q^* = B^* \intersect
  S''_O$. Since $[S''_O] = (\n -2) H$ in $A_2 \cp^3$, we have $[Q^*] =
  (\n - 2) \ell^*$ in $A_0 B^*$, and thus
$$
   [Q] = [\pi_*(Q^*)] = \pi_* ( [Q^*] ) = (\n - 2) \pi_* \ell^*  = (\n - 2) \ell
$$
in $A_0 B$.

To see that $\pi_* \ell^* = \ell$ it is enough to consider a
hyperplane in $\cp^3$ containing the point $O$.

(2)  Consider the rational function $r = f''_O/H^{(\n - 2)}$, where $f''_O$ is by
definition the equation of the second polar $\Pol^2(O,S)$, and $H$ is an equation of a
generic hyperplane containing the projection center $O$. Since the curves
$B^*$ and $B$ are birational, $r$ can be considered as a rational function on
$B$, where it gives the desired linear equivalence.

\end{proof}

%

\begin{remark} \label{remNodeCuspCartier}
\emph{
Note that both cusps and nodes on a curve are associated with Cartier divisors
on the curve, even though these Cartier divisors are {\it not positive}.
For example, on the affine cuspidal curve $C$ given by the equation $y^2 - x^3
= 0$ the divisor $(y/x) = 3[0] - 2[0] = [0]$ is a principle
Cartier, but since $y/x$ is not in the local ring of the point
[0], it is not locally given a section of the sheaf $\OO_C$.
For the nodal curve $C$ given by the equation $xy=0$, the divisor
$\left( \frac{y-x^2}{y-2x} \right) = 3[0] - 2[0] = [0]$ is also a principle Cartier,
though not positive.  }
\end{remark}

\subsubsection{Example: smooth cubic surface in $\pp^3$} \label{subsecCubicSurExample}
Let $S$ be a smooth cubic surface in $\cp^3$. Then Lemma \ref{lemNumNodeCuspSm}
imply that $B$ is a plane curve with 6 cusps and no other singularities, and
Lemma \ref{lemClassOfCusps} implies that
$$
      [Q] = \ell
$$
in $A_0 B$.
$Q$ is, of course, not a line section of the curve $B$; the linear
equivalence above implies that the map
$$
     \PP H^0(\pp^2, \OO(1)) \isom \PP H^0(B, \OO(1)) \to |\ell|
$$
is not epimorphic, where $|\ell|$ is the set of all Weil divisors linearly equivalent
to a generic line section of $B$. 
Even though $Q$ is associated with a Cartier divisor $b/a$, this Cartier divisor
is not positive.

It is well known that 6 points in general position on $\cp^2$ do
not lie on a conic. As for the 6 cusps $Q$ on the branch curve we
have the following result of Zariski and Segre (see \cite{Za1},\cite{Se}). 

\begin{corollary} \label{corSexticBranch} All 6 cusps of a degree 6
  plane curve $B$ which is a branch curve of a smooth cubic surface
  lie on a conic.
\end{corollary}

\begin{remark} \label{remConicConst}
 \emph{\emph{Explicit construction of a branch curve of a cubic}. By change of coordinates a cubic surface $S$
is given by the equation
$$
      f(z) = z^3 - 3 a z + b,
$$
where $a$ and $b$ are homogeneous forms in $(x,y,w)$ of degrees 2
and 3, and the projection $\pi$ is given by $(x,y,w,z) \mapsto
(x,y,w)$. In these coordinates the ramification curve is given by
the ideal $(f, f') = (f,z^2 - a) = (z^3 - \frac{1}{2} b, z^2 - a)$ and the
branch curve $B$ is given by the discriminant
$$
   \Delta(f) = b^2 - 4 a^3
$$
In particular, one can easily see that it has 6 cusps at the
intersection of the plane conic defined by $a$ and the plane cubic
defined by $b$, as illustrated on the Figure 3. It is also clear
that the conic defined by $a$ coincides with one constructed in
 Corollary \ref{corSexticBranch}.
\begin{center}
  \epsfig{file=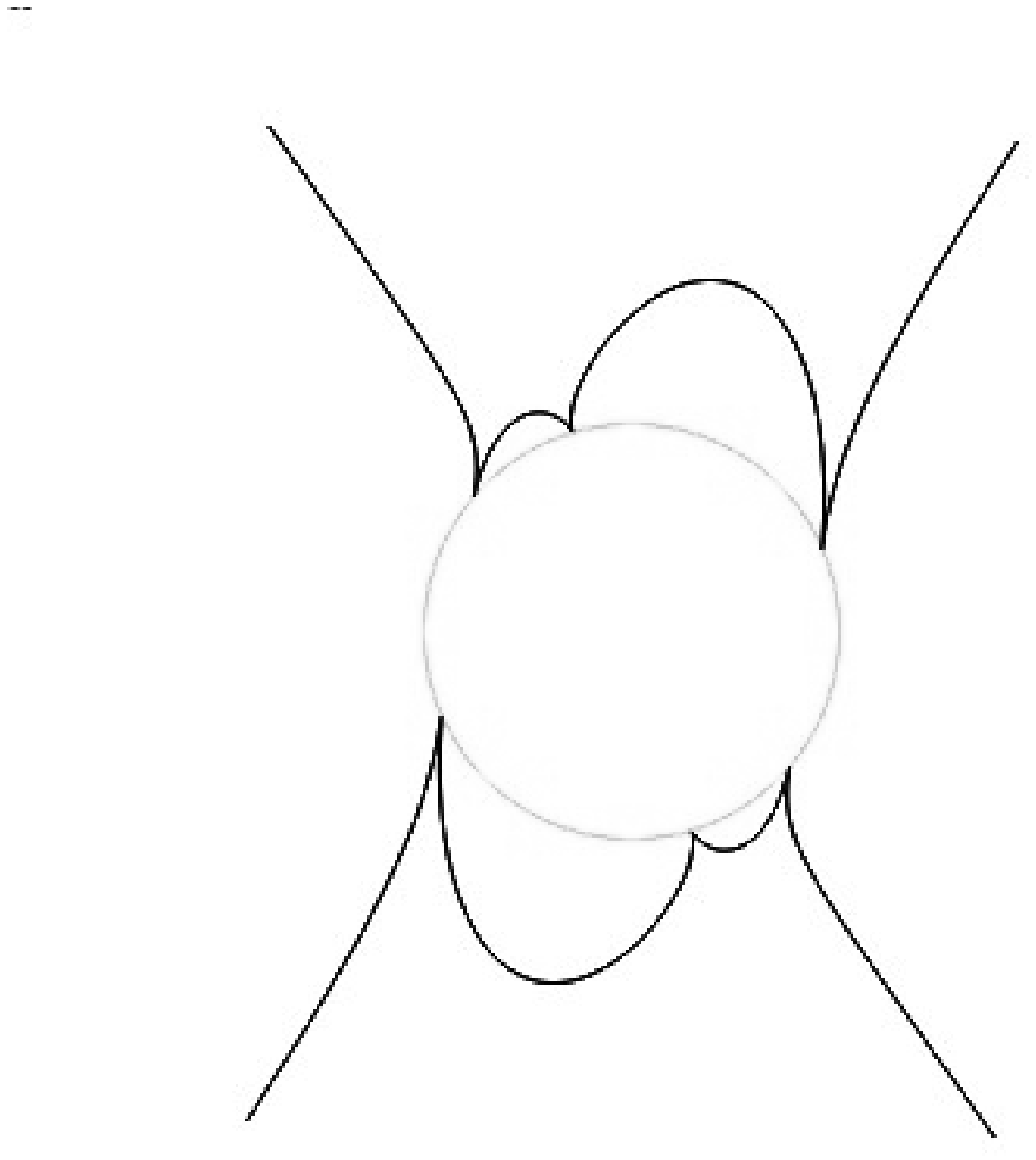, height=4cm, width=4cm}\\
  \small{Figure 3 : The branch curve of a smooth cubic surface}
\end{center}
The ideal of $Q^*$ is equal to $(f,f',f'') = (f,f',z) = (a,b,z)$.
Note that $z$ equal to $\frac{b}{2a}$ as a rational section of
$\OO_{B^*}(1)$. We want to explicate the linear equivalence of
$Q^*$ and the intersection of $B^*$ with the ``vertical" plane
(one containing the point $O$). For this, let $l(x,y,w)$ be a
linear form in $x,y,w$, and consider the rational function on
$B^*$
$$
   \phi = \frac{z}{l} = \frac{b}{2al}
$$
Then $\phi$ gives the
linear equivalence
$$
    0 = (\phi) = (b) - (a) - (l) = 3Q - 2Q - (l) = Q - (l),
$$
which gives an explicit proof that $[Q] = \ell$ in $A_0 B$. (We
used the fact that cubic $b$ is tangent to $B$ at the cusps, while conic
$a$ is not.) This example has a ``natural" continuation in
example  \ref{exmCubicSur}.}
\end{remark}

\subsection[Adjoint curves to the branch curve]{Adjoint curves to the branch curve}

We begin with the definition of an \textsl{adjoint} curve. This
type of curves will play an essential role when studying branch
curve.

\begin{definition} \label{defAdjoint}
Given a plane curve $C$, a second curve $A$ is said to be
\textsl{adjoint} to $C$ if it contains each singular point of $C$ of
multiplicity $r$ with multiplicity at least $r-1$. In particular, $A$ is
adjoint to a nodal-cuspidal curve $C$ if it contains all nodes and all cusps of
$C$.
 \end{definition}
For more on adjoint curves see \cite[$\S$ 7]{BN},
\cite[Chapter II, $\S$ 2]{SR}, or \cite{GV} for a more recent survey.

  Below,
following Segre, we construct more adjoint curves to $B$ (i.e. $W$, $L$, $L_1$) and
relate them to the geometry of $B^*$ in $\PP^3$.

We continue this subsection with Proposition \ref{prs2P3Q} from
\cite{Se} and we bring its proof for the convenience of the
reader.

\begin{prs} \label{prs2P3Q}
(a) One has in $A_0(B)$
$$
    2 [P] + 3 [Q] = \n(\n-2)\ell.
$$

(b) The   equality above can be lifted to $\Pic B$: there are Cartier divisors
$P_1$ and $Q_1$ such that in $\Pic B$:
\begin{gather*}
     \can(P_1) = 2 P, \\
     \can(Q_1) = 3 Q,    \;\;\;  \text{and}  \\
     [P_1] + [Q_1] = \nu (\nu - 2) \ell
\end{gather*}
\end{prs}

\newcommand{\Qt}{Q_{\tau}}
In fact, $Q_1$ is the canonically defined ``tangent" Cartier class $\Qt$.

\begin{proof} (following Segre \cite{Se}).

  Let us choose a plane $\Pi$ in $\cp^3$ not containing the point $O$,
and consider the projection  with center $O$ as a map to $\Pi$.
Let us also choose a generic point $O'  = (O'_0,O'_1,O'_2) \in
\Pi$, and let $B' = \Pol_{O'}(B)$ be the polar curve of $B$,
defined as follows: if $B$ is given by the homogeneous form
$g(x_0,x_1,x_2)$ of degree $d = \n(\n-1)$, then $\Pol_{O'}(B) =
\{\sum_{i=0}^2 O'_i\frac{\partial g}{\partial x_i}=0\}$.  Note that the first polar $B' = \Pol_O B$ is adjoint to $B$.

It is clear that
\begin{equation} \label{eq1}
      [B \cap B'] = 2P + 3Q + R,
\end{equation}
where $R$ (for ``residual") is the set of non-singular points $p$ on
$B$ such that the tangent line to $B$ at $p$ contains $O'$, and thus
\begin{equation} \label{eq2}
      [2P + 3Q + R] = (d-1) \ell
\end{equation}
in $A_0 B$ (Here we used the fact that $O'$ is generic, in particular,
it does not belong to $B$ and to the union of tangent cones to $B$ at
nodes and cusps.)

Let $R^*$ be the preimage of $R$ on $B^*$. We claim that $R^* = B^*
\intersect S'_{O'}$, where $S'_{O'} = \Pol_{O'}(S)$.  Indeed, if $p
\in R$, then $T_p B$ contains the point $O'$, and if $p^*$ is the
preimage of $p$ on $B^*$, then the tangent space to $S$ at $p^*$ can
be decomposed into a direct sum of the line $l$ joining $p$ and $p^*$
(and containing $O$) and the tangent line $T_{p^*}B^*$ which projects
to the tangent line $T_p B$, (as illustrated on Figure 4 below).
\begin{center}
  \epsfig{file=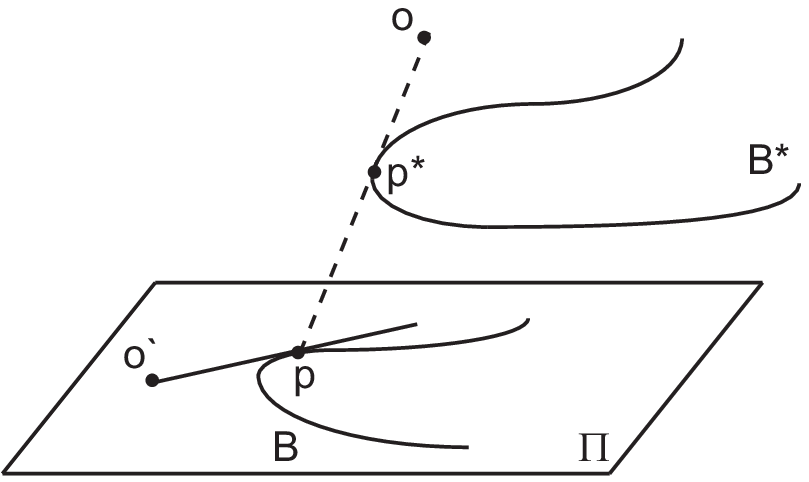}\\    
  \small{Figure 4 : $R^* = B^* \intersect S'_{O'}$}
\end{center}

It follows that
$$
   [R] = \pi_* ([R^*]) = \pi_* ([B^* \intersect S'_{O'}]) =
           \pi_* ((\n -1)\ell^*) = (\n - 1) \ell
$$
in $A_0 B$, and thus

$$
    [2P + 3Q] = [2P + 3Q + R] - [R] = %
           (d-1) \ell - (\n - 1) \ell = (d - \n) \ell = \n (\n-2) \ell
$$

The proof of the second part is parallel, as the Weil divisors $2P$ and $3Q$ can be lifted to $\Pic B$.

\end{proof}

Note that this gives yet another proof for the formula for the number
of nodes $n = n(\nu)$.

%
%

\begin{prs} \label{prsWcurve} There exist a (unique) curve $W$ in the
  plane $\Pi$ of degree $\n(\n-2)$ such that in $A_0(B)$ $$[W \cap B] = 2[P] + 3[Q].$$
\end{prs}

\begin{proof}
  By the previous proposition, the cycle $2[P] + 3[Q]$ is in the linear
  system $|\n (\n-2) \ell|$ on $B$. Note that $2P + 3Q$ is actually a Cartier divisor (see Remark
\ref{remNodeCuspCartier}). Now, since $\deg B = \n
(\n-1)$
  is greater than $\n (\n-2)$, there is a restriction isomorphism
$$
0 \to H^0(\Pi, \OO( \n(\n-2) ) ) \to H^0(B, \OO( \n(\n-2) ) ) \to 0
$$
%
which completes the proof.
\end{proof}

Note that $W$ is an adjoint curve to $B$ which is tangent to $B$ at
each cusp of $B$.

\begin{prs} \label{prsLcurve}Let $a = (\n-1)(\n-2)$.

(1) We have $$[2P + 2Q] = a\ell$$
In $A_0(B)$.

(b) The equality above can be lifted to $\Pic(B)$: there are Cartier divisors
$P_2$ and $Q_2$ such that in $\Pic(B)$:
\begin{gather*}
     \can(P_2) = 2 P, \\
     \can(Q_2) = 2 Q,    \;\;\;  \emph{and}  \\
     [P_2] + [Q_2] =  a \ell
\end{gather*}

(2) There is a (unique) curve $L$ of degree
  $a$ such that $$[L \cap B] = 2P + 2Q.$$
\end{prs}

\begin{proof}
(1) We have
$$
   [2P + 2Q] = [2P + 3Q] - [Q] = \n(\n-2)\ell  -(\n-2)\ell = (\n-1)(\n-2)
   \ell= a\ell.
$$
The computation in $\Pic(B)$ is parallel: we let $P_2 = P_1$ and $Q_2 = Q_1 - Q_0$.

(2) Note that $(\n-1)(\n-2) < \deg B$, which completes the proof.
\end{proof}

Note that $L$ is an adjoint curve to $B$ which is not tangent to $B$ at the cusps of $B$.

\begin{notation}
\emph{Let $\zeta_L$ be the Cartier divisor on $B$ given by restricting the
equation of
$L$ to $B$. Recall that $\zeta_L$ is supported on the 0--cycle of singularities $\xi$.}
\end{notation}

\begin{definition}
 Let $V(d,c,n)$ be the variety of plane curves of degree $d$ with $c$ cusps and $n$
 nodes, abd let $B(d,c,n)$ be the subvariety in $V(d,c,n)$ consisting of
 branch curves of all ramified covers of $\PP^2$.
\end{definition}

\begin{example}
\emph{ By substituting $\n=3$ and $\n=4$ we get the classical
example of a sextic with six cusps on a conic we discussed above, and the example of
a degree 12 curve with 24 cusps and 12 nodes, all of them are on a
sextic:
\item [(1)] The branch curve $C$ of a smooth cubic surface in $\pp^3$ is a
sextic
  with six cusps, $C \in B(6,6,0)$. We have $\deg L = 2$; two different
  constructions of this conic was given above in Corollary
  \ref{corSexticBranch} and Remark \ref{remConicConst}. See also Figure 3 above.
\item [(2)] The branch curve $C$ of a smooth quartic surface  in $\pp^3$ is of
  degree $12$, and has $24$ cusps and $12$ nodes, i.e., $C \in B(12,24,12)$.
  We have $\deg L = 6$. Moreover, the $24$ cusps lie on the intersection of a quartic and a sextic curves (see e.g. \cite{Se}).}

\end{example}

\subsection[Adjoint curves and Linear systems]{Adjoint curves and Linear systems}

We start with the following easy Lemma:

\begin{lemma} (Adjunction for a flag  $(\xi, \zeta, K, P)$)
Assume we are given a flag of 4 (arbitrary) schemes $(\xi, \zeta, K, P)$. Then there
is a diagram

$$
   \displayJ{J_{\xi,\zeta}}{J_{\xi,\zeta}}{J_K}{J_{\xi}}{J_{\xi,K}}{J_K}{J_{\zeta}}{J_{\zeta,K}}
$$

where $J_X = J_{X,P}$, and $X$ is either $\xi$, $\zeta$, or $K$.
\end{lemma}

\begin{corollary}
Coming back to our standard notations, let  $K$ be a plane curve,
$\xi \subset \zeta \subset K$ be a flag of 0-subschemes on $K$ such that
$\zeta$ is given by a positive {\it Cartier} divisor, and $\xi = \supp \zeta$.
In this case $J_{\zeta,K} =
\OO_K(- \zeta)$.  Then, given an integer $n < \deg K$, the diagram above gives
isomorphisms

$$
   \xymatrix
   {
     0 \ar[r]  &  H^0(\pp^2, J_{\xi,\zeta}(n)) \ar[r]       & H^0(K, J_{\xi,\zeta}(n)) \ar[r]  & 0 \\
     0 \ar[r]  &  H^0(\pp^2, J_{\xi}(n))       \ar[u]\ar[r] & H^0(K, J_{\xi,K}(n)) \ar[u]\ar[r]  & 0 \\
     0 \ar[r]  &  H^0(\pp^2, J_{\zeta}(n))     \ar[u]\ar[r] & H^0(K, \OO_K(-\zeta)(n)) \ar[u]\ar[r]  & 0 \\
     {}        &  0 \ar[u] &  0 \ar[u] & {} \\
   }
$$
\end{corollary}

\begin{corollary}
Assume that there is an integer $a$ and a positive Cartier divisor $\zeta = \zeta_0$ on $K$
such that there is a linear equivalence $\zeta_0 \sim a l$,
where $l$ is the class of a line section on $K$. (Such is the case of a branch curve and
the class $\zeta_L$ constructed above.)

Then, setting $n = a+i$, we get isomorphisms
$$
     j_{K,\zeta_0}(i): H^0(\pp^2, J_{\zeta_0}(a + i))  \to  H^0(K, \OO_K(i))
$$
for every $i \ge 0$.

In other words, adjoint curves on $\pp^2$ with given tangent conditions
at the singularities of the curve $K$ correspond to homogeneous forms on $K$ with given tangent conditions
at the singularities of the curve $K$ correspond to homogeneous forms on $K$.
\end{corollary}

We only need this isomorphism for $i = 0$; it implies that there is a curve
$L_0 \in  H^0(\pp^2, J_{\zeta_0}(a))$ of degree $a$ corresponding to the element $1 \in  H^0(K,
\OO_K)$, and $\zeta_0$ is locally given by the equation of $L_0$.
(This is exactly the case
of a branch curve $K=B$, where $\zeta_0 = \zeta_L$ and $L_0 = L$).

The isomorphism $j_{K,\zeta_0}(i)$ is given by $h \mapsto \frac{h}{f_{L_0}}$,
where $f_{L_0}$ is an equation of the curve $L_0$.

Our next goal is to study curves of various degrees $n > a$ containing the
0-cycle $\xi$ but restricting to different Cartier divisors with support on
$\xi$, not necessarily coinciding with  $\zeta_0$.
Assume that we are given a positive Cartier divisor $\zeta_1$ on $K$; We will
study adjoint curves restricting to $K$ as $\zeta_1$ .

Note that $J_{\zeta_1,K} = \OO_K(- \zeta_1)$, and consider the
restriction map
$$
   \res_K: H^0(\pp^2, J_{\zeta_1}(a+i) ) \to   H^0(K, \OO_K(- \zeta_1)(a+i) )
$$

To introduce notations we need to recall some basic facts about
linear equivalence of Cartier divisors.  Assume that we are given two positive
Cartier divisors $D_1$ and $D_2$ on a scheme $X$ and a linear equivalence $D_1
- D_2 = (r)$ for a meromorphic function $r$.  We realize both  $\OO_X(D_1)$ and
$\OO_X(D_2)$ as subsheaves of the sheaf $M_X$ of meromorphic functions on $X$,
and describe the isomorphism $\OO_X(D_1) \to \OO_X(D_2)$ given by the
function $r$ explicitly.  Locally, on a small enough affine open set $U
\subset X$, $U \isom \Spec A$, $D_1$ and $D_2$  are given by equations $f_1$
and $f_2$, $f_i \in A$,  $f_1/f_2 = r$ in the full ring of fractions $M_A$ of $A$, and $\OO(D_i)$
is given by the $A$-submodule $\frac{1}{f_i} \, A$ in $M_A$, $i = 1, 2$. The isomorphism
$j_r: \frac{1}{f_1} \, A  \to \frac{1}{f_2} \, A $, $a / f_1 \mapsto r \cdot (a/f_1)
= a/f_2 $ gives rise to an automorphism of the sheaf $M_X$ given by the
multiplication by $r$. Thus, globally, the sheaf automorphism $j_r: M_X \to M_X$ given
by the multiplication by $r$ takes $\OO(D_1)$ to $\OO(D_2)$.

Now, using the linear equivalence $\zeta_0 \sim a l$ on $K$, we get an isomorphism
$$
 j_r: \OO_K(- \zeta_1)(a+i) \to  \OO_K(- \zeta_1)(\zeta_0)(i) \isom \OO_K(\zeta_0 - \zeta_1)(i)
$$
given by multiplication with the rational function
$$
    r =  f_l^a / f_{L_0}
$$
where $f_l$ is an  equation of a line  $l$, and $f_{L_0}$ is the  equation of $L_0 \in  H^0(\pp^2, J_{\zeta_0}(a))$.
Thus the image for an adjunction belongs to the sheaf $\OO_K(\zeta_0 - \zeta_1) \tensor \OO_K(i)$,
which is the sheaf of meromorphic functions on $K$ with zeroes at $\zeta_1$ and poles at
$\zeta_0$, shifted by $i$.

Since we want to study adjoint curves to $K$, we are interested in positive
Cartier divisors of the form $\zeta_1 = \zeta_1^{\xi} + \zeta_1^{\res}$, where
$\zeta_1^{\xi}$ is supported on $\xi$, i.e.,
$\supp(\zeta_1^{\xi}) = \supp(\zeta_0) = \xi$,
and $\zeta_1^{\res}$ ($\res$ for "residual") is supported on the set of smooth points
of $K$.  Note that the sections of the sheaf $\OO_K(\zeta_0 -
\zeta_1)$ can locally be given by $r = h_1/h_0 \cdot g$, where
$h_i$ is the local equation for the Cartier divisor $\zeta_i$, and
$g$ is regular, i.e., $g \in \OO_{K,p}$.

Thus we introduce the following module and sheaf:

\begin{definition}
\label{definition:R}
{\rm For a commutative ring $A$, we define an $A$-submodule $R_A$ in the full
ring of fractions $M_A$,
$$
A \subset R_A \subset M_A,
$$
as the set of all fractions $r = g_1/g_0$ such that $\ord_p(g_1) \ge \ord_p(g_0)$
for each height one ideal $p$ of $A$.

Given a scheme $X$, one can define the sheaf $R_X$ ;
this sheaf is the subsheaf of the sheaf of
meromorphic functions $M_X$ given locally by fractions $r = g_1/g_0$ such that
$\ord_Z(r) = \ord_Z(g_1) - \ord_Z(g_0) \ge 0$ for each codimension one
subvariety $Z$ of $X$.}
\end{definition}

The sheaf $R_X$ coincides with the structure sheaf $O_X$ at the set of
smooth points of $X$, and there is a filtration
$$ O_X \subset R_X \subset M_X. $$

Moreover, we have the following easy Lemma:

\begin{lemma}
The normalization $N_A$ of $A$ in the full ring of fractions
$M_A$ is a submodule of $R_A$. I.e., there is a filtration
$$ A \subset N_A \subset  R_A \subset M_A $$
\end{lemma}

Note that sheaf $N_X$ coincides with the sheaf $\pi_*(\OO_{X^*})$, the
pushforward of the structure sheaf along the normalization $X^* \to X$.

Combining this all together, we get an adjunction sequence
\begin{multline*}
a_{K,i,\zeta_1}: \; \;
J_{\zeta_1, \pp}(a+i)    \stackrel{\res_K}{\to}
    \OO_K(-\zeta_1)(a+i)  \stackrel{\frac{f_l^a}{f_L}}{\Isom}
    \OO_K(\zeta_0-\zeta_1)(i) = \\     \\
    =
    \OO_K(\zeta_0-\zeta_1^{\xi})(i)(- \zeta_1^{res}) \subset
    \OO_K(\zeta_0-\zeta_1^{\xi})(i) \subset
    R_K(i)
\end{multline*}

and, taking union over all positive Cartier divisors $\zeta_1$, we finally get
our main adjunction

\newcommand{\aKi}{a_{K,i}}

\begin{equation}
\label{definition:adjunction-map-a}
     \aKi :
     J_{\xi,\pp}(a+i)
     \stackrel{r}{\to}   %
     R_K(i),
\end{equation}
where
$$
      r = f_l^a / f_L.
$$

Now we study the image of the map $a_{K,i}$.

\begin{definition} \label{defStrictTan}
{\rm
  Let  $C$ be a plane curve. We say that a line $l$ containing a cuspidal or
nodal point $p$ of $C$ is {\it strictly tangent} to $C$ at $p$ if $l$
intersects $C$ with multiplicity 3 at $p$.

We also say that a curve $C_1$ containing $p$ is strictly tangent to $C$ at the
nodal or cuspidal point $p$ of $C$ if $C_1$ intersects $C$ with multiplicity at
least 3 at $p$.
}
\end{definition}

Assume from now on that the adjoint  curve $L_0$ to $K$ is not (strictly)
tangent to $K$ at its singular points, and does not intersect $K$ elsewhere.

We want to introduce a sheaf of rational functions with denominator vanishing
exactly along $L_0$.  This sheaf is clearly the image of the adjunction map $a_K$
defined above.

\begin{definition}
  \emph{ Let $R^{L_0}_K$ be a subsheaf of $R_K$ consisting of sections $r$
  which can be given by $r = f/f_{L_0}$, where $f$ is a  homogeneous polynomial
  on $\pp^2$, and $f_{L_0}$ is an equation of the curve $L_0$.}
\end{definition}

\begin{prs}\label{prsRsepRl0Cuspidal}
If $K$ is a nodal-cuspidal curve and $L_0$ is an adjoint curve not tangent to $K$
at the singularities of $K$ and not intersecting it elsewhere, then the natural inclusion
$$
  R^{L_0}_K  \subset R_K
$$
is an equality. Moreover, they both coincide with the sheaf $\pi_* \OO_{K^*}$.
\end{prs}

\begin{proof} The proof follows easily from the fact that nodal and cuspidal
singularities of curves are resolved by a single blow-up, and, moreover, we can
take $t = f_1/f_0$ or $t = f_1/f_{L_0}$ as a local coordinate on the
resolution, where $f_1$  and $f_0$ vanish at the singular points of $K$ and
have separated tangents to $K$ at the singularities of $K$. In this way, both of
the sheaves are equal to $\pi_* \OO_{K^*}$, and thus they coinside.
\end{proof}

\begin{remark}\emph{
This proposition is an example for the analytic theory of weakly
holomorphic functions and universal denominator theorem (see, for example, \cite{Nara})
in case our base field is the field of complex numbers. In this case
the equation of the adjoint curve $L_0$ works as the universal denominator
for the sheaf of weakly holomorphic functions at each point of $K$.}
\end{remark}

Combining the proposition above and the construction of the adjunction map
$\aKi$ (which is essentially a division by the equation of $L_0$), we get the
following theorem:

\begin{thm} \label{thmRsep}
For a nodal-cuspidal curve $K$ and an adjoint curve $L_0$ as above,
the map $\aKi$ is epimorphic onto $R_K(i)$, and there is an exact sequence
$$
   \xymatrix
   {
    0 \ar[r] &   J_{K,\pp}(a+i)  \ar[r] &
    J_{\xi,\pp}(a+i)             \ar[r]^{\aKi} &
    {R^{L_0}_K(i)}                 \ar[r] \xyIsom[d]  & 0       \\
    &&& {R_K(i)}             &
   }
$$
In other words, adjoint curves of degree $a+i$ to the curve $K$ on the plane induce
rational functions on the curve $K$ for which $\ord_p(r) \ge 0$ for each point $p \in K$.

The map $\aKi$ is an isomorphism modulo ideal spanned by the equation of $K$.
\end{thm}

Passing to the global sections for $a + i < \deg K$, we get the following theorem:
\begin{thm}
\label{theorem:main-adjunction}
For $a+i < \deg K$, there are isomorphisms
$$
    \bigoplus H^0(\pp^2,J_{\xi}(a+i))
    \Isom
    \bigoplus H^0(K, R_K(i))
    \Isom
    \bigoplus H^0(K^*, \OO_{K^*}(i))
$$
\end{thm}

For higher degrees $i \ge \deg K - a$ one can modify these isomorphisms
readily to get a correct version including adjoint curves containing $K$ as a component.
\begin{proof}
This theorem  follows immediately from Theorem \ref{thmRsep} and Proposition \ref{prsRsepRl0Cuspidal}
if we take into account the projection formula for $\pi: K^* \to K$,
$$
   \pi_* (\OO_{K^*} (i) ) \isom
\pi_* (\OO_{K^*} \tensor \pi^* \OO_K(i) )  \isom
\pi_* (\OO_{K^*})  \tensor  \OO_K(i)   \isom
R_K \tensor  \OO_K(i).
$$

The meaning of the theorem is that plane curves through $\xi$ exactly
correspond to homogeneous functions on $K^*$.\end{proof}

\begin{remark} {\bf (Graded algebras interpretation)  }
{\rm
Assume we are given a smooth space curve $K^*$ not contained in a plane in
$\pp^3$ and a projection $p: K^* \to K$ to a plane curve $K$.  Since $K^*$ is
birational to $K$, in order to reconstruct $K^*$ from $K$, we have to say what
is the ``vertical coordinate $z$" on $K^*$ in terms of $K$. Since $K^*$ and $K$
are birational, the regular (holomorphic) objects on $K^*$ are rational
(meromorphic) objects on $K$, and thus we should have an equality of the form
$z = f_{n+1}/f_n$ for some integer $n$ and plane curves $f_n$ and $f_{n+1}$ of
degrees $n$ and $n+1$.

More precisely, let $S = \oplus S_i$, $S_i = H^0(K, \OO(i))$ be the graded
algebra of homogeneous functions on $K$, and $T$ be the graded algebra of
homogeneous functions on $K^*$. The inclusion $S \to T$ gives an isomorphism of
fraction fields $\QQ(S) \to \QQ(T)$, since  $K$ and $K^*$ are birational.  Now
$T_1 = S_1 \oplus k z$ for some element (``vertical coordinate") $z \in T_1$;
since $T_1 \subset \QQ(T) \isom \QQ(S)$, we would have $$ z =
\frac{f_{n+1}}{f_n} $$ for some integer $n$ and plane curves $f_n$ and
$f_{n+1}$ of degrees $n$ and $n+1$, both passing through the singularities of $K$.

} 
\end{remark}

\begin{corollary}
\label{choice-of-L1}

As in the previous remark, assume that we are  given a
smooth space curve $K^*$, a projection $p: \pp^3 \to \pp^2$ with center $O$ not
on $K^*$ such that $K = p(K^*)$ is a nodal-cuspidal curve, and an adjoint curve
$L_0$ of degree $a$ to $K$ which is smooth at the singularities of $K$ and
is not (strictly) tangent to $K$ there.

Then the ``vertical coordinate" $z$ on $K^*$,  $z \in H^0(K^*, \OO_{K^*}(1))$,
is the image of a uniquely defined plane curve $L_1$ of degree $a+1$ under the
adjunction map $a_{K,1}$ defined by the formula (\ref{definition:adjunction-map-a}).

In other words, we can choose $n = a$ in the remark above, and
$$ z = \frac{f_{L_1}}{f_{L_0}},$$
where $f_C$ is an equation of a  plane curve $C$, $C = L_0$ or $L_1$,
$\deg L_1 = a+1$, and the curve $L_1$ is not a union of $L_0$ and a line, i.e.,
is a ``new" adjoint curve.

The curves $L_0$ and $L_1$ are smooth at the points of $\xi$ and have different
tangents at every point $p \in \xi$.
%

\end{corollary}

\begin{proof}
There are two ways to prove it. First, this statement is a corollary of the
theorem \ref{theorem:main-adjunction}. The fact $L_1$ is ``new", i.e., not a
union of $L_0$ and a line, follows from the fact that $z$ is ``new", i.e., does
not come from a linear form on $\pp^2$ (explicitly, $z \in H^0(\pp^3,\OO(1))
\simeq H^0(K^*,\OO_{K^*}(1))$). The fact that $L_1$ is smooth at the singularities of $K$
follows from the
fact that the fraction $z = f_{L_1}/f_{L_0}$ resolves the singularities of $K$.

A more straightforward proof is the following: let $S$ be the graded homogeneous algebra
of $K$ and $T$ be the graded homogeneous algebra of $K^*$; and consider the
element $t = z \cdot f_{L_0}$ of $T_{a + 1}$.  It is enough to prove that $t$
actually belongs to $S_{a + 1}$, since then we can let $f_{a+1} = t$ and $z =
f_{a+1}/f_{L_0}$. Now this is an easy local computation for each singular point
of $K$, since the exact sequence
$$ 0 \to S_{a+1} \to T_{a+1} \to T_{a+1}/S_{a+1} \to 0 $$
is obtained from the sheaf exact sequence
$$ 0 \to \OO_K(a+1) \to p_* \OO_{K^*}(a+1) \to F(a+1) \to 0, $$
where $F$ is by definition the factor sheaf $p_* \OO_{K^*} / \OO_K $ , by passing to global sections:
$$ 0 \to H^0(K, \OO_K(a+1)) \arrow{p^*} H^0(K^*, \OO_{K^*}(a+1)) \to \coker p^* \to 0
$$
Since the factorsheaf $F$ is a product of sheaves supported at singular points
of $K$, this makes computing the image of $t$ in $H^0(K,F(a + 1) ))$ an easy local
computation at nodes and cusps.

The intuitive meaning of this computation is
that $f_{L_0}$ vanishes at the singularities of $K$, which implies that $t = z f_{L_0}$
is a regular (holomorphic) object on $K$, and thus belongs to $S_{a+1}$.
\end{proof}
%


In particular, this is the case when $K = B$ is a branch curve of a smooth surface
$S$ in $\pp^3$, where $\xi$ is the 0--cycle of singularities of $K$.
In this case we can take $L = L_0$, $a = (\nu - 1)(\nu - 2)$, where $\nu = \deg S$.
Segre refers to the existence of the second adjoint curve $L_1$ as
something known from the Cayley's "mono\"{\i}de construction" (see \cite[pg. 278]{Hal}).


\begin{remark} \label{remSummery}
\emph{Summarizing what is written above, the branch curve
$B$ has an adjoint curve $L$ of degree equal to $a$. In this case,
we have $$ z = \frac{f_{L_1}}{f_L}$$
The curves $L$ and $L_1$ are smooth at the points of $\xi = P + Q$ and have
separated tangents at every point $p \in \xi$.}
\end{remark}

\begin{remark}\emph{
Note that if the plane nodal-cuspidal curve $K$ has two adjoint curves
of degrees $n$ and $n+1$ with separated tangents at $\Sing K$ for any integer $n$,
then $K$ is the image of a smooth space curve $K^*$ under the projection from
$\pp^3$, but it is only $n = a = (\nu-1)(\nu-2)$ that $K$ may actually be a branch curve
of a surface projection.}
\end{remark}

\begin{remark}
\label{remO1Iso}
\emph{ We the following isomorphisms:
$$ H^0(\pp^2,J_{\xi}(a+1))
    \isom
  H^0(K^*, \OO_{K^*}(1)) \isom H^0(\pp^3,\OO(1)). $$
}

I.e., linear forms on $K^*$ correspond to adjoint curves of degree equal to $a+1$
on $K$.
\end{remark}

\begin{example} \label{exmCubicSur}
{\rm
For a cubic surface $f = z^3 - 3 a z + b$ the branch curve $B  = b^2
- 4 a^3$. The six  cusps of $B$ are given by the intersection of a conic and
a cubic $(a=b=0)$, and in this case $L = a$ is a conic in general position to
$B$ at the cusps, the cubic $W = b$ is strictly tangent to $B$ at the cusps (see definition \ref{defStrictTan}),
and both of them do not intersect $B$ elsewhere. We claim that $L_1 = W$ in this case.
Indeed, we have on $B^*$
\begin{gather*}
    f  = z^3 - 3 a z + b = 0, \\
    f' = 3 (z^2 - a) = 0
\end{gather*}
and thus
$$
     z = \frac{1}{2} \frac{b}{a}
$$
on $B^*$. It follows that $L_1$ is given by $b$.
} 
\end{example}

\begin{remark}
\emph{In the previous example we can choose the curve $L_1$
as any of the curves $W + l_0 L$, where $l_0$ is a linear form on $\pp$ (perhaps 0).
An easy computation shows that $L_1$ is strictly tangent to $K$ at $q \in Q$
iff $l_0$ contains the point $q$ (or if $l_0 = 0$), but even in this case $l_0$ the
curves $L$ and $L_1$ have different tangents at $q$.}
\end{remark}

\subsection{Segre's theorem}


Consider again a smooth surface $S$ in $\PP^3$ and  a projection $\pi: S \to
\PP^2$ with  a center $O \in \pp^3 - S$. %
Let $B$ be the branch curve of $p$, and $\xi$  be the 0-cycle of singularities
of $B$.

Consider now the graded vector space $\oplus H^0(\pp^2,J_{\xi}(n))$.  It follows
from the Segre's computation that $a = (\nu-1)(\nu - 2)$ is the smallest
integer such that there are adjoint curves of degree $a$ to $B$. The vector
space $H^0(\pp^2,J_{\xi}(a))$ is one-dimensional and generated by the the curve
$L$. Let $\zeta_L = L|_B$ be the corresponding divisor class in $\Pic (B)$.
Note that for $n=a$ the class $\zeta_L$ gives a canonical lifting of $2 \xi =
2P + 2Q$ to $\Pic B$, and thus $H^0(\cp^2,J_{\xi}(a)) \simeq
H^0(\cp^2,J_{\zeta}(a))$.  We have

\begin{gather}
 \zeta_L \in |a l|, \\
 [\zeta_L] = 2 \xi     \text{ in } A_0(B), \\
  k = k L \Isom H^0(\pp^2, J_{\zeta_L}(a)) \Isom  H^0(\pp^2, J_{\xi}(a)), \\
     H^0(\pp^2, J_{\zeta_L}(a))  \isom  H^0(B, \OO_B(-\zeta_L)(a)) \isom  H^0(B, \OO_B)
\end{gather}


Now $L$ is smooth at the points of $\xi$ and is not strictly tangent to
$B$ at these points by Remark \ref{remSummery}, and thus $\zeta_L$ is given by a tangent
vector to $p$ at each point $p \in \xi$, which follows from the descriptio of
Cartier divisors supported at nodes and cusps. The picture for
the branch curve of a smooth cubic surface is drawn below.

\begin{center}
  \epsfig{file=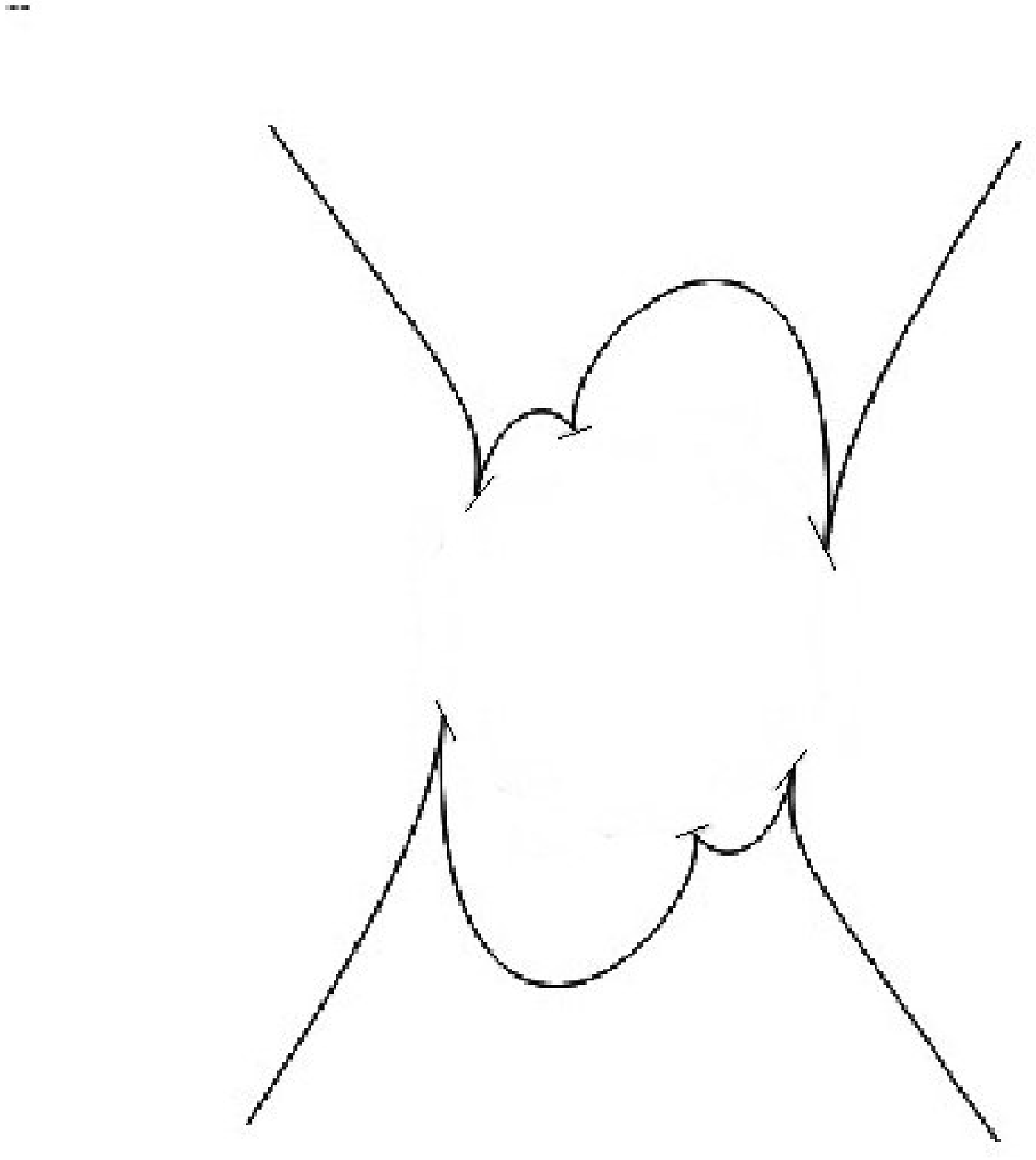, width = 6cm, height = 6cm}\\
  \small{Figure 5 : Cartier divisor $\zeta_L$}
\end{center}

Segre proves that this data is {\it sufficient} to reconstruct the surface $S$:

\begin{thm}[Segre]
\label{thmSegre}
   A plane curve $B$ of  degree $d = \n(\n-1)$  is a branch curve of a smooth
surface of degree $\n$ in $\pp^3$ if and only if

\begin{enumerate}
   \item [(1)] $B$ has $n =\frac{1}{2}\n(\n-1)(\n-2)(\n-3)$ nodes;
   \item [(2)] $B$ has $c = \n(\n-1)(\n-2)$ cusps;
   \item [(3)] There are two curves, $L$ of degree $a = (\n-1)(\n-2)$ and
$L_1$ of degree $a + 1$, which both contain the 0-cycle $\xi$ of singularities of $B$
and have separated tangents at the points of $\xi$.
\end{enumerate}
\end{thm}

\begin{proof}
The necessity of these conditions was proved in the preceding
sections. We now prove that they are sufficient.

Let $B$ be such a curve in the plane $\CP^2$. First, since $L$ is adjoint to $B$, the 0-cycle associated with
the scheme-theoretic intersection $L \cap B$ contains $2 \xi = 2P + 2Q$,
but by conditions of the theorem
$$
   \deg B \cdot \deg L = 2 \deg \xi = \n (\n-1)^2 (\n-2)
$$
It follows that the 0-cycle associated with $L \intersect B$ is
$$
   [L \intersect B] = 2 P + 2 Q.
$$

Let us denote $\xi = P + Q$. It follows immediately that $2 \xi$ is in the linear system $|a
\ell|$ on $B$, where $|\ell|$ is the linear system associated with the
given plane embedding of $B$. In particular, we conclude that
$$
       \xi \in \left| \frac{1}{2} a \cdot \ell \right|
$$

Note also that $[L_1 \cap B] = 2P + 2Q + R$, where $\deg R = d =
\n(\n-1)$.

Now the space $H^0(\CP^2, \Jk(a+1)) )$ contains a 4--dimensional
subspace of the form $k f_1 + k x f + k y f + k w f$, where $f_1$
is the equation of $L_1$ and $f$ is the equation of $L$. (Recall that $k$ is
our base field.)

Now consider the linear system on $B$ given by restriction of
$(f_1,xf, yf, wf) = k L_1 \oplus H^0(\cp^2,\OO(1)) \tensor kL$. It
has $\xi$ as a set of base points. It follows that it defines
a rational map
$$
    \phi: B - \xi \to \cp^3.
$$
Let $\pi: B^* \to B$ be the normalization of $B$. We claim that
the rational map $\phi$ can be lifted to give a regular map
$\phi^*: B^* \to \cp^3$. Indeed, we have the following lemma:

\begin{lemma*}[\textbf{A}] \label{lemModelP3}
    Let $B$ be a plane nodal-cuspidal curve with the set of singularities $\xi$,
and let $f  \in H^0(B,\Jk(j))$ and
        $f_1 \in H^0(B,\Jk(j+1))$ be non-zero elements determining adjoint curves
      $C = Z(f)$ and $C_1 = Z(f_1)$ on the plane,
such that $T_p C \neq T_p C_1$ at any point $p \in \xi$.

Let
$$
   \Omega = k f_1 \oplus H^0(\pp^2,\OO(1)) \tensor kf = (f_1, xf, yf, wf).
$$
Then the rational map $\phi_{\Omega}: B \dashrightarrow \pp^3$ can be resolved as
$$
\xymatrix
{
   B^*  \ar[r]\ar[d]^{\pi}  & \pp^3 \ar@{-->}[d]^{pr}  \\
   B     \ar[r]            &  \pp^2
}$$
where $\pi: B^* \to B$ is the normalization of $B$.
\end{lemma*}

Note that $T_p C \neq T_p C_1$ implies that $f_1 \notin H^0(\pp^2,O(1)) \tensor kf$,
and also that $\Omega \to T_p C$ is epimorphic at every point $p \in \xi$.

\begin{proof}
   It is clear that we only have to verify the statement at nodes and cusps of $B$
as well as smooth points $p$ on $B$ such that  $f_1(p) = f(p) = 0$.

   For a node $p$ we can choose coordinates in the local ring of $\pp^2$ at $p$ such
that $B$ is given by the equation $xy = 0$.

   Assume that $f_1$ is given by the equation
$a_{1,0} x + a_{0,1} y + ( \text{order } 2 \text{ terms} )$, and
$f$ is given by the equation
$b_{1,0} x + b_{0,1} y + ( \text{order } 2 \text{ terms} )$.
Note that $\phi_{\Omega} = (f_1, fx, fy,fw) = (f_1/f, x,y,w)$.
One can easily see that
$\phi_{\Omega}$ maps the point $p$ on the  branch $(y = 0)$ of $B$ to
$a_{1,0} / b_{1,0}$, and
the same point on the  branch $(x = 0)$ to  $a_{0,1} / b_{0,1}$. Thus, if
$a_{1,0} b_{0,1} - a_{0,1} b_{1,0} \neq 0$, then $\phi_W$ can be lifted to a
regular map $B^* \to \pp^3$ with a smooth image in the neighborhood of $p$.

   In the same way, in a neighborhood of a cups $B$ can be given by the local equation $y^2 - x^3 = 0$,
and thus
$$
    f_1/f =  \frac{a_{1,0} x + a_{0,1} y + (\text{order } 2)}{b_{1,0} x + b_{0,1} y + (\text{order } 2)} =
             \frac{a_{1,0} + a_{0,1} t + (\text{order } 2)}{b_{1,0}  + b_{0,1} t + (\text{order } 2)},
$$

where $t = y/x$ is the coordinate on the exceptional divisor in the resolution of the cusp.
Now it is clear that if $a_{1,0} / b_{1,0} \neq a_{0,1} / b_{0,1}$, then $\phi_W$ lifts to an embedding of the
exceptional divisor and thus the normalization of the curve as well.

   If now $p$ is a smooth point of $B$ such that $f_1(p) = f(p) = 0$,
then it is a standard fact that the map $(B - p) \to \pp^3$ can be uniquely extended
to the map $B \to \pp^3$ in a neighborhood of the point $p$, since $\pp^3$ is proper.
(Note also that we do not have any such points in the application of this Lemma below,
due to the intersection multiplicity computation for $C_1$ and $C$.)
\end{proof}

This gives a non-singular model $C \subset \cp^3$, and a projection
$\pi: C \to B$ with some center $O$. Note that if we start from a given ramification
curve $B^*$, the curve we reconstruct from $B$ coinsides with $B^*$.

\begin{lemma*}[\textbf{B}]
  If $B$ is a branch curve of the generic projection $\pi: S \to
  \pp^2$, where $S$ is a smooth surface in $\pp(V) \isom \pp^3$, and
  $B^*$ is the ramification curve of $\pi$, then there is an
  isomorphism $\PP(V) \to \PP(k L_1 \oplus H^0(\cp^2,\OO(1))
  \tensor kL)$ which takes $B^* \subset \pp(V)$ to $C$. In other words, the linear
  system $(f_1,xf,yf,wf)$ reconstructs the curve $B^*$.
\end{lemma*}

The idea of the proof is, as in the previous lemma, to set $z = f_1/f$ on $B^*$.

Recall that preimages of the nodes of $B$ belong to the bisecant lines to $B^*$
containing ithe point $O$, and preimages of cusps belong to the tangent lines
to $B^*$ containing the point $O$.  Considering tangent lines to $B^*$ as
a limiting case of bisecants to $B^*$,  we see that $B^*$ has  $$n + c =
\frac{1}{2}\n(\n-1)^2(\n-2)$$ of bisecants (and tangents) containing the
point $O$, which belong to a cone of order
$(\n-1)(\n-2)$ above $L$ with vertex $O$.

\begin{lemma*}[\textbf{C}]
   $B^*$ does not belong to a surface of degree $m < \n-1$.
\end{lemma*}
\begin{proof}
  Assume that $S_1$ is such a surface of degree $m$; we can assume
  that it is irreducible. Consider $S_1' = \Pol_O(S_1)$. First, if $S_1$ is smooth, note that
  $S_1'$ contains the preimage of the 0-cycle of cusps $Q^*$, since at
  each point $q \in Q^*$, the tangent line $l$ to $B^*$ is contained
  in $T_q S_1$, and also $l$ contains $O$, since $q$ projects to a
  cusp of $B$. It follows that $q \in S_1 \intersect S_1'$. Secondly, if $S_1$ is not smooth,
  then $S_1'$ still contains $q$.

  However, then it follows that the number of cusps $c \leq
  \n(\n-1)\cdot(m-1)$, which contradicts to assumption that $c =
  \n(\n-1)(\n-2).$
\end{proof}

We now have to prove that the model $B^*$ we constructed is a complete
intersection of a surface $S$ of degree $\nu$ and its polar $\Pol_O(S)$ of
degree $\nu-1$ with respect to the (fixed) point $O$ which is the center of the
projection $\pi: B^* \to B$. For these, following Segre, we apply the following
theorem belonging to Halphen (See \cite[pg. 359]{Hal}):

  \begin{thm*}[Halphen] \label{thmHalphen}
    Let $C$ be a space curve of order $a\cdot b$ in $\cp^3$ s.t. $a<b$ which has
    $\frac{1}{2}a(a-1)b(b-1)$ bisecants all lying on a cone of degree
    $(a-1)\cdot(b-1)$. Assume also that $C$ is not on a surface of degree smaller than $a$.
    Then $C$ is a complete intersection of two surfaces of degree $a$ and $b$.
  \end{thm*}

The inverse statement to
the Halphen's theorem is easy;  see \cite[art. 343]{Sa} or \cite[Chapter IX,
sections 1.1, 1.2]{SR}.

Alternatively, instead of invoking Halphen's theorem, one
can invoke a theory of Gruson and Peskine, as it is done
by D'Almeida in \cite{DA}; we cite his reasoning for the convenience of
the reader:

\begin{lemma*}[\textbf{D}] \cite[pg. 231]{DA}
  The curve $B^*$ constructed above is a complete  intersection of two surfaces
  of degrees $\n$ and $\n-1$.
\end{lemma*}

\begin{proof}
To prove the lemma, we introduce first the following definition:
%
\begin{definition}
Given a space curve $C$, we define its index of speciality as
$$
   s(C) = \emph{max} \{n : h^1(C, \OO_C(n))\neq 0\}.
$$
\end{definition}
Now we state the following \textbf{Speciality Theorem} of Gruson and Peskine \cite{GP}:
\\\\
\emph{Let $C$ be an integral curve in $\pp^3$ of degree $d$, not contained in a
surface of degree less than $t$. Let $s = s(C)$. Then $s \leq t + \frac{d}{t} - 4,$ with equality holding if
and only if $C$ is a complete intersection of type $(t,\frac{d}{t})$ (and thus $\OO_C(s)$ is special, i.e., $h^1(\OO_C(s)) \neq 0$)}.\\\\

Let now $p: B^* \to B$ be the projection from the point $O$. The
conductor of the structure sheaf $\OO_{B^*}$ in $\OO_B$ is
$Ann(p_*\OO_{B^*} / \OO_B)$,
which by duality is isomorphic to
$Ann(\omega_B / p_*(\omega_{B^*}))$ (see e.g. \cite[Chapter 8]{AK}).
By the definition  of the conductor,
we get that
$Ann(\omega_B / p_*(\omega_{B^*}))  = Hom(\omega_B , p_*(\omega_{B^*})) =  p_*(\omega_{B^*}) \tensor \omega_B^{\dual}$. It is well known that for a nodal-cuspidal curve, $H$ is a global section of the conductor sheaf
iff $H$ passes through the nodes and the cusps of the curve (see e.g. \cite[Proposition 3.1]{Chian}).

 By Serre duality, for
all $i,\, H^1(\OO_{B^*}(i)) = H^0(\omega_{B^*}(-i))$. Thus, the
minimal degree of the curve containing the singular points of $B$
is
$$
\n(\n-1) - 3 - s(B^*).
$$
Indeed, for a curve to pass through the singular points of $B$, the
conductor has to have sections, i.e. $p_*(\omega_{B^*}) \tensor (\omega_B)^{\dual}$ has sections.
Since we know that the
minimal degree of the curve containing the singular points of $B$ is $(\n-1)(\n-2)$, we get $s(B^*) = 2\n - 5$.

As $B^*$ does not lie on any surface of degree $\n-2$ (by Lemma (C)),
then the Speciality Theorem shows that $B^*$ is a complete
intersection of two surfaces of degrees $\n$ and $\n-1$ (taking $t~=~\n~-~2, d~=~\n(\n-1)$).
\end{proof}


Either way, by results of Halphen or Gruson-Peskine,
the curve $B^*$ is a complete
intersection of two surfaces, say, $S^{\n}$ and $F^{\n-1}$ of
degrees $\n$ and $\n-1$.

  We still have to prove that $B^*$ can be written as an intersection
of a surface of degree $\nu$ and its polar with respect to the given point $O$.

 Let $W = H^0(\pp^3,J_{B^*}(\nu))$ be the linear system of surfaces of
degree $\n$ containing $B^*$,
$$
    W =  k S \oplus \left( H^0(\pp^3,\OO(1)) \otimes k F \right) ,
$$
as for any complete intersection of type $(\nu, \nu-1)$. For a point $t \in \pp W$, let
$S_t$ be the corresponding surface of degree $\nu$ containing $B^*$.
(here we also denoted by $S$ and $F$ some particular equations for the surfaces $S$ and $F$,
even though they are defined only up to $G_m$ action).

Consider now the linear map
$$
      \partial_O : W = H^0(\pp^3,J_{B^*}(\nu))
      \to
      H^0(\pp^3,\OO(\n-1)),
$$
which maps $f$ to $Pol_O f = \sum O_i \partial_i f$, its polar with respect to the
fixed point $O$.  We claim that $\partial_0$ is injective. Indeed, if
$\partial_0(f)=0$, then $f$ vanishes on a cone of degree $\n$, containing the
curve $B^*$. Note that $F^{\n-1}$ vanishes on $B^*$ but also gives a degree
$\n-1$ form on every line generator of the cone $(f=0)$, which implies that the
projection map $B^* \to B$ has degree $\n-1$, which is not the case.

Now, for every $t \in \pp(W)$ and the corresponding surface
$S_t$ of degree $\nu$, consider the triple intersection
$$
 \eta_t = S_t \intersect F^{\n-1} \intersect \Pol_O S_t
$$

First, we have $S_t \intersect F = B^*$.  Let $R_t = S_t \intersect \Pol_O S_t$.
$R_t$ is a ramification curve for the surface $S_t$ with respect to the projection with the given
center $O$. 
We have $\eta_t = B^* \intersect R_t$.

Note that $Q^* \subset R_t$ for every $t$, since $B^*$ belongs to $S_t$ and has all tangent lines at the
points of $Q^*$ contain the projection center $O$.

Thus for every $t$ either the polar surface $\Pol_O S_t$ contains the curve $B^*$, or we have a decomposition
of 0-cycles on $B^*$ of the form
$$
  \eta_t = Q^* + r_t
$$

Also note that $Q^* \subseteq B^* \intersect \Pol_O F$ by the same geometric argument,
i.e., since at the points of $Q^*$ the tangent lines to $B^*$ contain the projection center $O$,
these points are on the intersection of $B^*$ with $O$-polar of every surface containing $B^*$.
But since these two 0-cycles have the same degree, they coinside. It follows that
$Q^* \in |(\n-2)h|$ on the curve $B^*$, where $h$ is a class of hyperplane section.

Now, since $\eta_t \in | \partial_0 S_t | = |(\n-1) h|$ on $B^*$, we have
$r_t \in  |h|$ on $B^*$ whenever $\Pol_O S_t$ intersects non-trivially with the curve $B^*$,
i.e., does not contain it.

Since $B^*$ is complete intersection, it is linearly normal (which follows
easily from the cosideration of Koszul complex). It follows that
$r_t$ gives a map
$$
     W \to H^0(B^*,\OO(1))
$$ from the 5-dimensional space $W$ to the 4-dimensional vector space $H^0(B^*,\OO(1)) \isom H^0(\pp^3,\OO(1))$.

Such a map must have a kernel, and let $S_0$ be the corresponding surface
in the linear system $|W|$.  It follows that $\Pol(O,S_0)$ contains the curve
$B^*$, and thus $B^* = S_0 \intersect \Pol(O,S_0)$, i.e., $B^*$ is a
ramification curve for the projection of the surface $S_0$ to $\pp^2$ with the given
center $O$. This finishes the proof.

\end{proof}

\begin{remark}
   \emph{We generalize Segre's theory for smooth surfaces in $\pp^N$,  $N>3$,
   in the subsequent paper \cite{FLL}.}
\end{remark}

Let us notice that the 0--cycle of singularities of the branch curve $B$ is special. We would like
to emphasize this in the next subsection.

\subsection[Special 0-cycles]{Special 0-cycles} \label{subSecSpCycle}
Let $\xi$ be a 0-cycle in $\cp^2$. Define the superabundance of $\xi$ (relative to degree $n$ curves) as:
$$
   \delta(\xi,n) = h^1 \Jk(n)   
$$

We have the following
\begin{lemma}
 If $\deg \xi \le \dim |nh|$, then
$$ \dim |nh - \xi| = (\dim |nh| - \deg \xi) + \delta(\xi,n),$$
in other words, $\delta(\xi,n)$ is the speciality index of the
0-cycle $\xi$ with respect to the linear system $|nh|$.
\end{lemma}


Also note that
$$
    \delta(\xi, n+1) \le \delta(\xi, n).
$$
Let now $\xi = P+Q$ - the zero cycle of singularities of $B$, and, as before, $a = (\n-1)(\n-2)$.

\begin{prs} (Speciality index of $\xi$)
There are following identities for the speciality index of $\xi$:
\begin{align*}
     \delta(\xi,a)    & =  \frac{1}{2} (\nu - 1) ( \nu - 2) (2 \nu - 5)  \\
     \delta(\xi,a+1)  & =  \frac{1}{2} (\nu - 3)( 2\nu^2  - 7\nu + 4)
\end{align*}
In particular, the 0-cycle $\xi$ is special with respect to $|a h|$ for all  surfaces of degree at least 3,
and special with respect to $|(a + 1) h|$ for all  surfaces of degree at least 4.
\end{prs}

\begin{proof}
  For the expected dimension $\vdim |\Jk(a)|$ we have
$$
  \vdim |\Jk(a)| = \dim |a h| - \deg \xi = \frac{1}{2} a (a+3) - \frac{1}{2}  \nu (\nu - 1)^2 ( \nu - 2)
$$
Since $a = (\nu - 1) (\nu - 2)$, we get
$$
   \vdim |\Jk(a)| =  \frac{1}{2} (\nu - 1) (\nu - 2) (5 - 2 \nu)
$$
Since, by definition of speciality index,
$$
  \dim |J_{\xi}(d)| = \vdim |J_{\xi}(d)|  + \delta(\xi,d)
$$
and since
$|J_{\xi}(a)| = \{ L \} $, we get the first equality.

The proof of the second formula is parallel; we use isomorphism $|\Jk(a+1)|
\isom \pp H^0(\pp^3, \OO(1))$ (see Remark \ref{remO1Iso}).
\end{proof}

\begin{example}[6-cuspidal sextic]
\emph{ Let $\xi_6$ be a 0-cycle of degree 6 on a plane which is an
  intersection of conic and cubic curves in $\pp^2$, given by a degree 2 (resp. 3) polynomial $f_2$ (resp. $f_3$).
  Note that generic 0-cycle of degree 6 is not like this, because generic 6 points do not belong to
  a conic.
Note that for $\xi_6$ given by $(f_2,f_3)$ there is a Koszul
resolution
$$
   0 \to \OO_{\pp^2}(n-5) \overset{
{\left[ {\begin{array}{*{5}c}
   {\small{f_3} }  \\
   {\small{f_2} }  \\
\end{array}} \right]}
}{\to} \OO_{\pp^2}(n-2) \oplus \OO_{\pp^2}(n-3) \overset{[-f_2
\,\,f_3]}{\to} J_{\xi_6}(n) \to 0
$$
%
An easy computation shows that
%
$\delta(\xi_6,2) = 1 $, $\delta(\xi_6,3) = 0 $, $\delta(\xi_6,4) =
0 $ (see Subsection \ref{subSecSpCycle} for the definition of
$\delta(\cdot, \cdot)$), and that $ H^0 J_{\xi_6}(2) = k f_2
$,\quad $ H^0 J_{\xi_6}(3) = k f_3 + k x f_2 + k y f_2 +  k w
f_2.$
Note that we start the computation from $n = 2$, since
$\delta(\xi,1) = 3$ is not a defect w.r.t. the linear system. Also
note that For a generic 0-cycle $\xi$
 of degree 6, $\delta(\xi,2) = 0$, otherwise it would lie on conic.}
 \end{example}

\subsection[Dimension of $B(d,c,n)$]{Dimension of $B(d,c,n)$} \label{subsecDimB3}
In this subsection, let $d(\nu)=\nu(\nu-1)$, $c(\nu)=\nu(\nu-1)(\nu-2) $,
 $n(\nu)=\frac{1}{2}\nu(\nu-1)(\nu-2)(\nu-3).$ Motivated by Segre's theory and the Chisini conjecture, We want to compute the dimension of the component
$B_3(\n)$ of $B(d(\n),c(\n),n(\n))$ which consists of branch curves of smooth surfaces in $\pp^3$ of degree $\n$ with respect to generic projection.

Let $S(\n)$ be the variety parameterizes  smooth surfaces in $\pp^3$ of degree $\n$. It is well known that
$\dim S(\n) = \frac{1}{6}(\nu+1)(\nu+2)(\nu+3)-1$. Let $B \in B_3(\n)$, a branch curve in the plane $\Pi$ of a smooth surface $S$ in $\pp^3$ of degree $\n$, when projected from the point $O=(0:0:0:1)$ (we work with the coordinates $(x:y:w:z)$). Now, $B$ is also the branch curve of a smooth surface $S'$ iff there is a
 linear transformation in $PGL_4(\C)$
 that fixes that point $O$, fixes the plane $\Pi$ (with coordinates $(x:y:w)$) and takes $S$ to $S'$. It is easy to
 see that the dimension of this subgroup of transformations $G$
 is $5$ (in $GL_4(\C)$), but as we are in a projective space, $\dim P(G) = 4$.

 By the Chisini's conjecture (proven completely for a generic projection, see \cite{Ku2}), the branch curve $B$ determines the surface uniquely up to an action of $P(G)$. Thus,
 \begin{equation}
 \dim S(\n) -4 = \dim B_3(\n).
 \end{equation}
Denote by
$V(\n) = \frac{1}{2}d(\n)(d(\n)+3) - n(\n) - 2c(\n)$  the virtual dimension of a family of degree $d(\n)$ curves
with $n(\n)$ nodes and $c(\n)$ cusps.

\begin{example}

\begin{enumerate}
   \item [(1)] \emph{For $\n=3,4$ , $\dim B_3(\n) = S(\n) - 4 = V(\n)$,  as expected (as for these
 branch curves, $c(\n) < 3d(\n)$. See \cite[p. 219]{Za2}).}
 \item [(2)]
 \emph{For $\n \geq 5$, $\dim B_3(\n) = S(\n) - 4 > V(\n)$. This gives examples of   nodal cuspidal curves,
 whose characteristic linear series is incomplete (for other examples see e.g. Wahl \cite{Wahl}).}
 \end{enumerate}
\end{example}

\subsection{Projecting surfaces with ordinary singularities.}\label{subSecSingHyp}
We bring here a short subsection on surfaces in $\pp^3$ with ordinary singularities, as we use it
in the next section, where we classify branch curves of small degree. The generalization of Segre's theory
for these surfaces will be presented in \cite{FLL}.

It is classical that (see e.g. \cite{GH}) any projective surface
in characteristics 0 can be embedded in $\pp^3$ in such a way that
its image has at most so-called ordinary singularities, i.e., a
double curve with some triple and pinch points on it. Any
projection $S \subset \pp^n \rightarrow \pp^2$ can be factorized
then as a composition of projections $S \subset \pp^n \rightarrow
\pp^3 \rightarrow \pp^2$ such that the image $S_1$ of $S$ in
$\pp^3$ has ordinary singularities in $\pp^3$. However, if we
project $S$ to $\pp^3$ first, and then from $\pp^3$ to $\pp^2$, we
get an extra component of the branch curve, which would be  the image of the double curve.

Assume now that we are given a degree $\n$ surface $S \subset
\cp^3 = \cp(V)$ with ordinary singularities and a point $O$ not on
$S$. Let $E^*$ be the double curve of $S$. Consider the projection
map $\pi: S \to \cp(V/l_O) \simeq \cp^2$. We define the
ramification curve $B^*$ of the projection as an intersection of
$S$ and the polar surface $S'_O$. (To justify this definition, one
can check that $S \intersect S'_O$ is the support
of the sheaf $\Omega^1_{S / \pp^2}$.)

 One can now see that $B^*$ can be decomposed as
$$
    B^* =  B^*_{res} + F^*,
$$
where $[F^*] = 2 [E^*]$, when $[F^*]$ is the Weil divisor
associated  with the 1-dimensional Cartier divisor $2 [E^*]$. Note
that  $B^*_{res}$ in its intersection with the smooth locus of $S$
is set-theoretically the set of smooth points $p$ on $S$ such that
the tangent plane $T_p(S)$ contains $O$. (To be more careful,
$B^*_{res}$ is the scheme-theoretical support of the kernel sheaf
of the canonical map $\Omega^1_{S / \pp^2} \to i_* i^* \Omega^1_{S
/ \pp^2} \to 0$, where $i$ is the embedding of $F^*$ to $S$. For a
different scheme-theoretic description of $E^*$ and $B^*_{res}$,
see \cite[Section 2]{Piene}).

It follows that the branch curve $B$ can also be decomposed as
$$
    B = B_{res} + 2E,
$$
where $E$ is the image of $E^*$.

Let $e = \deg E^*$ and $d = \deg B^*_{res} = \n(\n-1) - 2e$.  Now a
generic hyperplane section of $S$, $S \intersect H$, is a plane curve of degree $\n$
with nodes at the finite set $E^* \intersect H$, and thus
there is a restriction
$$
  0 \leq e \leq \frac{(\n-1)(\n-2)}{2},
$$ since the number of nodes of a plane curve can not exceed its arithmetic
genus.

 It follows that the pair $(\n,d)$
satisfies
$$
     2(\n-1) \leq d \leq \n(\n-1),
$$
as illustrated on Figure 6 below.

\begin{center}
  \epsfig{file=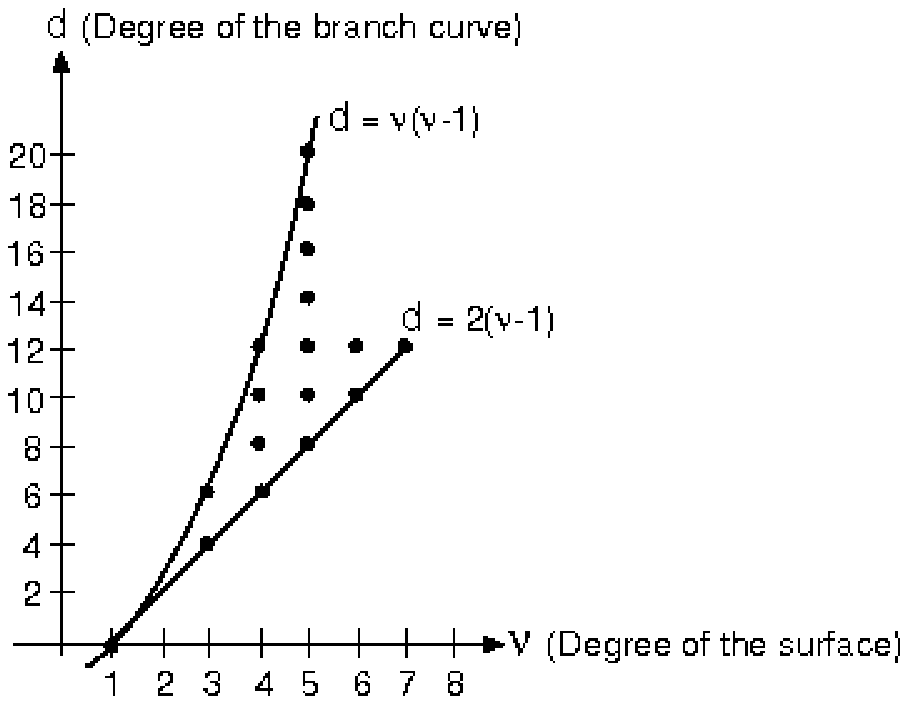} \\
  \small{Figure 6 : Geography of surfaces in $\cp^3$ with a double curve. \\
   We examine in Subsection \ref{subsecExSingSur} the cases where $\n = 3,4$.}
\end{center}

What is important here is that for a given $d$ there is only a finite
number of possible $\nu$'s such that a plane curve $C$ of degree $d$
can be a pure branch curve of degree $\nu$ surface in $\cp^3$ with ordinary singularities. \\

As before, we define $Q^*$ to be an intersection of $B^*$
and the second polar surface $S''_O$, i.e., as an intersection of
$S$, $S'_O$ and $S''_O$. However, for a singular surface $S$ not
all points of $Q^*$ form cusps on the branch curve. This is shown,
for example, at \cite[Chapter IX, section 3.1]{SR}.

\begin{notation}
\emph{Denote by $v^* \in E^*$ a point, such that the tangent plane to
$S$ at $v^*$ contains the center of projection $O$. These points
are called {\em vertical points} (or {\em points of immersion}) and we denote the set of such
points as $V^*$.\\ Denote by $T^*$ the set of triple points of
$E^*$, and by $t$ the number of these points. Let also $Pi^*$ be the set of pinch points of $E^*$ and let $p$ be the number of these points.}
\end{notation}

\begin{remark} \label{remNumOfPinchPositive}
\emph{Note that the number of pinch points $p$ is always positive (see \cite{FH}). We will use this fact to prove the inexistence of branch curves in $V(8,12,0)$ in Section \ref{subSecExample}.}
\end{remark}

The following Lemma is proved at \cite[Chapter IX, sections 3.1,
3.2]{SR}. This Lemma is the base for generalizing Segre's theory for singular surfaces, a generalization
which will be presented in \cite{FLL}.

\begin{lemma}
(1)
$Q^* = S''_O \intersect B^*$ can be decomposed as
$$
   Q^* =  (S''_O \intersect 2E^*)   +  \Qpure^*
$$
Note that the images of $(S''_O \intersect E^*)$ under the projection are
smooth points on $\Kpure$.

(2) points in $\Kpure^* \intersect E^*$ do not form cusps of the
branch curve, i.e., their images are smooth points on $\Kpure$.
Explicitly, $$\Kpure^* \intersect E^* = Pi^* + V^*.$$

(3) $ S''_O \intersect E^*$ can be decomposed as
$$
   S''_O \intersect E^* = V^*   +  3T^*
$$
and $ S''_O \intersect B^*_{res}$ can be decomposed as
$$
   S''_O \intersect B^*_{res} = V^*   +  Q^*_{res}
$$

\end{lemma}

\begin{remark} \label{remNumCuspsNodesSingSur}
\emph{Denote by $e^*$ the degree of $E^{\vee}$ the dual curve of $E$ in $\pp^2$. Given
a surface $S$ in $\pp^3$, we can express the number of nodes and
cusps of its branch curve $\Kpure$ by terms of $\n,e,e^*$ and $t$.
The following result is proved at \cite[Chapter IX, section
3]{SR}:}
\begin{gather*}
c = \n(\n-1)(\n-2) - 3e(\n-2) + 3t ,\\
n = \frac{1}{2}\n(\n-1)(\n-2)(\n-3) - 2e(\n-2)(\n-3) - 2e^* - 12t
+ 2e(e-1).
\end{gather*}
\end{remark}

\begin{remark} \label{remNumOfPinch}
\emph{Let $u$ be the number of components of $E^*$, and $g = \sum^u_{i=1} g_i$ the geometric genus of $E^*$. By \cite[pp. 624, 628]{GH} we can express $c_1^2,\,c_2$ and the number of pinch points $p$ by terms of $\n, e, t$ and $(g-u)$:
\begin{gather*}
c_1^2 = \n(\n-4)^2 - 5\n e + 24e + 4(g-u) + 9t,\\
c_2 = \n^2(\n-4) + 6\n + 24e - 7\n e + 8(g-u) + 15t,\\
p = 2e(\n-4) - 4(g-u) - 6t.
\end{gather*}}
\end{remark}

\subsubsection[Examples]{Examples} \label{subsecExSingSur}

We survey the well known examples of surfaces of degree 3 and 4 in $\pp^3$ with ordinary singularities and
use the results from Remarks \ref{remNumCuspsNodes} and
\ref{remNumCuspsNodesSingSur} in order to calculate the number of
nodes and cusps of the branch curve $\Kpure$ of the surface $S$.
These numbers can be expressed in
terms of $c_1^2(S), c_2(S), \deg(S)$ and $\deg(\Kpure)$ or in terms of $\n,e,e^*$ and $t$.\\\\

\noindent\underline{Degree 3 surfaces}\\

We know from the inequality above that $0 \leq \deg E^* = e \leq 1$,
in other words, the only cubic surfaces with ordinary singularities
are those with double line.

\begin{enumerate}
\item[(1)]$e = 0$.
  This is a smooth cubic surface, with the branch curve $B$ being a
  6-cuspidal sextic.
\item[(2)]$e = 1$.
%


  Such a surface has a double line, and thus $d = \deg \Kpure = 4$.
  Since we consider only generic projections, we can choose
  coordinates $(x,y,w,z)$ in $\cp^3$ in such a way that the projection
  center $O = (0,0,0,1)$ and the double line $E^* = l^*$ is given by
  equations $(z = w = 0)$. In these coordinates the projection is
  given by the (rational) map $(x,y,w,z) \mapsto (x,y,w)$, and $E = l$
  is the ``line at infinity'' $(w = 0)$ in the ``horizontal'' plane
  $(z = 0)$.

  It is easy to see that such a cubic surface can be given by a degree
  3 form
$$
   f = z^3 + a_1 z^2 + b_1 w z + c_1 w^2,
$$
where $(a_1,b_1,c_1)$ are homogeneous forms in $(x,y)$ of degree 1.

One can see from the definition of the normal cone (\cite{Ful})
that the normal cone to $l^*$ in $S$ is given by the degree 2 part
of $f$ in $(z,w)$, i.e., by the form
$$
  [f]_2 =  a_1 z^2 + b_1 w z + c_1 w^2
$$

We can consider $[f]_2$ as a section of $\OO(1,2)$ on the ruled
surface
$$
  \PP N_{l^* / \cp^3} \isom l^* \times \cp^1.
$$

Note that $[f]_2$, being a quadratic form of the variables $(z,w)$
with coefficients in $k[x,y]$, degenerates in the zeroes of its
discriminant $\Delta([f]_2) = b_1^2 - 4 a_1 c_1$. It follows that
there are 2 points $p_1$ and $p_2$ on $l^*$ where this quadratic form
degenerates into a double line, which proves that a cubic surface with
a double line has 2 pinch points.

Note also that $\Kpure^* \intersect l^*$ consists of such points
$p$ on $l^*$ such that one of the normal lines to $l^*$ in $S$ at
$p$ is the ``vertical'' line (one that contains the point $O$): it
is the only point of immersion. In the normal plane to the line
$l^*$ with coordinates $(z,w)$ this vertical line is given by the
equation $(w = 0)$. It follows that such points $p$ are exactly
those where $a_1$ vanishes. This gives just one point $p_0$,
different from the two pinch points $p_1$ and $p_2$ defined above,
and a decomposition
$$
   S''_O \intersect B^*_{res} = p_0 + \Qpure^*
$$
We have $\deg (S''_O \intersect B^*_{res}) = \deg \Kpure^* = 6
- 2 = 4$, and thus $\deg \Qpure^* = 3$. It follows that the pure
branch curve $\Kpure$ has 3 cusps. Note also that $\Kpure$ has no
nodes, since a plane quartic with 3 cusps is rational and can not have
any other singularities; i.e., we obtain a point $[\Kpure]$ in
$B(4,3,0)$.

\end{enumerate}

\noindent\underline{Degree 4 surfaces}\\

We should have $0 \leq e = \deg E^* \leq 3$.
\begin{enumerate}
\item[(1)]
  $e = 0$. This is the case of a smooth quartic surface with degree 12
  branch curve, which belongs to $B(12,24,12)$.

\item[(2)]$e = 1$. Let $S$ be a quartic surface with a double line $l^*$.
We have
$$
   B^* = 2l^* + \Kpure^*,
$$
where $d = \deg \Kpure^* = 4 \cdot 3 - 2 = 10$.

Arguing as above we can see that the normal cone to $l^*$ in $S$ can
be given by the equation
$$
  a_2 z^2 + b_2 w z + c_2 w^2 = 0
$$
for some homogeneous forms $(a_2,b_2,c_2)$ of degree 2 of
variables $(x,y)$.  It follows that $S$ has 4 pinch points on the
line $l^*$, and the intersection of $l^*$, $\Kpure^*$ and $S''_O$
consists of two (different) points $p_1,p_2$ which are the points
of immersion. It follows that
$$
   S''_O \intersect B^*_{res} = p_1 + p_2 + \Qpure^*,
$$
where $\deg \Qpure^* = \deg (S''_O \intersect B^*_{res}) - 2 = 18$.

It is known that a quartic surface with a double line is the image
of $\cp^2$ blown up at 9 points (see \cite[pg. 632]{GH}).
Computing its Chern invariants $c_1^2$ and $c_2$ and using the
formulas from Remark \ref{remNumCuspsNodes}, one can check that
the number of cusps is indeed 18, and the number of nodes is 8.
Thus $[\Kpure] \in B(10,18,8)$. Alternatively, since $e^* = t = 0$, by
remark \ref{remNumCuspsNodesSingSur} we find out that indeed
$[\Kpure] \in B(10,18,8)$.

\begin{remark}
\emph{It is easy to see from the above the classical fact that for a singular surface of degree $\n$ in $\pp^3$ with a double line, the number of pinch points is $p = 2(\n-2)$ and the number of the vertical points is $\n-2$.}
\end{remark}

\item[(3)]$e = 2$.
  In this case we have $\deg E^* = 2$. A curve of degree 2 in $\cp^3$
  is either a smooth conic contained in a plane, or a union of two
  skew lines, or a union of two intersecting lines, or a double line.
  By definition, the last two curves can not be double curves of a
  surface with ordinary singularities. 
  Both of the two remaining cases are actually realized, as explained,
  for example, in \cite{GH}.

  If the double curve is a smooth conic, then it is classical that the
  surface $S$ is a projection to $\cp^3$ of the intersection of two
  quadrics in $\cp^4$, (cf. \cite{GH}), and one can check (using remark \ref{remNumCuspsNodes}
  or \ref{remNumCuspsNodesSingSur}) that the branch curve is in $B(8,12, 4)$.

  If the double curve is a union of 2 skew lines, then it is known
  that $S$ is a ruled surface over elliptic curve (cf, say, \cite{GH},
  who deduces it from the classification of surfaces with $q = 1$).

From this classification one can conclude now that $c_2(S)=c_1^2(S)=0$, and, using the same
formulas as before, that the branch curve $B_{res}$ gives a point
in $B(8,12,8)$.

\item[(4)]$e = 3$.
  A double curve $E^*$ of a surface $S$ with ordinary singularities is
  either smooth, or has some triple points. Thus $E^*$ can be either
  (a) a rational space cubic, or (b) a non-singular plane cubic, or
  (c) a union of a conic and a non-intersecting line, or (d) a union
  of 3 skew lines, or (e) union of 3 lines intersecting in a point. It
  is explained, for example, in \cite{GH}, that only cases (a) and (e)
  are realized. 

  In the case (a) the surface $S$ is the projection of $\cp^1 \times
  \cp^1$ embedded with the linear system $|\ell_1 + 2\ell_2|$ to
  $\cp^5$, and the branch curve is in $B(6,6,4)$. We discuss this case
  in details in Subsection \ref{subSecExample}.

  In the case (e) the surface $S$ is the projection of the
  2-Veronese-embedded $\pp^2$ in $\cp^5$, and the branch curve is in
  $B(6,9,0)$; see the discussion in Subsection \ref{subSecExample}.

\end{enumerate}

\section{Classification of singular branch curves in small degrees} \label{subSecExample}


For $B$ a smooth curve of even degree $d$ defined by the equation
$\{f_B = 0\}$, let $\pi: S \to \cp^2$ a degree $\nu$ cover
ramified over $B$ which is generic in sense of Subsection
\ref{subsecGenProj}. We have that $d$ is even (see Remark
\ref{remEvenDeg}). By Zariski-Van Kampen theorem \cite{VK}, $\p
(\cp^2 - B) \simeq \Z/d\Z $ which is abelian. Since the monodromy
representation of a generic cover into the symmetric group should
be epimorphic (see the proof of Lemma \ref{lemNoAbelian}), we
conclude that $\pi$ is of degree 2, i.e., isomorphic to the double
cover given by $z^2 = f_B$ in the total space of the line bundle
$O_{\cp^2}(d/2)$.
\\
\begin{remark}
\emph{As was stated in subsection \ref{subsecGenProj}, we study generic linear projections $p: \pp^N = \PP(V) \to
\PP(V/W)$ (where $W \subset V$ be a codimension 3 linear subspace) where $\PP(W) \intersect S =
\emptyset$. Explicitly, if $S$ is a surface in $\pp^3$, then the projection is from a point $O \not\in S$.
However, see remark \ref{remQuarticDoubleLine}.}
\end{remark}

All possible non-smooth branch curves of degrees 4 and 6 are known
and we list them in the next paragraphs.  For each case we give
examples (and sometimes complete classification) of coverings with
a given branch curve. We then give all the numerical possible
singular degree 8 branch curves (see Theorem \ref{thmPosDeg8}).
\\

We denote
    $\langle a,b \rangle = (aba)(bab)^{-1}$,\quad %
  $ a^b = bab^{-1}$, and for the rest of this section we will use
  the coordinates $(d,c,n)$ in the variety of the nodal--cuspidal
  curves.
\\\\

\noindent\underline{Degree 4 singular branch curves}

There is only one branch curve of degree 4, as the following has
to be satisfied: $4 | n$, $3 | c$, and the
geometric genus $g(B) = (d-1)(d-2)/2 - n - c \geq 0$. 
It follows that the only possibility is
 $(c = 3,\, n = 0)$. This unique curve is the famous complexification of
  the classical deltoid curve, which is a cycloid with 3 cusps, i.e., the trace of a point on a
  circle of radius 1/3 rotating within a circle of radius 1. It is not hard to show that all other curves in
  $V(4,3,0)$ are obtained from the deltoid by linear transformation, since the dual curve
  belongs to $V(3,0,1)$, which is an irreducible space.

Zariski computed the braid monodromy for a deltoid using elliptic
curves \cite{Za1} and proved that $\p(\cp^2 - B)$ is isomorphic to
the group with presentation $$ \,\bigg \{ a,b : \langle a,b\rangle
= 1, a^2b^2=1 \bigg \},$$ where the notation $\langle a,b \rangle
$ was introduced above.  This is the dicyclic group of order 12.
The monodromy representation is the obvious one: $a \mapsto (1,2),
b \mapsto (2,3)$.

Zariski \cite{Za1} noted that the discriminant of a cubic surface
$S$ in $\cp^3$ with a double line is a plane curve of degree 6
which is a union of double line (the image of the double line of
$S$) and a quartic curve (which is straightforward), and moreover
proved that the residual quartic has 3 cusps. Thus the variety $
B(4,3,0)$ is not empty and thus $B(4,3,0) = V(4,3,0)$
\\\\

\noindent\underline{Degree 6 singular branch curves}

\begin{enumerate}
\item [(i)] The cases $(c=0, n > 0)$ and $c = 3$ are not realized.

  For $c=0$, $\p(\cp^2-B)$ is abelian (by Remark \ref{remNodalCurve}), and there are no generic covers ramified over $C$,
  as we argued in Lemma \ref{lemNoAbelian}. In the second case, $c = 3$ , Nori's result 
  we cited (see Equation (\ref{eq2_2})) implies that the group $\p(\cp^2-B)$
  is also abelian.

  It follows that there are no branch curves with these $(d,c,n)$
  triples, even though the corresponding varieties $V(6,0,n)$ and
  $V(6,3,n)$ are not empty.

\item [(ii)] $(c = 6, n = 0)$: This case was studied by Zariski, as we
  discussed in the introduction to Section \ref{subsecSmHyp}. If $S$ is a smooth cubic surface in $\cp^3$,
   then the branch curve
$B$ of a generic projection
  of $S$ to $\cp^2$ is in $V(6,6,0)$, and Segre's result we
  discussed (or a direct computation) shows that these 6 cusps lie on a
  conic. See subsection \ref{subsecCubicSurExample}  and Corollary \ref{corSexticBranch}.

Zariski proved the inverse statement: $C \in V(6,6,0)$ is a branch
curve if and only if 6 cusps of $C$ lie on a conic. It follows
from Zariski's work that branch curves form one of the connected
components of $V(6,6,0)$; and, moreover, Zariski proved the
existence of other connected components. Degtyarev proved in
\cite{D} that $V(6,6,0)$ has exactly two irreducible components.

Zariski also proved \cite{Za1} that for $B \in B(6,6,0)$ the group $\p(\cp^2-B)$
is isomorphic to $\Z/2 \ast \Z/3$, whereas for $C \in
V(6,6,0)\setminus B(6,6,0)$ the group $\p(\cp^2-C)$ is isomorphic to
$\Z/2 \oplus \Z/3$.

\begin{remark}
  \emph{ As a generalization of the above result, Moishezon \cite{Mo0} proved that the fundamental group of the
    complement of $B$ in $\cp^2$ is isomorphic to the quotient
    $\Braid_{\n}/Center(\Braid_\n)$ of the braid group $\Braid_{\n}$ by its center.  }
\end{remark}

\item [(iii)] \label{deg6_3} $(c = 6,\, n = 4)$.  Consider the surface
  $S = \cp^1 \times \cp^1$ embedded to $\cp^5$ by linear system $|
  \ell_1 + 2\ell_2 |$. Then $S$ is of degree 4 in $\cp^5$, and the
  image of its generic projection to $\cp^3$ is a quartic with a
  rational normal curve (the twisted cubic) as its double curve (see
  \cite[pg. 631]{GH}). The branch curve $B$ of $S$ is in
  $B(6,6,4)$ as can be seen from Remark \ref{remNumCuspsNodes} or from Remark \ref{remNumCuspsNodesSingSur} and
  it is known \cite{ZaPoincare} that the fundamental group $\p(\pp^2 - B)$ is braid group of the sphere with 3 generators (see \cite{AmS} for an explicit calculation).
    Note that $V(6,6,4)$ is irreducible since it is dual to $V(4,0,3)$.

\item[(iv)]$(c = 9,\, n = 0)$. First, we describe the variety $V = V(6,9,0)$. For a curve $B \in
  V$, its dual is a smooth plane cubic; this gives an isomorphism of V
  and an open subset in the linear system of plane cubics $|3h|$ consisting of smooth curves.
It follows immediately that $V$ is irreducible.

In this case $B(6,9,0)=V(6,9,0)$: there is a direct classical
construction of a cover with a given branch curve $C \in V(6,9,0)$
from the dual smooth cubic, discussed in Remark \ref{remSpecCons}. Moreover, every curve in
$V(6,9,0)$ is a branch curve of exactly four different ramified
coverings, the construction of which was given by Chisini (\cite{Ci}).
This is the only counterexample to the Chisini's conjecture (see
subsection \ref{secChisiniConj}).


More precisely, we have the following proposition:
\begin{prs} \label{PrsCountEx} \emph{(a)} Given a sextic $B$ with 9
  cusps and no nodes, there are four covers having $B$ as a branch
  curve. Three of them are degree 4 maps $\cp^2 \to \cp^2$, obtained
  as three various projections of Veronese-embedded $\cp^2$ in
  $\cp^5$, and the fourth one is of degree 3. The construction of the
  fourth is given in Remark \ref{remSpecCons}.

  \emph{(b)} The fundamental group $\pi(\cp^2 - B)$ has exactly 4
  non-equivalent representations into symmetric groups $\Sym_\nu$ for all
  $\nu$ which rise to smooth generic covers ramified over $B$.
\end{prs}

%

\begin{proof}
  (a) See \cite{C}.\\
  (b): Note that $G= \p(\cp^2 - B)$ was already calculated by Zariski in
  \cite{ZaTop}, showing that:
$$
G \simeq \ker(B_3(T) \to H_1(T)),
$$
where $B_3(T)$ is the braid group of the torus. We compute it here in a different method, using the
  degeneration techniques explained in \cite{MoTe3}. Let $S$ be the
  image of the Veronese embedding of $\cp^2$ into $\cp^5$; the branch
  curve $B$ of a generic projection $S \to \cp^2$ belongs to
  $V(6,9,0)$ (see e.g. \cite{MoTe0}). Since $V(6,9,0)$ is irreducible,
  it is enough to look at $B$. Now $S$ can be degenerated into a union
  of four planes is $\cp^5$, with combinatorics shown on the Figure 7
  below, as explained in \cite{MoTe3}.

 \begin{center}
\epsfig{file=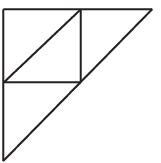}\\    %
\small{{Figure 7 : degeneration of $V_2$}}
\end{center}
Using the techniques of \cite{MoTe4},\cite{MoTe5}, one can prove that
$G$ has a presentation with generators\\ $\{\g_1, \g_{1'},\g_2,
\g_{2'},\g_3, \g_{3'}\}$ and relations
\begin{gather*}
\{\langle \g_2,\g_1 \rangle,
       \langle \g_2,\g_{1'} \rangle,
       \langle \g_{1'},\g_{2'}^{\g_2} \rangle,\quad
       \langle \g_3,\g_1 \rangle,
       \langle \g_3,\g_{1'} \rangle,
       \langle \g_3,\g_{1}^{\g^{-1}_{1'}} \rangle,\\
       \langle \g_{2'},\g_3 \rangle,
       \langle \g_{2'},\g_{3'} \rangle,
       \langle \g_{2'},\g_{3'}^{\g_3}  \rangle, \quad
        \g_2^{-1}\cdot \g_{2'}^{\g_2\g_{1'}\g_1} ,
       \g_3^{-1}\cdot \g_{3'}^{\g_3\g_{1'}\g_1},
       \g_2^{-1}\cdot \g_{2'}^{\g_3^{-1}\g_{3'}^{-1}} , \\
       \g_1\g_{1'}\g_2\g_{2'}\g_3\g_{3'} \}
\end{gather*}
Using GAP \cite{GAP2007} one can prove that $G$ is actually
generated by the set $\{\g_{1'}, \g_2, \g_3, \g_{3'}\}$.
and having the following relations
$$\bigg \{
 \g_{3}^{-1}\g_{2}\g_{3}\g_{2}\g_{3}^{-1}\g_{2}^{-1},
 \g_{1'}^{-1}\g_{3'}\g_{1'}\g_{3'}\g_{1'}^{-1}\g_{3'}^{-1},
 \g_{2}^{-1}\g_{3'}^{-1}\g_{2}^{-1}\g_{3'}\g_{2}\g_{3'},
 \g_{2}\g_{1'}\g_{2}\g_{1'}^{-1}\g_{2}^{-1}\g_{1'}^{-1},$$$$
 \g_{3}\g_{1'}\g_{3}\g_{1'}^{-1}\g_{3}^{-1}\g_{1'}^{-1},
 \g_{3'}\g_{1'}^{-1}\g_{3'}^{-1}\g_{3}^{-1}\g_{1'}^{-1}\g_{2}^{-1}\g_{3'}^{-1}\g_{2}\g_{3}^{-1}\g_{2}^{-1},
 \g_{3'}^{-1}\g_{1'}\g_{3}\g_{1'}^{-1}\g_{3'}\g_{1'}\g_{2}\g_{3}\g_{3'}\g_{2},$$$$
 \g_{3}\g_{3'}\g_{2}\g_{3'}^{-1}\g_{3}^{-1}\g_{3'}\g_{2}^{-1}\g_{1'}^{-1}\g_{3'}^{-1}\g_{3}^{-1}\g_{1'}^{-1}\g_{3'}^
1\g_{3}\g_{2}^{-1}\g_{3'}^{-1}\g_{2}\g_{3}^{-1}\g_{3'}^{-1} \bigg
\}. $$
It follows that it suffices to look for the homomorphisms $G
\rightarrow Sym_\nu$, when $\nu=3,4,5$, since 4 transpositions can generate at most
symmetric group on 5 letters. Note also that the homomorphisms, in order to correspond to
generic covers, have to satisfy Proposition \ref{prsChisiniEquiv}. Using GAP again, one shows that the
only
epimorphisms are  the following:\\

$\p(\cp^2 - B) \rightarrow Sym_3: $
\begin{enumerate}
\item[(1)]$  \{\g_{1'}, \g_2, \g_3, \g_{3'}\}\to \{(1,3), (2,3), (1,2), (1,2)\} $
\end{enumerate}
$\p(\cp^2 - B) \rightarrow Sym_4: $
\begin{enumerate}
\item[(2)]$ \{\g_{1'}, \g_2, \g_3, \g_{3'}\}\to \{(2,4), (2,3), (3,4), (1,2)\}$
\item[(3)]$ \{\g_{1'}, \g_2, \g_3, \g_{3'}\}\to \{ (1,3), (2,3), (3,4), (1,2)\}$
\item[(4)]$ \{\g_{1'}, \g_2, \g_3, \g_{3'}\}\to \{ (2,4), (2,3), (1,2), (1,2) \}$
\end{enumerate}
and there are no epimorphisms to $\Sym_5$.
\end{proof}

\begin{remark} \label{remSpecCons} \emph{We recall a construction of
    Chisini (see \cite{Ci} or \cite[Section 3]{C}). Let $B \subset
    \cp^2$ a curve with nodes and cusps only, such that its dual $A =
    B^{\vee} \subset (\cp^2)^{\vee}$ is a
    smooth curve of degree $d$. Let
    $$
    \Sigma = \{(\lambda,y) \in (\cp^2)^{\vee} \times \cp^2
    : \lambda(y) = 0,\, \lambda \in A\}
    $$
    and $\psi : \Sigma \to \cp^2$ be the projection to the second
    factor.  For a given point $y \in \cp^2 - B, $ the line
    $l_{\lambda}$ in $\cp^2$ is not tangent to $A $ (i.e. intersects
    $A$ in $d$ distinct points) iff $\psi^{-1}(y)$ has $d$ points.
    Hence $\Sigma$ is a degree $d$ covering of $\cp^2$ with $B$ as the
    branch curve. For $d=3$ we get the fourth example in Proposition
    \ref{PrsCountEx}, a degree 3 ramified cover of $\cp^2$ branched
    along a 9-cuspidal sextic. Note that this does not contradict to
    the fact that the only plane curves which are branch curves of
    projections of smooth cubic surfaces in $\cp^3$ are sextics with
    six cusps, since the surface constructed above is naturally
    embedded into $Fl \doteq \{(l,x) : l $ is a line in $\cp^2, x \in
    \cp^2, x\in l\} \subset (\cp^2)^{\vee} \times \cp^2$,
    and not in $\cp^3$.\\}
\end{remark}

\end{enumerate}
\noindent\underline{Degree  8 singular branch curves}

\begin{thm} \label{thmPosDeg8}
The only degree 8 singular branch curves of generic linear projections have either 9 cusps and
12 nodes or 12 cusps and 4  nodes.
\end{thm}

\begin{proof}
By simple calculations (see subsection
\ref{subsecNodeCupsCur},\ref{subsecBranch} for the obstructions),
one can conclude that there are only finite number of
possibilities for $c$ and $n$ for a degree 8 singular branch
curve.

It was proven by Zariski-Deligne-Fulton theorem
\ref{remNodalCurve} on nodal curves that the case $(c=0, n>0)$
cannot be realized as a branch curve. The cases  $(c=3, n > 0)$,
$(c=6, n > 0)$, $(c=9, n = 0)$ and $(c=9, n = 4)$ are ruled out as
branch curves by Nori's theorem \ref{lemNoAbelian} (though the
corresponding nodal--cuspidal varieties are not empty). Moreover,
the case $(c=18,n=0)$ cannot be realized even as nodal--cuspidal
curve: By the Zariski's inequality
  (\ref{eqZar}), the number of cusps of a degree 8 curve should be
  less then 16, and thus $V(8,18,0)$ is empty. By considering the dual curve, it's easy to see that also
$V(8,15,4)$ is empty.

We are left to show that there are no degree 8 branch curves with
$(c,n) = (9,8),(12,0),(15,0), (12,8)$ (although the corresponding
nodal--cuspidal varieties are not empty since the genus of these
curves is less than 5).
\begin{enumerate}
\item [(i)] $(c = 9, n = 8)$.

Assume that there exists a surface $S \subset \pp^3$ such that its
branch curve is $B \in B(8,9,8)$, such that $\tilde{S}$ is its
smooth model in $\pp^5$ (i.e. $\tilde{S} \to S$ by generic
projection). Since $d \geq 2\n - 2$, $\n =
3,4$ or $5$. By the examples in subsection \ref{subsecExSingSur}
we see that $\n = 5$ and thus $e = 6$. In this case, by Lemma
\ref{lem_c_1^2} and \ref{lem_c_2}, we can see that
  $c_1^2(\tilde S) =  c_2(\tilde S) = 12$. The degree 6 double curve (of the quintic) cannot lie on a hyperplane for degree reasons. Therefore the canonical system $|K_{\tilde{S}}|$  is empty (since it is the pull-back
 of the linear system cut out by hyperplanes passing through the
 double curve. See \cite[pp. 627]{GH}). Therefore $p_g(\tilde{S}) = 0$.
But since $c_1^2 = c_2 = 12$, we get that
$\chi(\mathcal{O}_{\tilde{S}}) = 2$. Thus $2 = 1 - q + p_g$ or $q
= -1$ -- contradiction. Thus $B(8,9,8)$ is empty.

\item [(ii)] $(c = 12,\, n = 0$) Assume that there exist $B \in
B(8,12,0)$. s.t. it is the branch curve of a surface
  $S$ in $\pp^3$. By the same argument as in case (i) we see that $\n = 5$ and thus $e = 6$. In this case,
   by Lemma \ref{lem_c_1^2} and \ref{lem_c_2}, we can see that
  $c_1^2(S) = 17, c_2(S) = 19$. So by Remark \ref{remNumOfPinch}, we can find that the number of pinch
  points $p = 0$ -- but this cannot happen, by Remark \ref{remNumOfPinchPositive}. Therefore $B(8,12,0)$ is empty.

  \begin{remark}
  \emph{Note that  Zariski proved
  (\cite{Za0}) that the twelve cusps of $C \in V(8,12,0)$ cannot be the intersection of a
  cubic and a quartic curves. We conjecture that this restriction is directly linked to the
  fact that $B(8,12,0)$ is empty.}
  \end{remark}

  \begin{remark} \label{remQuarticDoubleLine}
  \emph{Considering a quartic surface $S$ with a double line in $\pp^3$, we can project it from a generic smooth point
  $O \in S$. The resulting branch curve will be a curve in $V(8,12,0)$ (see \cite{DuVal}). However, we do not consider this projection as generic. Note that this phenomena happens  also in other cases. For example, a branch curve in $V(10,18,0)$ of a generic projection does not exist, but if we project a smooth quartic surface in $\pp^3$ from a point on the quartic, the branch curve of this projection would be in $V(10,18,0)$.}
  \end{remark}

\item [(iii)]  $(c = 15,\, n = 0)$ As in cases (i) and (ii), we
can see that a surface $S$ with such a branch  curve could only be
a quintic in $\cp^3$ with a degree 6 double curve $E^*$
  with 3 triple points (by Remark \ref{remNumCuspsNodesSingSur}).
  Considering $\Pi$ -- the plane passing through these three points -- and looking at $E^* \cdot \Pi$, we see
  that $E^* \subset \Pi$. However, deg\,$\Pi \cap S = 5$, so such a surface does not exist.


   Note that in this case
  Zariski demonstrated in (\cite{Za0}) that the 15 cusps cannot lie
  on a quartic curve.

  \begin{remark}
  \emph{The nonexistence in cases (ii) and (iii) can also be proven by the method indicated in (i).}
  \end{remark}

  \item [(iv)] $(c = 12,\, n = 8$). The variety $V(8,12,8)$ is
  irreducible, since it is dual to the $V(4,0,2)$.  If $S$ is a
  quartic surface in $\cp^3$ which double curve is a union of two skew
  lines, then we prove in Subsection \ref{subSecSingHyp} that a branch
  curve of $S$ is in $B(8,12,8)$. By \cite{ZaTop},
$\p(\pp^2 - B) \simeq \ker(B_4(T) \to H_1(T)).$ However the double curve of an image a smooth surface in $\pp^N$ in $\pp^3$ is an irreducible curve (unless the surface is the Veronese surface, where in this case the double curve is a union of three lines. See e.g. \cite[Theorem 3]{Mo1}). Thus  $B$ is not a branch curve of generic linear projection.

\end{enumerate}

We now shall construct degree 8 branch curves with $(c,n) =
(9,12),(12,4)$. With this we covered all the possible
numerics for the possible number of nodes and cusps of a degree 8
branch curve.

\begin{enumerate}
\item [(I)] $(c = 9, n = 12)$. First, note that $V(8,9,12)$ is
  irreducible, since it is dual to the variety $V(5,0,6)$.

  Now note that if we consider the Hirzebruch surface $F_1$ embedded
  into $\cp^6$ by $|2f + s|$ (where $f$ is the class of a fiber and
  $s$ is the class of a movable section, so that $f^2 = 1$, $f \cdot s
  = 1$, $s^2 = 1$.) A projection of this model of $F_1$ to $\cp^2$
  factorizes as a composition of a projection to $\cp^3$, where the
  image of $F_1$ is a quintic surface with a double curve of degree 6,
  and a projection from $\cp^3 \to \cp^2$. One can check that the
  branch curve $B$ of the resulting map has 9 cusps and 12 nodes (see
  \cite{MoRoTe} or Remark \ref{remNumCuspsNodes}) and that $\p(\pp^2 - B)$ is isomorphic to the braid group of the sphere with 4 generators (see \cite{ZaPoincare}).

\item [(II)] $(c = 12,\, n = 4$). We do not know whether
$V(8,12,4)$
  is irreducible. If $S$ is a smooth intersection of two quadrics in
  $\cp^4$, then the branch curve of a projection of $S$ to $\cp^2$ is
  in $B(8,12,4)$, see Subsection \ref{subSecSingHyp} for the details. By \cite{Robb}, $\p(\C^2 - B)  \simeq Braid_4/\langle [x_2,x_2^{x_1x_3}] \rangle$.

\end{enumerate}
\end{proof}

\section[Appendix A : New Zariski pairs (By Eugenii Shustin)]
{Appendix A : New Zariski pairs (By Eugenii Shustin)}

\subsection{Introduction}\label{intro}
Along Lemma \ref{lemBranchVar}, The family $B(d,c,n)$ of the plane
branch curves of degree $d$ with $c$ cusps and $n$ nodes as their
only singularities consists of entire components of $V(d,c,n)$, the
space parameterizing all irreducible plane curves of degree $d$ with
$c$ cusps and $n$ nodes as their only singularities. In the
particular case of the branch curves of generic projections of
smooth surfaces of degree $\nu\ge 3$ in $\PP^3$ onto the plane, one
has (cf. \cite{Sa} and Lemma \ref{lemNumNodeCuspSm})
\begin{equation}d(\nu)=\nu(\nu-1),\quad c(\nu)=\nu(\nu-1)(\nu-2),\quad
n(\nu)=\frac{1}{2}\nu(\nu-1)(\nu-2)(\nu-3)\
.\label{e1}\end{equation} The celebrated Zariski result \cite{Za1}
says that, in the case $\nu=3$, $d=6$, $c=6$, $n=0$, the variety
$V(6,6,0)$ contains a component which is disjoint with $B(6,6,0)$.
This suggests

\begin{conjecture}\label{c1}
For each $\nu\ge 3$, the variety $V(d(\nu),c(\nu),n(\nu))$ contains
a component disjoint with $B(d(\nu),c(\nu),n(\nu))$.
\end{conjecture}

Here we confirm this conjecture for few small values of $\nu$.

\begin{thm}\label{t1}
Conjecture \ref{c1} holds true for $3\le\nu\le10$.
\end{thm}

We prove Theorem \ref{t1} explicitly constructing curves \mbox{$C\in
V(d(\nu),c(\nu),n(\nu))\backslash B(d(\nu),c(\nu),n(\nu))$}. Our
construction is based on the patchworking method as developed in
\cite{Sh98}. It seems that this method does not allow one to cover
sufficiently large values of $\nu$. 

\subsection{Construction}\label{sec-con}

\subsubsection{The main idea} The variety $V(d,c,n)$ is said to be $T$-smooth at $C\in
V(d,c,n)$ if the germ of $V(d,c,n)$ at $C$ is the transverse
intersection in $| {\OO}_{\pp^2}(d)|$ of the germs at $C$ of the
smooth\footnote{In the case of nodes and cusps, the smoothness of
these equisingular strata always holds.} equisingular strata
corresponding to individual singular points of $C$ (cf.
\cite{Sh96}). In particular, this implies that, for any subset
$S\subset\Sing(C)$, there exists a deformation $C^S_t$,
$t\in(\C,0)$, such that $C^S_0=C$, $C^S_t\in V(d,c',n')$, where
$c'$, $n'$ are the numbers of cusps and nodes in $\Sing(C)\backslash
S$, respectively, and, furthermore, the deformation $C^S_t$ smoothes
out all the singular points of $C$ in $S$. Clearly, the $T$-smooth
part $V^T(d,c,n)$ of $V(d,c,n)$ is an open subvariety (if not empty)
of $V(d,c,n)$.

We derive Theorem \ref{t1} from

\begin{prs}\label{p1}
(1) Let $\nu\ge 3$. If\, $V^T(d(\nu),c(\nu),n(\nu)+1)\ne\emptyset$,
then $$V(d(\nu),c(\nu),n(\nu))\backslash
B(d(\nu),c(\nu),n(\nu))\ne\emptyset.$$

(2) If $3\le\nu\le10$, then
$V^T(d(\nu),c(\nu),n(\nu)+1)\ne\emptyset$.
\end{prs}

\subsubsection{Proof of Proposition \ref{p1}(1)}
If $\nu\ge5$ then as noticed in Subsection \ref{subsecDimB3},
$B_3(d(\nu),c(\nu),n(\nu))$ is not $T$-smooth, since a $T$-smooth
family must have the expected dimension. On the other hand,
smoothing out one node of a curve $C\in
V^T(d(\nu),c(\nu),n(\nu)+1)$, we obtain an element of
$V^T(d(\nu),c(\nu),n(\nu))$, a component whose dimension differs
from that of $B_3(d(\nu),c(\nu),n(\nu))$.

This reasoning does not cover the case of $\nu=3$ and $4$. We then
provide another argument which, in fact, works for all $\nu\ge3$.

Let $\nu\ge 3$, $C\in V^T(d(\nu),c(\nu),n(\nu)+1)$. We intend to
show that there is a nodal point $p\in C$ such that
$\Sing(C)\backslash\{p\}$ is not contained in a plane curve of
degree $d'(\nu)=(\nu-1)(\nu-2)$. This is enough, since by \cite{Se}
(see also Proposition \ref{lemNumNodeCuspSm}), all the singular points of a
branch curve $D\in B(d(\nu),c(\nu),n(\nu))$ lie on a plane curve of
degree $d'(\nu)$, and, on the other hand, a deformation $C^p_t\in
V(d(\nu),c(\nu),n(nu))$, $t\in(\C,0)$, of $C$ which smoothes out the
node $p$, contains curves whose singular points are not contained in
a curve of degree $d'(\nu)$.

We prove the existence of the required node $p\in C$ arguing for
contradiction. Let $p_1,p_2$ be some distinct nodes of $C$ and let
$C_1\supset\Sing(C)\backslash\{p_1\}$,
$C_2\supset\Sing(C)\backslash\{p_2\}$ be some curve of degree
$d'(\nu)$. Since $d(\nu)\cdot d'(\nu)=2(c(\nu)+n(\nu))$, the curves
$C_1$ and $C_2$ are non-singular along their intersection with $C$,
in particular, they are reduced. Furthermore, $p_1\not\in C_1$ and
$p_2\not\in C_2$. Let $D$ be the (possibly empty) union of the
common components of $C_1$ and $C_2$ with $\deg D=k$, $0\le
k<d(\nu)$. So, $C_1=DD_1$, $C_2=DD_2$, where the curves $D_1,D_2$ of
degree $d(\nu)-k$ have no component in common. By Bezout's theorem
$D\cap C$ consists of $kd(\nu)/2$ points of $\Sing(C)$, and $D_i\cap
C=\Sing(C)\backslash(D\cap C\cup\{p_i\})$, $i=1,2$. Take two
distinct generic straight lines $L_1,L_2$ through $p_1$. By
Noether's $AF+BG$ theorem (see, for instance,
\cite{VDW})\footnote{For the sake of notation we denote a plane
curve and its defining homogeneous polynomial (given up to a
constant factor) by the same symbol, no confusion will arise.},
\begin{itemize}\item if $k\le d'(\nu)+1-d(\nu)/2$, then there are polynomials
$A_1,A_2$ of degree $d'(\nu)+2-k$ and polynomials $B_1,B_2$ of
degree $2d'(\nu)+2-d(\nu)-2k$ such that $D_2^2L_i^2=A_iD_1+B_iC$,
$i=1,2$,
\item if $k\ge d'(\nu)+2-d(\nu)/2$, then there are polynomials $A_1,A_2$ of degree
$d'(\nu)+2-k$ such that $D_2^2L_i^2=A_iD_1$, $i=1,2$.
\end{itemize} The latter case is impossible since $D_2^2L_i^2$ and
$D_1$ have no component in common. In the former case, we obtain
that $D_1$ divides $D_2^2(L_1^2B_2-L_2^2B_1)$, and hence divides
$L_1^2B_2-L_2^2B_1$. In view of $$\deg
D_1=(\nu-1)(\nu-2)-k>2+(\nu-1)(\nu-4)-2k=\deg(L_1^2B_2-L_2^2B_1)\
,$$ we conclude that $L_1^2B_2=L_2^2B_1$, in particular, $L_1^2$
divides $B_1$, but then $L_1^2$ divides $A_1$ too, and hence
$D_2^2=A'_1D_1+B'_1C$ contrary to the fact that $p_2\not\in D_2$ and
$p_2\in D_1\cap C$.

\subsubsection{Proof of Proposition \ref{p1}(2)}
We suppose that $\nu\ge4$. Using the patchworking construction of \cite[Theorem
3.1]{Sh98}, we obtain curves in $V^T(12,24,16)$, $V^T(20,63,67)$,
$V^T(30,126,191)$, $V^T(42,216,435)$, $V^T(56,336,902)$,
$V^T(72,504,1550)$, and $V^T(90,720,2526)$, what suffices for our
purposes in view of Proposition \ref{p1}(1), since
$$d(4)=12,\quad c(4)=24,\quad n(4)=12<16\ ,$$
$$d(5)=20,\quad c(5)=60<63,\quad n(5)=60<67\ ,$$ $$d(6)=30,\quad
c(6)=120<126,\quad n(6)=180<191\ ,$$ $$d(7)=42,\quad
c(7)=210<216,\quad n(7)=420<435\ ,$$ $$d(8)=56,\quad c(8)=336,\quad
n(8)=840<902\ ,$$ $$d(9)=72,\quad c(9)=504,\quad d(9)=1512<1550\ ,$$
$$d(10)=90,\quad c(10)=720,\quad n(10)=2520<2526\ .$$

Referring to \cite{Sh98} for details, we only recall that the
patchworking construction uses a convex\footnote{Convexity means
that the subdivision lifts up to a graph of a convex function linear
on each subdivision polygon and having a break along each common
edge.} lattice subdivision of the triangle
$T_d=\conv\{(0,0),(0,d),(d,0)\}$. Pieces $\Del_1,...,\Del_N$ of the
subdivision will serve as Newton polygons of polynomials in two
variables (called {\it block polynomials}) which define curves with
nodes and cusps leaving in the respective toric surfaces:
$C_k\subset\Tor(\Del_k)$, $k=1,...,N$. Along \cite[Theorem 4.1]{Sh98},
the patchworking construction can be performed under the following
sufficient conditions:
\begin{enumerate}\item[(C1)] Any two block polynomials have the same
coefficients along the common part of their Newton polygons, and the
truncations of block polynomials to any edge is a nondegenerate
(quasihomogeneous) polynomial.
\item[(C2)] The adjacency graph of the pieces of the subdivision can be oriented
without oriented cycles so that if, for each polygon $\Del_k$, $\le
k\le N$, we mark its sides which correspond to the arcs of the
adjacency graph coming inside the polygon and denote by
$D_k\subset\Tor(\Del_k)$ the union of the unmarked toric divisors,
then the number of cusps of any component $C$ of $C_k$ is less than
$CD_k$. \end{enumerate}

By \cite[Theorems 3.1 and 4.1]{Sh98}, the resulting curve with the
Newton triangle $T_d$ belongs to $V^T(d,c,n)$, and the numbers $c$
and $n$ are obtained by summing up the numbers of cusps and nodes
over all the block curves.

Condition (C2) formulated above is, in fact, sufficient for the
following transversality property defined in \cite[Definition
2.2]{Sh98} and used in the patchworking construction of \cite[Theorem
3.1]{Sh98}: the variety $V_C$ consisting of the curves in the linear
system $|C|$ on $\Tor(\Del_k)$ which are equisingular to $C$ and
intersect $D'_k$ at the same points as $C$, where $D'_k$ is the
union of the marked toric divisors of $\Tor(\Del_k)$, is smooth at
$C$ of codimension $2c(C)+n(C)+CD'_k$, where $c(C),n(C)$ are the
numbers of cusps and nodes of $C$. We shall call this property {\it
the $T$-smoothness relative to $D'_k$}. Let $L$ be a coordinate line
in $\pp^2$ corresponding to a side $\sigma$ of $T_d$. Then the
patchworking construction produces a curve, where the variety
$V(d,c,n)$ is $T$-smooth relative to $L$, if one replaces condition
(C2) by the following one:
\begin{enumerate}\item[(C2')] The adjacency graph of the pieces of the subdivision can be oriented
without oriented cycles so that if, for each polygon $\Del_k$, $\le
k\le N$, we mark its sides which correspond to the arcs of the
adjacency graph coming inside the polygon or are contained in
$\sigma$, and denote by $D_k\subset\Tor(\Del_k)$ the union of the
unmarked toric divisors, then the number of cusps of any component
$C$ of $C_k$ is less than $CD_k$.\end{enumerate}

\medskip

We consider the subdivisions of $T_{12}$, $T_{20}$, $T_{30}$,
$T_{42}$, $T_{56}$, $T_{72}$, and $T_{90}$ shown in Figures
\ref{f1}, \ref{f2}(b), \ref{f3}, where all the slopes are $0$, $-1$,
or $\infty$. We leave to the reader an easy exercise to check that
these subdivisions are convex. Next we describe the block
polynomials.

\medskip

(1) If $d=12$, we take the block polynomial $F_1(x,y)$ with Newton
triangle $T_6$ defining a curve $C_1\in V(6,6,4)$ (such a curve is
dual to an irreducible quartic with three nodes). The other block
polynomials are obtained via affine automorphisms of $\Z^2$ which
interchange two adjacent triangles of the subdivision keeping their
common side fixed. Thus, the patchworking construction gives a curve
$C_{12}\in V^T(12,24,16)$.

\medskip

(2) If $d=20$, we take the block polynomial $F_2(x,y)$ with Newton
triangle $T_6$ defining a curve $C_2\in V(6,9,0)$ (such a curve is
dual to a non-singular cubic). The other block polynomials with the
Newton triangles with side length $6$ are obtained from $F_2$ by
suitable reflections. For any other polygon in the given
subdivision, we take a block polynomial splitting into the product
of linear polynomials, defining (reducible) nodal curves, and
satisfying (C1). It is easy to check that condition (C2) holds, and
thus, one obtains a curve $C_{20}\in V^T(20,63,67)$.

\medskip

(3) If $d=30$ or $42$, for each triangle in the subdivision having
side length $6$ and intersecting with a coordinate axis, we take the
block polynomial obtained from $F_2$ as described above, and, for
any other polygon of the subdivision, we take a suitable polynomial
splitting into the product of linear polynomials. This gives us the
curves $C_{30}\in V^T(30,126,191)$ and $C_{42}\in V^T(42,216,435)$.

\begin{figure}
\setlength{\unitlength}{1cm}
\begin{picture}(14,16)(0,0)
\thinlines \put(0,1){\vector(0,1){5.5}}\put(0,1){\vector(1,0){5.5}}
\put(6.5,1){\vector(0,1){7.5}}\put(6.5,1){\vector(1,0){7.5}}
\put(0,10.5){\vector(0,1){5}}\put(0,10.5){\vector(1,0){5}}
\put(6.5,10.5){\vector(0,1){4.5}}\put(6.5,10.5){\vector(1,0){4.5}}
\put(1,1){\line(0,1){4}}\put(2,1){\line(0,1){3}}\put(3,1){\line(0,1){2}}\put(4,1){\line(0,1){1}}
\put(0,2){\line(1,0){4}}\put(0,3){\line(1,0){3}}\put(0,4){\line(1,0){2}}\put(0,5){\line(1,0){1}}
\put(1,1){\line(-1,1){1}}\put(2,1){\line(-1,1){2}}\put(3,1){\line(-1,1){1}}\put(4,1){\line(-1,1){1}}
\put(1,3){\line(-1,1){1}}\put(1,4){\line(-1,1){1}}\put(5,1){\line(-1,1){5}}\put(2,0){\rm
(c)}
\put(0.9,0.6){$6$}\put(1.8,0.6){$12$}\put(2.8,0.6){$18$}\put(3.8,0.6){$24$}\put(4.8,0.6){$30$}
\put(7.5,1){\line(0,1){6}}\put(8.5,1){\line(0,1){5}}\put(9.5,1){\line(0,1){4}}
\put(10.5,1){\line(0,1){3}}\put(11.5,1){\line(0,1){2}}\put(12.5,1){\line(0,1){1}}
\put(6.5,2){\line(1,0){6}}\put(6.5,3){\line(1,0){5}}\put(6.5,4){\line(1,0){4}}
\put(6.5,5){\line(1,0){3}}\put(6.5,6){\line(1,0){2}}\put(6.5,7){\line(1,0){1}}
\put(7.5,1){\line(-1,1){1}}\put(8.5,1){\line(-1,1){2}}\put(9.5,1){\line(-1,1){1}}
\put(10.5,1){\line(-1,1){1}}\put(11.5,1){\line(-1,1){1}}\put(12.5,1){\line(-1,1){1}}
\put(7.5,3){\line(-1,1){1}}\put(7.5,4){\line(-1,1){1}}\put(7.5,5){\line(-1,1){1}}
\put(7.5,6){\line(-1,1){1}}\put(13.5,1){\line(-1,1){7}}
\put(10,0){\rm (d)}
\put(7.4,0.6){$6$}\put(8.3,0.6){$12$}\put(9.3,0.6){$18$}\put(10.3,0.6){$24$}\put(11.3,0.6){$30$}
\put(12.3,0.6){$36$}\put(13.3,0.6){$42$}
\put(2,10.5){\line(0,1){2}}\put(0,12.5){\line(1,0){2}}\put(2,10.5){\line(-1,1){2}}
\put(4,10.5){\line(-1,1){4}}\put(2,9.5){\rm
(a)}\put(1.9,10.1){$6$}\put(3.8,10.1){$12$}
\put(7.5,10.5){\line(0,1){2.5}}\put(8.5,10.5){\line(0,1){1.5}}\put(6.5,11.5){\line(1,0){2.5}}
\put(6.5,12.5){\line(1,0){1.5}}\put(7.5,10.5){\line(-1,1){1}}\put(8.5,10.5){\line(-1,1){2}}
\put(9.5,10.5){\line(-1,1){1}}\put(10,10.5){\line(-1,1){3.5}}\put(8.5,9.5){\rm
(b)}\put(7.4,10.1){$6$}\put(8.3,10.1){$12$}\put(9.3,10.1){$18$}\put(9.9,10.1){$20$}
\end{picture}
\caption{Patchworking construction: The case
$\nu=4,5,6,7$}\label{f1}
\end{figure}
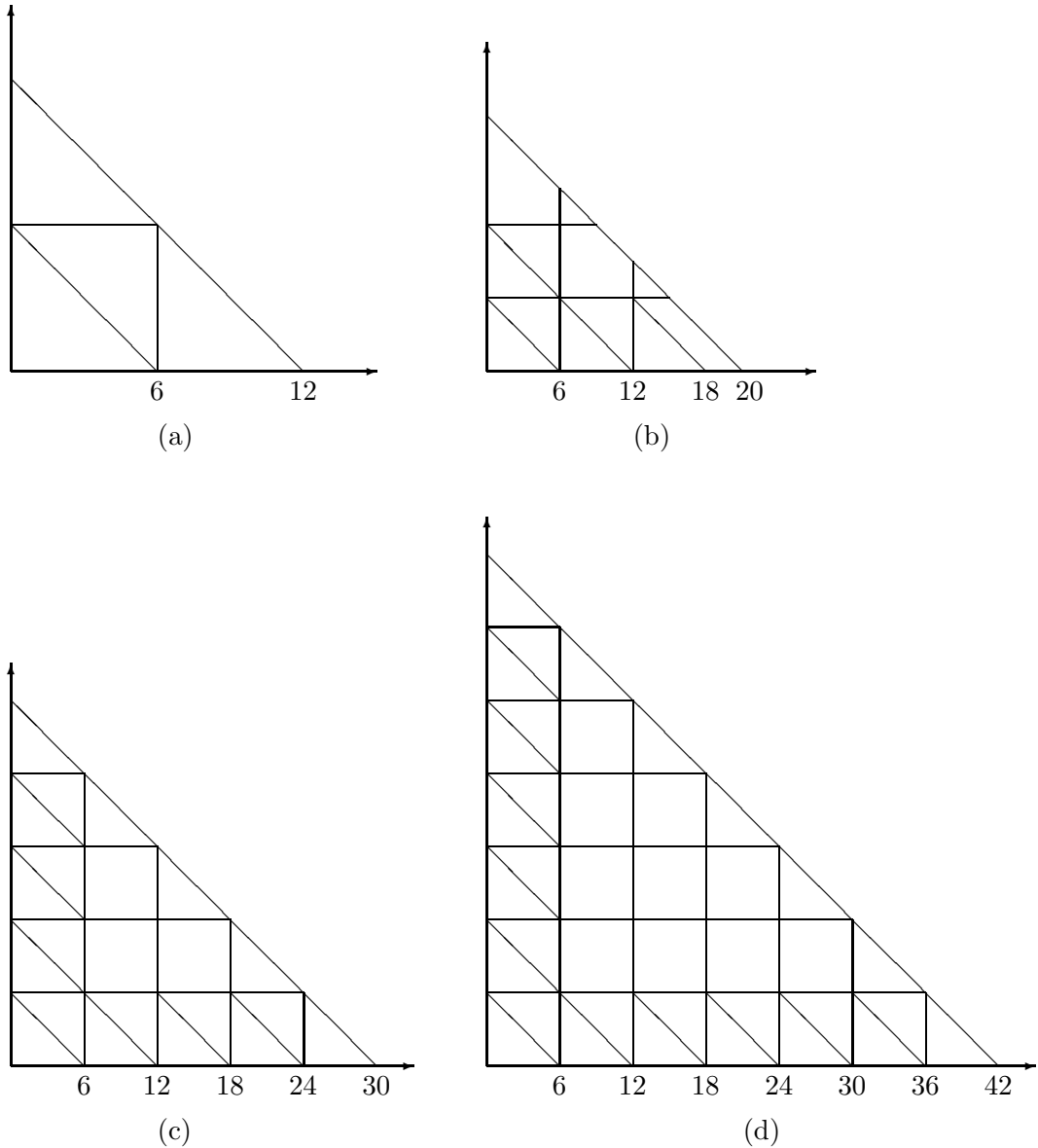

\medskip

(4) If $d=56$ we use the block polynomial $F_3$ defining a curve
$C_9\in V(9,16,10)$ which is obtained via a slight modification of
the construction of a curve $C'_9\in V(9,20,0)$ in \cite[Section
4.3]{Sh98}. We consider the subdivision of $T_9$ shown in Figure
\ref{f2}(a) (cf. \cite[Figure 2]{Sh98}) and take the following block
curves: those with two symmetric Newton quadrangles have $8$ cusps
each (as in \cite[Proof of Theorem 4.3]{Sh98}), the block curve with
Newton triangle has one node, and the block curve with Newton square
splitting onto $6$ lines, has $9$ nodes.

Now we subdivide the triangle $T_{56}$ as shown in Figure
\ref{f2}(b) and take the following block curves:\begin{itemize}\item
The block curves with the triangles intersecting with the coordinate
axes and the triangle marked with asterisk are defined by the
polynomial $F_3$ and its appropriate transforms, \item each other
block curve is defined by a polynomial splitting into linear
factors.
\end{itemize} Observe that the conditions (C1) and (C2) can be
satisfied in this situation, which, finally, gives us a curve in
$V^T(56,336,902)$.

\begin{figure}
\setlength{\unitlength}{1cm}
\begin{picture}(12,14.5)(0,0)
\thinlines \put(2,1){\vector(0,1){7}}\put(2,1){\vector(1,0){7}}
\put(2,9.5){\vector(0,1){5}}\put(2,9.5){\vector(1,0){5}}
\put(3,1){\line(0,1){5.5}}\put(4,1){\line(0,1){4}}\put(5,1){\line(0,1){3}}\put(6,1){\line(0,1){2}}
\put(7,1){\line(0,1){1}}\put(3.5,9.5){\line(0,1){1.5}}
\put(2,2){\line(1,0){5.5}}\put(2,3){\line(1,0){4.5}}\put(2,4){\line(1,0){3.5}}\put(2,5){\line(1,0){2.5}}
\put(2,6){\line(1,0){1.5}}\put(2,11){\line(1,0){1.5}}
\put(2,2){\line(1,-1){1}}\put(2,3){\line(1,-1){2}}\put(2,4){\line(1,-1){3}}\put(2,5){\line(1,-1){1}}
\put(2,6){\line(1,-1){1}}\put(2,7){\line(1,-1){6}}\put(2,7.5){\line(1,-1){6.5}}
\put(5,2){\line(1,-1){1}}\put(6,2){\line(1,-1){1}}\put(3.5,11){\line(2,1){1}}
\put(3.5,11){\line(1,2){0.5}}\put(2,14){\line(1,-1){4.5}}
\dashline{0.2}(4,9.5)(4,12)\dashline{0.2}(4.5,9.5)(4.5,11.5)
\put(10,4){\rm (b)}\put(10,12){\rm (a)}
\put(3.4,9.1){$3$}\put(1.7,10.9){$3$}\put(3.9,9.1){$4$}\put(4.4,9.1){$5$}\put(6.1,9.1){$9$}
\put(2.9,0.6){$9$}\put(3.8,0.6){$18$}\put(4.8,0.6){$27$}\put(5.8,0.6){$36$}
\put(6.8,0.6){$45$}\put(7.8,0.6){$54$}\put(8.4,0.6){$56$}
\put(3.6,2.6){$*$}
\end{picture}
\caption{Patchworking construction: The case $\nu=8$}\label{f2}
\end{figure}

\medskip

(5) Observe that the construction of step (3) gives a curve $C_{30}$
at which the variety $V(30,126,191)$ is $T$-smooth relative to the
$y$-axis. Clearly, this $T$-smoothness property at $C_{30}$ hold
relatively to almost all lines in $\pp^2$, and hence by an
appropriate coordinate change we can make $V(30,126,191)$ to be
$T$-smooth at $C_{30}$ relative to each of the coordinate lines. Let
$F_4(x,y)$ be a defining polynomial of $C_{30}$.

Consider the subdivision of $T_{72}$ presented in Figure
\ref{f3}(a). For the triangle incident to the origin, we take the
above polynomial $F_4(x,y)$, and, for the other triangles with side
length $30$, we take the transforms of $F_4$ as described in step
(1). For the other polygons of the subdivision we take appropriate
polynomials splitting into linear factors. The relative
$T$-smoothness of $V(30,126,191)$ at $C_{30}$ ensures condition
(C2); hence the patchworking procedure is performable, and it gives
a curve $C_{72}\in V^T(72,504,902)$.

\medskip

(6) In the case $d=90$ we need a curve $C'_{30}\in V(30,120,199)$ at
which the variety $V(30,120,199)$ is $T$-smooth relatively to each
of the coordinate lines. Assuming that such a curve exists, we take
its defining polynomial $F_5(x,y)$ and spread it through all the
triangles in the subdivision shown in Figure \ref{f3}(b) except for
the right-most one (marked by asterisk). For the latter triangle and
for the parallelogram we take suitable polynomials splitting into
linear factors. Again, due to the relative $T$-smoothness of
$V(30,120,199)$ at $C'_{30}$, the condition (C2) holds true, and
hence the patchworking procedure gives a curve in $V(90,720,2526)$.

The required curve $C'_{30}$ can be constructed using the modified
construction of step (3) and a generic coordinate change afterwards.
The modification is as follows: for the upper and the right-most
triangles with side length $6$ in the subdivision shown in Figure
\ref{f1}(c), we take polynomials defining curves in $V(6,6,4)$
(instead of $V(6,9,0)$ as in the original construction of step (3)).
Curves of degree $6$ with $6$ cusps and $4$ nodes do exist: they are
dual to rational quartics with $3$ nodes. To ensure condition (C1),
we need the sextics as above which cross one of the coordinate lines
along a prescribed configuration of $6$ points. Notice that, given a
straight line $L\subset\pp^2$, the varieties $V(6,6,4)$ and
$V(6,9,0$ are $T$-smooth elative to $L$ at each curve crossing $L$
transversally (it immediately follows from condition (C2')). In
particular, this yields that the rational maps
$V(6,6,4)\to\Sym^6(L)$ and $V(6,9,0)\to\Sym^6(L)$ defined by
$C\mapsto C\cap L$ are dominant. Hence we can choose polynomials
defining curves a curve in $V(6,6,4)$ and a curve in $V(6,9,0)$ so
that the considered patchworking data will meet condition (C1).

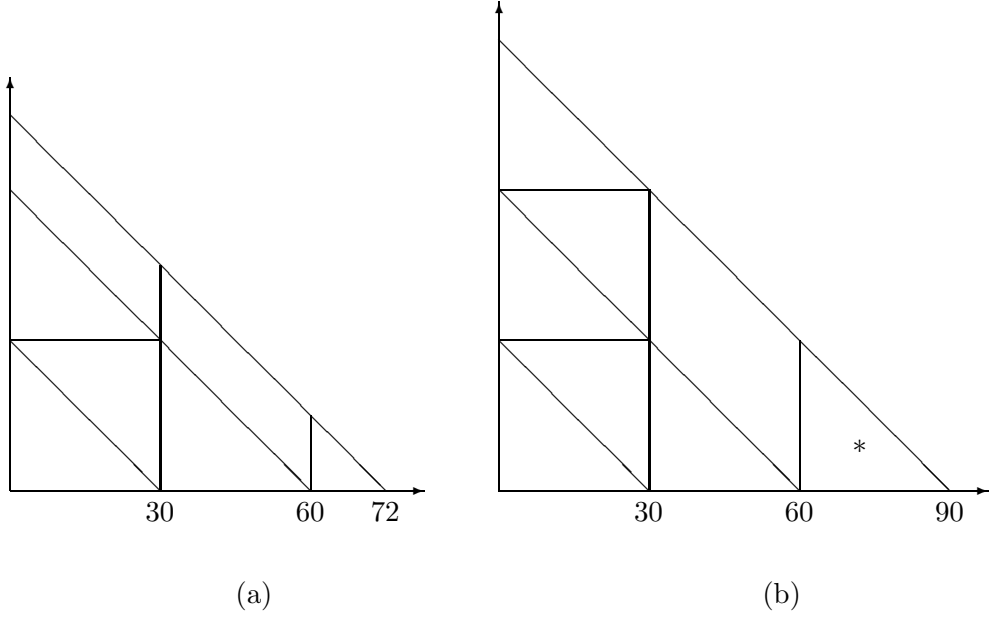
\begin{figure}
\setlength{\unitlength}{1cm}
\begin{picture}(13,8)(0,0)
\thinlines
\put(0,1.5){\vector(0,1){5.5}}\put(0,1.5){\vector(1,0){5.5}}
\put(6.5,1.5){\vector(0,1){6.5}}\put(6.5,1.5){\vector(1,0){6.5}}
\put(2,1.5){\line(0,1){3}}\put(4,1.5){\line(0,1){1}}\put(0,3.5){\line(1,0){2}}\put(0,3.5){\line(1,-1){2}}
\put(0,5.5){\line(1,-1){4}}\put(0,6.5){\line(1,-1){5}}
\put(8.5,1.5){\line(0,1){4}}\put(10.5,1.5){\line(0,1){2}}\put(6.5,3.5){\line(1,0){2}}\put(6.5,5.5){\line(1,0){2}}
\put(6.5,3.5){\line(1,-1){2}}\put(6.5,5.5){\line(1,-1){4}}
\put(6.5,7.5){\line(1,-1){6}} \put(3,0){\rm (a)}\put(10,0){\rm (b)}
\put(1.8,1.1){$30$}\put(3.8,1.1){$60$}\put(4.8,1.1){$72$}\put(8.3,1.1){$30$}\put(10.3,1.1){$60$}
\put(12.3,1.1){$90$} \put(11.2,2){$*$}
\end{picture}
\caption{Patchworking construction: The case $\nu=9$ and
$10$}\label{f3}
\end{figure}

\section[Appendix B : Picard and Chow groups for nodal-cuspidal curves]
{Appendix B : Picard and Chow groups for nodal-cuspidal curves}

In this Appendix we remind the reader the connections between
Cartier and Weil divisors and the connection of the Picard and
Chow groups on a nodal--cuspidal curve with $c$ cusps and $n$ nodes. This connection is
implicit in  Segre \cite{Se}, and here we recall the explicit
formulation.


%
%
%

Let $B$ a nodal--cuspidal plane curve, with  $B^*$ its
normalization in $\pp^3$. First, by definitions of $\Pic$ and
$A_0$, we have for $B$
$$
  \xymatrix
  {
       {} & {} & {0}\ar[d] & {0}\ar[d] & {} \\
       {} & {0}\ar[r]\ar[d] & {Cart.\, Princ.\, B}\ar[d]\ar[r] & {Weil.\,Princ.\, B}\ar[r]\ar[d]  & {0} \\
       {0}\ar[r] & {G_S}\ar[d]\ar[r] & {Cartier\, B}\ar[d]\ar[r] & {Weil\, B}\ar[d]\ar[r] &  {0} \\
       {0}\ar[r] & {G_S}\ar[d]\ar[r] & {\Pic B}\ar[d]\ar[r] & {A_0 B}\ar[d]\ar[r] & {0} \\
       {} & {0} & {0} & {0} & {}
  }
$$

where the canonical map $\Pic B \to A_0 B$ is the map induced by
associating the class of a Weil divisor with each Cartier divisor
on $B$, $S$ is the set of singular points of $B$, and $G_S$ is the
subgroup of the group of Cartier divisors on $B$ such that their
associated Weil divisors are trivial.
Note that the map $Cartier\,
B \to Weil\, B$ is surjective since $B$ is a nodal-cuspidal curve.

Secondly, there is an exact sequence
$$
    0 \to H^0 Q_S  \to \Pic B
    \stackrel{\pi^*}{\to} \Pic B^* \to 0
$$

where $Q_S = \pi_*(\OO^*_{B^*}) /\OO^*_B = \prod\limits_{p \in
\Sing B} (\prod\limits_{p^* \stackrel{\pi}{\to} p} O^*_{p*} ) /
O^*_p$. 

We also have the following excision diagram:

$$
  \xymatrix
  {
         0 \ar[d]             &        0  \ar[d]         &                                 &      \\
     { \ZZ^P } \ar[d]\ar[r]   & {\ZZ^P / \backsim } \ar[d]\ar[r] & 0  \ar[d]               &      \\
     {A_0 \xi^*} \ar[r]\ar[d] & A_0 B^* \ar[r]\ar[d]     & A_0 U^* \ar[r]\ar[d]^-{\simeq}  &   0  \\
     {A_0 \xi}   \ar[r]\ar[d] & A_0 B   \ar[r]\ar[d]     & A_0 U   \ar[r]\ar[d]            &   0  \\
        0                     &  0                       &   0                             &
  }
$$
where $\xi = P + Q$, $\xi^* = P^* + Q^*$, $U = B - \xi$, $U^* =
B^* - \xi^*$, and the map $A_0 \xi^* \to A_0 \xi$ can be described
as $\ZZ^{P^*} \oplus \ZZ^{Q^*} \to \ZZ^{P} \oplus \ZZ^{Q}$ which
is the factorization of $\ZZ^{P^*}$ by a subgroup generated by
$(p_1^* - p_2^*)$ for a preimage of each node $p$ of $B$. (For a
different proof that the map $A_0B^* \to A_0B $ is epimorphic, see
\cite[Example 1.9.5]{Ful}). Denote $T = \ZZ^P / \backsim$.

 Combining the two diagrams together, we get
$$
  \xymatrix
  {
       {}       & {0}\ar[d]          & {0}\ar[d]                     & {T}\ar[d]     &  \\
     0 \ar[r] & H^0 Q_S \ar[r]\ar[d] &  {\Pic B}\ar[r]^{\pi^*}\ar[d]^{\cong} & {\Pic  B^*}\ar[r]\ar[d] &  0 \\
     0 \ar[r] & G_S \ar[r]\ar[d] &  {\Pic B}\ar[r]\ar[d] & {A_0 B}\ar[r]\ar[d] &  0 \\
     {}       & {T'}   & {0}                    & 0            &  \\
  }
$$

Where the last column is induced from the fact that $\Pic B^*
\simeq A_0 B^*$ and from the exact sequence
\begin{equation} \label{eqnExSeqA0}
0 \to T  \to A_0B^* \to A_0B \to 0.
\end{equation}
Note that the map $H^0 Q_S \to G_S$ is injective, and cannot be
surjective, otherwise the map $\Pic  B^* \to A_0 B$ would be an
isomorphism. Thus  $T' \cong T$.
%



\section[Appendix C : Bisecants to a complete intersection curve in $\pp^3$]
{Appendix C : Bisecants to a complete intersection curve in
$\pp^3$} \label{subsecBisecantComInter} We note here that the
inverse statement to Theorem \ref{thmHalphen} is easy. Explicitly,
we have
the following Theorem:\\\\
\emph{Let $C$ be a curve in $\PP = \PP^3$ which is a complete
intersection of type $(\mu, \nu)$, and let $a$ be a point in $\PP$
which is not on $C$ such that $C$ does not admit any 3-secants
through $a$.  Then $C$ has $\frac{1}{2} \mu \nu (\mu-1)(\nu-1)$
bisecants passing through $a$.\\}

\begin{remark} \label{remNumNodes}
\emph{The above theorem gives a direct proof that the number of
nodes of the branch curve $B$ is indeed $n = \frac{1}{2} \n (\n-1)
(\n-2)(\n-3)$ (recall that $B^*$ is a complete intersection of $S$
and $Pol_OS$, i.e., of type $(\n,\n-1)$ and that the line
$\overline{Oq^*}$ for each $q^* \in Q^*$ is also considered as a
bisecant of $B^*$). }\end{remark}

\begin{proof} Our proof is essentially a reformulation of a proof by
Salmon, \cite[art. 343]{Sa}. See also \cite[Chapter IX, sections
1.1,1.2]{SR} for another way to induce this formula. Consider the
moduli space $M$ of data \{line $l$ in $\PP$ which is bisecant for
$C$, a point $p' \in l \intersect C$, a point $p \in l$, $p \notin
C$\}. (see the following Figure)
\begin{center}
\epsfig{file=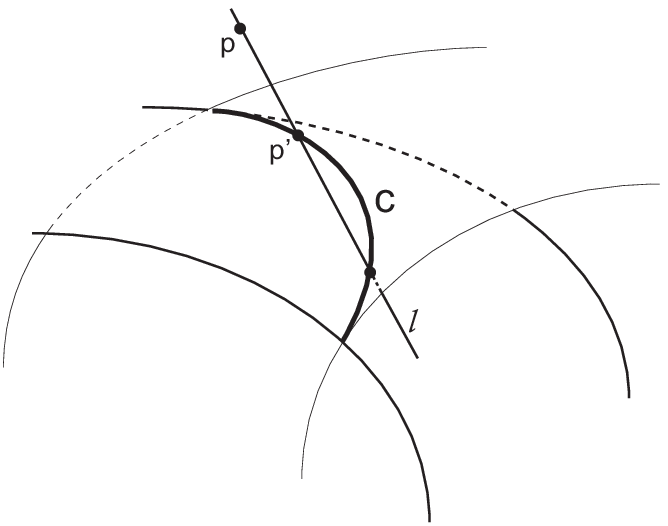} \\
\small{ the parameters $(l,p',p)$}
\end{center}
   It is clear that the line $l$ can be reconstructed uniquely from
$p'$ and $p$ as $l_{p,p'}$, and thus $M$ can be embedded into $\PP
\times \PP$, $(l,p',p) \mapsto (p',p)$.
For a point $(l,p',p)$ in $M$, let $q$ be a point in $l \intersect
C$ different from $p'$. Then there is a number $t \in k$ such that
$q = p' + t p$. If $C$ is given by 2 equations $u$, $v$, then we
have $u(q) = 0$, $v(q) = 0$.  Let us write $u(q) = u(p' + t p) =
u_0(p') + t u_1(p',p) + \dots + t^{\mu} u_{\mu} (p',p)$, where
$u_i$ is of degree $\mu-i$ in $p'$ and $i$ in $p$. In the same way
we can write $v(q) = v(p' + t p) = v(p') + t v_1(p',p) + \dots +
t^{\nu} v_{\nu} (p',p).$
Consider now two polynomials,
\begin{gather*}
   a(t) = u_1(p',p) + \dots + t^{\mu}- u_{\mu} (p',p), \\
   b(t) = v_1(p',p) + \dots + t^{\nu}- u_{\nu} (p',p).
\end{gather*}
Let $R(p',p)$ be the resultant of $a(p',p,t)$ and $b(p',p,t)$ in
$t$. It has (see the the Sylvester definition of resultant)
bidegree $\left( (\mu-1)(\nu-1), \mu \nu - 1 \right)$ in $(p',p)$.
Lemma:  Let $U \subset \PP \times \PP = \{(p', p) : p' \neq p \}.$
Then
$$
      U \intersect (R = 0) \intersect  (C \times \PP) = M
$$
Indeed, let $(p',p)$ be such that $R(p',p) = 0$, $p' \neq p$. Then
there is a number $t \in k$ such that $a(p',p,t)=0, b(p',p,t) =
0$. Let $q = p' + t p$.
We have
\begin{gather*}
u(q) = u(p') + t a(p',p) = u(p') \\
v(q) = v(p') + t b(p',p) = v(p').
\end{gather*}
Thus $q \in C$ iff $p' \in C$. It follows that $(p',p)$ is in $M$
iff $p' \in C$. This proves the lemma.
It follows now that
$$ R \intersect (C \times {a}) = M \intersect (\PP \times {a}) = ( %
    \text{ bisecants through $a$ to $C$ with a marked point $p'$ in $l
    \intersect C$ }) %
$$
The order of this set is equal to $\deg_{p'} C \cdot deg C =
(\mu-1)(\nu-1) \mu \nu$. Since C does not have any 3-secants
through $a$, it follows that the number of bisecants through $a$
is one half of the number above. \end{proof}

\end{document}